\documentclass[reqno]{amsart}

\usepackage{packages_and_commands}
\title{Global bifurcation for steady viscous roll waves on an incline}

\author{Daniel Abraham}
\address{Department of Mathematical Sciences, University of Bath, Bath BA2 7AY, United Kingdom}
\email{da979@bath.ac.uk}

\author{Miles H.~Wheeler}
\address{Department of Mathematical Sciences, University of Bath, Bath BA2 7AY, United Kingdom}
\email{mw2319@bath.ac.uk}

\begin{document}

\begin{abstract}
    We construct a branch of travelling periodic `roll wave' solutions to the free-boundary incompressible Navier--Stokes equations on an inclined plane in two dimensions. These solutions bifurcate from a parallel shear flow, under natural assumptions on the related Orr--Sommerfeld equation. Using techniques from analytic global bifurcation theory, we extend the local branch to a global curve of solutions. A key step of the proof is reformulating the problem, including the unknown free boundary, as an elliptic system in the sense of Agmon--Douglis--Nirenberg. Finally, we verify the hypotheses on the Orr--Sommerfeld equation for two regimes: small wavenumber and low Reynolds number.
\end{abstract}

\maketitle

\tableofcontents

\section{Introduction}
The progression of surface waves down an incline is an intriguing physical phenomenon, with nontrivial profiles preserved by a delicate balancing between the dissipative force due to viscosity and the gravitational force acting down the slope. Periodic disturbances of this form can be called \emph{roll waves}. While roll waves are often turbulent~\cite{jeffreys1925, dressler1949mathematical, balmforth2004dynamics}, there are everyday occurrences with laminar profiles, such as heavy rainfall on window panes or flow down street gutters. This travelling laminar viscous flow is modelled by the free-boundary Navier--Stokes equations, which are governed at a linear level by the `Orr--Sommerfeld' equation~\cite{orr1907stability1, orr1907stability2, sommerfeld1909beitrag}, with long-wave stability results for inclined flow by Benjamin~\cite{benjamin1957wave} and Yih~\cite{yih1963stability}. 

Until recently, due to the challenges in nonlinear analysis of the Navier--Stokes equations, research in steady free-surface flow has largely advanced by using model equations with shallow-water approximations or by studying the inviscid model, known as the travelling `water wave' problem. The literature for steady water waves is very well established, with global constructions of solutions under a variety of scenarios; see for instance the survey~\cite{haziot2022traveling}. Rigorous work on the full steady free-boundary Navier--Stokes problem, on the other hand, has only appeared in the last few years~\cite{leoni2023traveling, stevenson2021traveling, stevenson2023well, stevenson2024well, koganemaru2023traveling, koganemaru2024traveling, stevenson2025gravity, banihashemi2026large}, beginning with the construction by Leoni and Tice~\cite{leoni2023traveling} of small-amplitude waves under generic external bulk and stress forcing. This literature however leaves open the global existence of solutions, and only~\cite{stevenson2025gravity} constructs solutions with gravity as the sole external force, finding travelling bore wave solutions in a shallow water regime.

In this paper, we give the first `large-amplitude' construction for the full steady free-boundary Navier--Stokes problem, finding a global curve of solutions bifurcating from shear flow. These solutions are two-dimensional, laminar, periodic roll waves which are subject only to gravity as an external force and may have overturning profiles. Our hypotheses depend solely on the associated Orr--Sommerfeld system. Furthermore, we verify these hypotheses in parameter regimes where the wavenumber or the Reynolds number are small, consistent with the earlier formal work by Benjamin and Yih.

Our proof involves an analytic global bifurcation argument. Reformulating the problem as a system for the stream function, pressure and a conformal map which flattens the domain, we show this system is elliptic in the sense of Agmon--Douglis--Nirenberg~\cite{agmon1964estimates}. Because the problem does not possess a reflection symmetry as is common in water wave scenarios, we use a local bifurcation argument due to Hale~\cite{hale1978bifurcation} rather than the more standard result by Crandall--Rabinowitz~\cite{crandall1971bifurcation}. The global continuation of this local curve of solutions uses theory developed by Dancer, Buffoni and Toland~\cite{dancer1973bifurcation, buffoni2003analytic}.

\subsection{The inclined travelling viscous surface wave problem}
Let us define our problem more precisely. We consider a viscous incompressible fluid in a two-dimensional inclined domain $\Omega$, with a flat lower bed $\Bed$ at an angle $0 < \theta \leq \pi/2$ to the horizontal and a free upper surface $\Surf$. Here $\Surf$ is not required to be the graph of a function, allowing for overhanging waves. In fact, we will see $\Omega$ as the image of a rectangular strip $\Real \times (0,1)$ under a conformal map 
\begin{equation}\label{eq: f_intro}
    f \colon \Real \times (0,1) \to \Omega,
\end{equation}
satisfying the Cauchy--Riemann equations and mapping the upper and lower boundaries of the strip to $\Surf$ and $\Bed$ respectively. This conformal map is detailed in Section~\ref{section: prelims} below.

Our interest is in travelling wave solutions, i.e.~solutions whose velocity field is stationary with respect to time in a reference frame moving at a constant speed $\gamma > 0$ down the incline. We will work in Cartesian coordinates $(X,Y)$ for physical space, with $X$ parallel to the inclined plane and $Y$ perpendicular to it. The origin of the $Y$-axis is set so that the bed $\Bed$ is given by the line $Y=0$. See Figure~\ref{fig:fluid_domain} for an illustration.
\begin{figure}[ht]\label{fig:fluid_domain}
    \centering
    \includegraphics{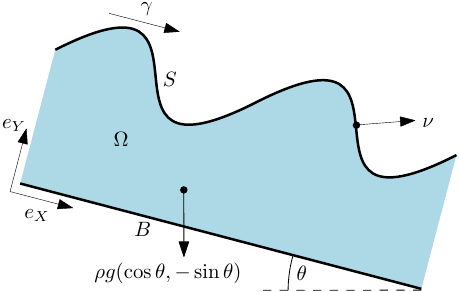}
    \caption{A sample portion of the fluid domain.}
\end{figure}

\begin{subequations}\label{eq:dimensional}
    We consider such a fluid travelling with a constant density $\rho > 0$ and viscosity $\mu > 0$, subject to a uniform vertical gravitational field $\rho g (\cos \theta, -\sin \theta)$, where $g > 0$ is the magnitude of gravitational acceleration. This gravitational force is the only bulk force acting on the fluid. Let $V \map{\Omega}{\Real^2}$ denote the velocity of the fluid in the travelling frame, and let $P \map{\Omega}{\Real}$ denote its gauge pressure, i.e.~the fluid pressure minus the constant atmospheric pressure. The flow in the bulk is governed by the steady incompressible Navier--Stokes equations,
    \begin{alignat}{2}
        \rho (V \cdot \nabla) V - \mu \Delta V + \nabla P &= \rho g (\cos \theta e_X - \sin \theta e_Y)     \qquad &\tn{in } \Omega,\\
        \nabla \cdot V &= 0      &\tn{in } \Omega, \label{eq: incomp}
    \end{alignat}
    which represent momentum balance and conservation of mass for the fluid.

    The only surface force considered is surface tension $\sigma H$, where $\sigma \geq 0$ is the coefficient of surface tension and $H$ is the mean curvature of $\Surf$. Let $\nu \in \Real^2$ denote the outward unit normal to the surface $\Surf$ and $\D V \in \Real_{\mathrm{sym}}^{2 \cross 2}$ denote twice the symmetrized gradient of $V$,
    \begin{equation*}
        \D V := D V + {(D V)}^T.
    \end{equation*}
    The equations on $\Surf$ are then given by a kinematic boundary condition, representing no penetration across the surface,
    \begin{equation}
        V \cdot \nu = 0   \quad \tn{on } \Surf, \label{eq: kinematic}
    \end{equation}
    and a dynamic boundary condition, which represents the force balance on the surface,
    \begin{equation}
        \left[(P + \sigma H)I - \mu \D V\right] \nu = 0     \quad \tn{on } \Surf.
    \end{equation}
    We assume a standard no-slip condition on $\Bed$, i.e.~that the fluid velocity vanishes at the lower boundary in the lab frame. In our travelling reference frame, this condition is 
    \begin{equation}
        V = - \gamma e_X \quad \tn{on } \Bed. \label{eq: noslip}
    \end{equation}
\end{subequations}
This flow has relative mass flux $m \in \Real$, defined by 
\begin{equation}\label{eq: mass_flux}
    m := \rho \int_{\Omega \cap \{X = X_0\}} V \cdot e_X \, dY,
\end{equation}
which is a constant independent of $X_0 \in \Real$, by conservation of mass~\eqref{eq: incomp} and the no-penetration boundary conditions~\eqref{eq: kinematic} and~\eqref{eq: noslip}. The problem for inclined travelling viscous surface waves is then to find $V$, $P$ and $\Omega$ satisfying the system of equations~\eqref{eq:dimensional} above. 

For any wavespeed $\gamma > 0$, density $\rho > 0$, viscosity $\mu > 0$, gravitational constant $g > 0$, mass flux $m \in \Real$ and incline angle $0 < \theta \leq \pi/2$, the problem~\eqref{eq:dimensional} admits a unique shear solution $V = \VNusselt e_X$, $P = P_0$ at a certain height $h$,
\begin{equation}\label{eq: dim_shear}
    \VNusselt(Y) := - \gamma + \frac{\rho g \sin \theta}{2 \mu}\left(2hY-Y^2\right), \quad P_0(Y) := \rho g \cos \theta (2h - Y), \quad \Omega_0 = \Real \times (0, h).
\end{equation}
This shear profile is referred to in the literature as the \emph{Nusselt solution}~\cite{nusselt1916oberflachenkondensation}. Note that the height $h$ is uniquely determined by computing~\eqref{eq: mass_flux} for the Nusselt profile,
\begin{equation}\label{eq: flux_nusselt}
    m = \rho \left(- \gamma h + \frac{\rho g h^2 \sin \theta}{3 \mu} \right).
\end{equation}

To non-dimensionalise~\eqref{eq:dimensional}, we use the unique Nusselt solution~\eqref{eq: dim_shear} of the same mass flux to identify characteristic scales for our flow. The Nusselt height $h$ defined by~\eqref{eq: flux_nusselt} will serve as a characteristic length scale for our problem, and we use as characteristic velocity scale the surface velocity $\mathcal{V}$ of this Nusselt flow in the lab frame, namely 
\begin{equation*}
    \mathcal{V} := \VNusselt(h) + \gamma = \frac{\rho g h^2 \sin \theta}{2 \mu}.
\end{equation*}
We define a Reynolds number $\Rey \geq 0$ by 
\begin{equation*}
    \Rey := \frac{\rho h \mathcal{V}}{\mu},
\end{equation*}
and fix units for mass, length and velocity, so that $\rho = 1$, $h=1$ and $\mathcal{V} = 1$. In these units, the mass flux of the fluid $m = \frac{2}{3} - \gamma$ and~\eqref{eq:dimensional} becomes 
\begin{equation} \label{eq:dimensionless}
\begin{cases}
  \Rey (V \cdot \nabla) V - \Delta V + \nabla P = 2 e_X - 2 \cot \theta e_Y &\tn{in } \Omega, \\
  \nabla \cdot V = 0 &\tn{in } \Omega, \\
  V \cdot \nu = 0 &\tn{on } \Surf, \\
  [(P + \sigma H)I - \D V] \nu = 0   &\tn{on } \Surf, \\
  V = - \gamma e_X         &\tn{on } \Bed,
\end{cases}
\end{equation}
while the Nusselt shear solution~\eqref{eq: dim_shear} to~\eqref{eq:dimensionless} becomes
\begin{equation} \label{eq: shear_sol}
    \VNusselt(Y) = - \gamma + 2Y-Y^2, \qquad P_0(Y) = 2 \cot \theta (1 - Y), \qquad \Omega_0 = \Real \times (0,1).
\end{equation}

\subsection{Orr--Sommerfeld equation}\label{section: os_intro}
Linearising~\eqref{eq:dimensionless} about~\eqref{eq: shear_sol} and separating variables for periodic disturbances leads to the celebrated \emph{Orr--Sommerfeld equation}~\cite{orr1907stability1, orr1907stability2, sommerfeld1909beitrag},
\begin{subequations}\label{eq: os_system}
    \begin{equation}\label{eq: os_int}
        {\left(\frac{d^2}{dY^2} - k^2\right)}^2 \hat{\varphi} - ik \Rey \VNusselt \left(\frac{d^2}{dY^2} - k^2\right) \hat{\varphi} + ik \Rey \VNusselt''\hat{\varphi} = 0 \qquad \text{for}~Y \in [0,1], 
    \end{equation}
    with boundary conditions for our free-surface flow given by 
    \begin{alignat}{2}
        \hat{\varphi}'' + \Big(k^2 +  \frac{2}{\VNusselt}\Big) \hat{\varphi} = 0& \qquad & \text{at}~Y=1, \\
        \hat{\varphi}''' - \left(3  k^2 + ik \Rey \VNusselt\right) \hat{\varphi}' + \frac{ik}{\VNusselt} \big(2\cot \theta + \sigma k^2 \big) \hat{\varphi} = 0& \qquad & \text{at}~Y=1, \\
        \hat{\varphi} = 0& \qquad & \text{at}~Y=0, \\
        \hat{\varphi}' = 0& \qquad & \text{at}~Y=0.
    \end{alignat}
\end{subequations}
Here $k \geq 0$ is the wavenumber of the disturbance and the function $\hat{\varphi} \colon [0,1] \to \Complex$ is related to the perturbative stream function for the flow. For our velocity profile $\VNusselt$,~\eqref{eq: os_int} is a variable coefficient ODE which is analytically intractable for general parameter values, with no exact solution in terms of standard functions. Instead, the Orr--Sommerfeld equation is usually studied via numerical methods~\cite{orszag1971accurate} or asymptotic expansions in parameter regimes~\cite{benjamin1957wave, yih1963stability}; see for example~\cite{chang2002complex, schmid2012stability} for a review of the theory.

Let us express~\eqref{eq: os_system} as $\OS \hat{\varphi} = 0$, where the operator $\OS \map{\Xos}{\Yos}$ and spaces $\Xos$, $\Yos$ will be carefully defined in Section~\ref{section: os} below. We write $\OS = \OS(k, \Rey, \theta, \gamma)$, making the dependence on the parameters $k$, $\Rey$, $\theta$ and $\gamma$ explicit. Throughout our analysis we will fix $\sigma \geq 0$.

\subsection{Statement of main result}
Our first result constructs roll wave solutions to~\eqref{eq:dimensionless} close to the Nusselt solution~\eqref{eq: shear_sol} by techniques from local bifurcation theory. The assumptions for this result are stated solely in terms of the Orr--Sommerfeld operator $\OS$ introduced above. A more detailed version of this result is given in Theorem~\ref{thm: gen_local} below. Here, we fix once and for all a H\"older exponent $\alpha \in (0,1)$.
\begin{theorem}[Local bifurcation]\label{thm: general_local}
    Fix $k > 0$ and $\Rey \geq 0$, and suppose that for some $\theta_* \in (0,\pi/2]$, $\gamma_* \in (0,\infty)\setminus\{2\}$, the following hold.
    \begin{enumerate}[label=\rm(\Roman*)]
        \item\label{assumption: os_kernel}\textup{(Orr--Sommerfeld kernel)} There exists $\hat{\varphi}_* \in \Xos$ such that 
        \begin{align}
            \ker \OS(k, \Rey, \theta_*, \gamma_*) &= \operatorname{span} \{\hat{\varphi}_*\}, \label{eq: os_1dker}\\
            \ker \OS(nk, \Rey, \theta_*, \gamma_*) &= \{0\} \quad \text{for } n = 2, 3, \ldots \label{eq: os_nonres}
        \end{align} 
        \item\label{assumption: os_transv}\textup{(Orr--Sommerfeld transversality)} For all nonzero $(\dot{\theta}, \dot{\gamma}) \in \Real^2$, we have
        \begin{equation*}
            \big(\d_\theta \OS(k,\Rey,\theta_*,\gamma_*) \dot{\theta}\big) \hat{\varphi}_* + \big(\d_\gamma \OS(k,\Rey,\theta_*,\gamma_*) \dot{\gamma}\big) \hat{\varphi}_* \notin \operatorname{ran} \OS(k,\Rey,\theta_*,\gamma_*).
        \end{equation*}
    \end{enumerate}
    Then there is a local nontrivial analytic curve of solutions to~\eqref{eq:dimensionless},
    \begin{equation*}
        \mathscr{C}_{\mathrm{loc}} := \left\{\left( V(s), P(s), \Omega(s), \theta(s), \gamma(s) \right) : 0 \leq \lvert s \rvert < \epsilon \right\},
    \end{equation*}
    bifurcating from the Nusselt solution~\eqref{eq: shear_sol} with $\theta = \theta_*$ and $\gamma = \gamma_*$ at $s=0$. The solutions are periodic in $X$, with H\"older regularity $V \in C^{3,\alpha}(\Omega)$ and $P \in C^{2,\alpha}(\Omega)$ with $\Omega \in C^{4,\alpha}$.
\end{theorem}
\begin{remark}\label{solution_expansions}
    The solutions in Theorem~\ref{thm: general_local} have explicit asymptotic expansions for small $s$. Roughly speaking, the physical wavenumber is approximated by $k$ up to order $s$, and the free surface $\Surf$ is given by $\Surf = \{(X,\mathcal{Y}(X,s)) : X \in \Real \}$, where
    \begin{equation}\label{eq: Y_exp}
        \mathcal{Y}(X; s) = 1 + s \cos kX + O(s^2),
    \end{equation}
    while the velocity and pressure have expansions
    \begin{equation}\label{eq: VP_exp}
        \begin{aligned}
            V(X,Y; s) &= V_0(Y)\begin{pmatrix}
                1 \\ 0
            \end{pmatrix} + s \Re \bigg\{ \begin{pmatrix} \hat{\varphi}'_*(Y) \\ -iK \hat{\varphi}_*(Y) \end{pmatrix} e^{ikX} \bigg\} + O(s^2), \\
            P(X,Y; s) &= P_0(Y) + s \Re \big\{ \hat{q}_*(Y) e^{ikX} \big\} + O(s^2),            
        \end{aligned}
    \end{equation}
    where $\hat{q}_*$ is a function determined by $\hat{\varphi}_*$. See the discussion after Theorem~\ref{thm: gen_local} for more detail. 
\end{remark}
\begin{remark}
  The hypotheses of Theorem~\ref{thm: general_local} avoid the `zero frequency' $k=0$ degeneracy at $\gamma = 2$, i.e.~when the travelling wavespeed is twice the Nusselt surface velocity in the lab frame. This degenerancy is discussed in more detail in the proof of Lemma~\ref{Kernels_equivalence} and is, roughly speaking, the source of the local bifurcation for small $k$ in Theorem~\ref{thm: smallk} below.
\end{remark}
The hypotheses~\ref{assumption: os_kernel}--\ref{assumption: os_transv} are natural assumptions on the Orr--Sommerfeld operator at the bifurcation point. The kernel hypothesis~\ref{assumption: os_kernel} corresponds to the linearised PDE operator having a one-dimensional kernel, and the transversality hypothesis~\ref{assumption: os_transv} is a non-degeneracy condition. However, these hypotheses are difficult to assess in general, due to the challenges in rigorously analysing the spectrum of the Orr--Sommerfeld equation (see, for instance~\cite{drazin2004hydrodynamic}).

Note that our curve has two parameters. For simplicity, we choose the angle $\theta$ and wavespeed $\gamma$ as bifurcation parameters, however analogues of Theorem~\ref{thm: general_local} hold with any pair of parameters from the list $k, \Rey, \theta, \gamma, \sigma$. 

Our second result extends the local curve found in Theorem~\ref{thm: general_local}. Recall from~\eqref{eq: f_intro} the conformal map $f$ from the rectangular strip to $\Omega$, and denote by $\Top := \Real \times \{1\}$ the upper boundary of the strip. This result is written in more detail in Theorem~\ref{thm: gen_global}.
\begin{theorem}[Global continuation]\label{thm: general_global}
    In the setting of Theorem~\ref{thm: general_local}, the local curve $\mathscr{C}_{\mathrm{loc}}$ can be extended to a global curve $\mathscr{C}$ of periodic solutions, parameterised by $s \in \Real$. Along this curve, one of the following alternatives occur:
    \begin{enumerate}[label=\rm(\alph*)]
        \item \textup{(Blow-up)} If $\sigma = 0$, then as $\lvert s \rvert \to \infty$,
        \begin{equation*} 
            \sup_{\substack{z_1, z_2 \in \Top \\ z_1 \neq z_2}} \frac{\lvert z_1 - z_2 \rvert}{\lvert f(z_1) - f(z_2) \rvert} + \sup_\Top \lvert Df \rvert + \sup_\Top \lvert D^2 f \rvert + \sup_\Omega \lvert V(s) \rvert + \sup_\Surf\lvert DV(s)\rvert + \frac{1}{\inf_\Surf \lvert V(s)\rvert} + \lvert \cot \theta(s)\rvert \longrightarrow \infty,
        \end{equation*}
        and if $\sigma > 0$, then as $\lvert s \rvert \to \infty$,
        \begin{equation*} 
            \sup_{\substack{z_1, z_2 \in \Top \\ z_1 \neq z_2}} \frac{\lvert z_1 - z_2 \rvert}{\lvert f(z_1) - f(z_2) \rvert} + \sup_\Top \lvert Df \rvert + \sup_\Top \lvert D^2 f \rvert + \sup_\Omega \lvert V(s) \rvert + \sup_\Surf\lvert DV(s)\rvert + \lvert \cot \theta(s)\rvert \longrightarrow \infty.
        \end{equation*}\label{blowup_alternative}
        \item \textup{(Closed loop)} At some point the solution curve $\mathscr{C}$ returns back to the Nusselt profile~\eqref{eq: shear_sol} with angle $\theta_*$ and wavespeed $\gamma_*$.\label{loop_alternative}
    \end{enumerate}
\end{theorem}

Divergence of the first term in alternative~\ref{blowup_alternative} represents loss of injectivity of the conformal map or $C^1$ blow-up of its inverse on $\Top$. Loss of injectivity could correspond to the wave profile self-intersecting, while blow-up of the inverse of the conformal map in $C^1$ could be due to the formation of inward-pointing corners or the surface touching the bed. Blow-up of the gradient of the conformal map on $\Top$ would be consistent with the formation of outward-pointing corners, while blow-up of the Hessian of $f$ would be consistent with surface points of large curvature. Divergence of the fourth term would imply large velocity in the moving frame. Blow-up of the gradient of the velocity on the surface would be consistent with blow-up of the stress tensor or vorticity. In the case of zero surface tension, vanishing of the (travelling) velocity on the surface corresponds to approaching a stagnation point on the surface, consistent with a surface singularity. However, the presence of surface tension removes this term in the alternative, consistent with the smoothing effect of surface tension. The term in $\theta$ represents possible vanishing of the angle of inclination, which could correspond to full stagnation, i.e.~$V \equiv \gamma e_1$. 
 
Regarding the possibility of a closed loop, alternative~\ref{loop_alternative}, in global bifurcation results for water wave problems, maximum principle arguments and reflection symmetry can often be used to rule out such a loop. However,~\eqref{eq:dimensionless} does not possess a reflection symmetry, and it is not clear that the problem admits a scalar elliptic maximum principle. Indeed, it seems plausible that a loop could occur for some values of $k, \Rey$, and $\sigma$. 

\subsection{Results in special parameter regimes}
As mentioned above, the hypotheses~\ref{assumption: os_kernel}--\ref{assumption: os_transv} in Theorem~\ref{thm: general_local} are difficult to verify in general. However, in two asymptotic parameter regimes, we can find angles and wavespeeds that satisfy these hypotheses. 

The first such regime is when the conformal wavenumber $k$ (and consequently the physical wavenumber $K$, at least locally along the bifurcation branch) is small, corresponding to long waves or thin films. 
\begin{theorem}[Bifurcation at small wavenumber]\label{thm: smallk}
    Fix $\Rey \geq 0$. Then, there exist $\varepsilon > 0$ and real analytic functions $\tilde{\theta} \colon (-\varepsilon, \varepsilon) \to \Real$ and $\tilde{\gamma} \colon (-\varepsilon, \varepsilon) \to \Real$ with the following properties.
    \begin{enumerate}[label=\rm(\alph*)]
        \item Up to order $k^2$, $\tilde{\theta}$ and $\tilde{\gamma}$ are given by
        \begin{equation}\label{eq: intro_smallk_param_exp}
            \cot \tilde{\theta}(k) = \frac{4}{5} \Rey + O(k^2), \qquad
            \tilde{\gamma}(k) = 2 + O(k^2),
        \end{equation}
        with Orr--Sommerfeld kernel spanned by $\tilde{\varphi}(y; k)$, where
        \begin{equation}\label{eq: intro_smallk_exp}
            \tilde{\varphi}(y; k) = \frac{y^2}{2} + ik \Rey \left(\frac{y^{5}}{60} - \frac{y^{4}}{12} + \frac{2 y^{3}}{15}\right) + O(k^2).
        \end{equation}
        \item For $k \in (0, \varepsilon) $ outside a countable set, $\theta_* = \tilde{\theta}(k)$ and $\gamma_* = \tilde{\gamma}(k)$ satisfy the hypotheses of Theorem~\ref{thm: general_local}. Thus, each such $(\theta_*, \gamma_*)$ corresponds to a bifurcation point for a global curve of solutions to~\eqref{eq:dimensionless} described in Theorem~\ref{thm: general_global}.
    \end{enumerate}
\end{theorem}
We prove this result by finding a curve of parameters $\big(\tilde{\theta}, \tilde{\gamma}\big)$ along which the Orr--Sommerfeld operator $\OS$ satisfies the first part~\eqref{eq: os_1dker} of the kernel assumption~\ref{assumption: os_kernel} via an auxiliary bifurcation argument at $k=0$. This corresponds to a `neutral stability' curve found in~\cite{benjamin1957wave}, with values $\tilde{\gamma}(0) = 2$ and $\cot \tilde{\theta}(0) = \frac{4}{5} \Rey$ at its branch point. The technical requirement that the wavenumber lies outside a countable set is to avoid potential resonances in the kernel of the linearised operator, that is, to ensure~\eqref{eq: os_nonres} in~\ref{assumption: os_kernel}. Verifying the transversality condition uses the rigorous asymptotic expansions~\eqref{eq: intro_smallk_param_exp} and~\eqref{eq: intro_smallk_exp} for $\tilde{\theta}(k)$, $\tilde{\gamma}(k)$ and $\tilde{\varphi}(k)$, as well as for the kernel of a formal adjoint operator $\OS^*$. 

We are also able to verify the hypotheses of Theorem~\ref{thm: general_local} for the case of low Reynolds number, i.e.~highly viscous flow.
\begin{theorem}[Bifurcation at low Reynolds number]\label{thm: lowR}
    Set $\sigma = 0$, and fix $k > 0$. Then, there exists $\varepsilon > 0$ and analytic functions $\tilde{\theta} \colon (-\varepsilon, \varepsilon) \to \Real$ and $\tilde{\gamma} \colon (-\varepsilon, \varepsilon) \to \Real$ with the following properties.
    \begin{enumerate}[label=\rm(\alph*)]
        \item $\displaystyle{\tilde{\theta}(0) = \frac{\pi}{2}}$ and $\displaystyle{\tilde{\gamma}(0) = 1 + \frac{1}{k^2 + \cosh^2 k}}$.
        \item For any $\Rey \in [0, \varepsilon)$, the conclusion of Theorem~\ref{thm: general_local} holds with $\theta_* = \tilde{\theta}(R)$ and $\gamma_* = \tilde{\gamma}(R)$. Thus, each of these $(\theta_*, \gamma_*)$ corresponds to a bifurcation point for a global curve of solutions to~\eqref{eq:dimensionless} described in Theorem~\ref{thm: general_global}.
    \end{enumerate}
\end{theorem}
We prove this result by perturbation of the $\Rey = 0$ limit, which is referred to as `Stokes flow'. At zero Reynolds number and without surface tension, the Orr--Sommerfeld kernel assumption~\ref{assumption: os_kernel} holds at a vertical incline, as observed in~\cite{yih1963stability}. The transversality condition~\ref{assumption: os_transv} is verified by explicit computation. Since the vanishing Reynolds number limit is non-singular, by a soft argument Hale's local bifurcation result can be perturbed with an additional parameter, and we obtain bifurcations at low Reynolds numbers.

Note that since surface tension is not permitted in Theorem~\ref{thm: lowR}, the corresponding global bifurcation statement includes the possibility of stagnation in the blow-up alternative~\ref{blowup_alternative} of Theorem~\ref{thm: general_global}. However, if one selected $\sigma$ as one of the bifurcation parameters and could verify the transversality condition for the new choice of parameters, one would obtain a similar result with nonzero surface tension along the curve.

A final remark on these results is that despite the curve of solutions bifurcating with fixed $k$ and $\Rey$, along the global curve, the amplitude of solutions could grow large enough to no longer lie in a thin-film regime for the curve in Theorem~\ref{thm: smallk} or a highly viscous regime in Theorem~\ref{thm: lowR}.

\subsection{Review of previous work}
The mathematical study of surface waves has a long history, originating with the development of theory for inviscid, irrotational potential flow in the 18th and 19th century (see, for instance~\cite{craik2004origins, darrigol2003spirited}) with later incorporation of viscosity and free-boundary constraints. Given the breadth of the literature on surface waves, we restrict our discussion below to roll waves and the rigorous theory for travelling incompressible inviscid and viscous waves.

Although our work focuses on laminar roll waves, there is an older literature for roll waves in turbulent open channels~\cite{cornish1907progressive, cornish1934ocean}, historically motivated by engineering applications in hydraulics. There is a vast array of experimental, theoretical and numerical results for turbulent roll waves; see~\cite{jeffreys1925, thomas1940propagation, dressler1949mathematical, brock1969development, brock1970periodic, needham1984roll, kranenburg1992evolution, balmforth2004dynamics, barker2017stability} and the references therein. Models studied include the Saint-Venant equations with Ch\'ezy drag, viscous Saint-Venant equations and more recent models such as~\cite{richard2012new, richard2024roll}. We emphasize that while the term `roll wave' has historically been used (for example in~\cite{dressler1949mathematical}) to refer to this turbulent case, more recently it has broadened to encompass any periodic open channel flow down an incline, including in mud flow~\cite{liu1994roll, ng1994roll, balmforth2004roll}, granular flow~\cite{forterre2003long} and laminar flow~\cite{julien1986formation}.

Laminar roll waves were first experimentally observed by Kapitza and Kapitza~\cite{kapitza1949wave} for vertical flow, and these waves are sometimes referred to eponymously as `Kapitza' waves. Seminal theoretical contributions came from the work of Benjamin~\cite{benjamin1957wave} and Yih~\cite{yih1963stability}. Benjamin derived a critical Reynolds number for wave instability for the linearised Navier--Stokes system by computing formal asymptotic expansions for the eigenvalue problem for Orr--Sommerfeld in a long wave regime. Yih added an analysis at low Reynolds number, finding a neutral stability curve at zero Reynolds number without surface tension. For a more complete review of hydrodynamic stability theory, we refer the reader to textbooks such as~\cite{drazin2004hydrodynamic, craster2009dynamics, schmid2012stability}. Due to the challenges in analysing the Navier--Stokes equations and the related Orr--Sommerfeld system, subsequent work focussed on deriving and analysing nonlinear model equations: see~\cite{mei1966nonlinear, benney1966long, shkadov1967wave, gjevik1970occurrence, nakaya1975long, topper1978approximate, alekseenko1985wave, julien1986formation, ruyer1998modeling, ruyer2000improved} and the discussion of these in textbooks for falling liquid films~\cite{chang2002complex, kalliadasis2011falling}.

On the other hand, there is a largely separate literature on rigorous surface wave theory. The viscous time-dependent problem first received rigorous treatment by Beale who established a well-posedness result for the initial value problem~\cite{beale1981initial, beale1984large}, with subsequent extensions and decay results reviewed in, for example~\cite{zadrzynska2004free, shibata2007free} and the introduction of~\cite{leoni2023traveling}. Of closest relevance to our work is the literature for inclined periodic flow, with local and global existence results~\cite{teramoto1985initial, ninomiya1992surface, teramoto1992navier, nishida1993navier} and stability results~\cite{sun1997stability, nishida2004global, padula2013stability, tice2018asymptotic}.

The steady inviscid case, known as the travelling `water wave' problem, has been comprehensively studied and is reviewed in~\cite{toland1996stokes, groves2004steady, strauss2010steady, haziot2022traveling}. Here we mention a few aspects which are motivating for our work. The use of conformal maps is a classical technique in this community, used for instance to reformulate the irrotational Stokes wave problem into an integral equation by Nekrasov~\cite{nekrasov1921steady} and Levi-Civita~\cite{levi1925determination}, and later to the remarkable `Babenko' equation~\cite{babenko1987some, toland2002pseudo}. More recently it has enabled constructions for the rotational problem of waves which are allowed to overturn, by Constantin, Strauss and Varvaruca~\cite{constantin2016global}; see also~\cite{hur2022overhanging, haziot2023large, wahlen2024large, davila2026overhanging}. Global bifurcation techniques were first employed for water waves by Keady and Norbury~\cite{keady1978existence}, who improved Krasovskii's result~\cite{krasovskii1961theory} by applying Rabinowitz's topological global bifurcation theory~\cite{rabinowitz1971some} to the irrotational Stokes wave problem. Dancer strengthened Rabinowitz's global result in the special case of analytic operators~\cite{dancer1973bifurcation, dancer1973global}, and this was refined by Buffoni and Toland~\cite{buffoni2003analytic}, producing a powerful tool which is now standard for water waves problems. Haziot and Wheeler~\cite{haziot2023large} were the first to employ theory for local elliptic systems rather than scalar boundary-value problems (see also~\cite{doak2026large}).

In contrast, rigorous work on the viscous travelling wave problem is much more recent, starting from the result of Leoni and Tice~\cite{leoni2023traveling} who constructed travelling viscous waves under small external forcing. This has been generalised to multilayer~\cite{stevenson2021traveling}, compressible~\cite{stevenson2023well} and stationary~\cite{stevenson2024well} flow by Stevenson and Tice, with inclination and periodicity~\cite{koganemaru2023traveling} and Navier slip conditions at the bed~\cite{koganemaru2024traveling} incorporated by Koganemaru and Tice. Banihashemi and Nguyen~\cite{banihashemi2026large} constructed solutions to the Stokes and Navier--Stokes system close to a special class of large waves produced by surface stress and proved asymptotic stability for the dynamic Stokes problem. Most recently, Choi and Tice~\cite{choi2026traveling} obtained solutions for a generalised Navier--Stokes--Fourier system. These works mentioned above employ perturbative arguments based on the implicit function theorem or for~\cite{stevenson2023well} a Nash--Moser inverse function theorem, avoiding potential bifurcation points. The solutions are produced by generic bulk, stress or heat forcing. As mentioned already, Stevenson and Tice have also constructed shallow-water bore wave solutions to the inclined free-boundary Navier--Stokes equations~\cite{stevenson2025gravity}, alongside related work for viscous shallow water models~\cite{stevenson2026traveling, stevenson2025periodic, stevenson2025stationary}. We also mention the recent results for Darcy flow in a porous medium~\cite{nguyen2024traveling, nguyen2026largecapillary, brownfield2024slowly, nguyen2026large} of which~\cite{nguyen2026largecapillary, nguyen2026large} produce large-amplitude solutions.

Our result is the first global construction for the travelling free-boundary Navier--Stokes equations and also the first rigorous construction of roll wave solutions to this Navier--Stokes system. We draw heavily from techniques used in the inviscid water wave literature, in particular conformal mappings, analytic global bifurcation and theory for elliptic systems. However, the substantial challenges of the viscous case appear in many aspects of our analysis. The linearised problem is not reducible to a scalar elliptic boundary value problem, truly requiring elliptic systems theory rather than scalar elliptic theory. Related to this is the lack of a maximum principle, which as mentioned in the discussion of Thorem~\ref{thm: general_global} above would simplify the bifurcation alternatives. Since there is not a reflection symmetry, the linearised operator has Fredholm index $-1$ around the Nusselt solution, motivating our use of Hale's result~\cite{hale1978bifurcation} for local bifurcation rather than the more classical result by Crandall--Rabinowitz~\cite{crandall1971bifurcation}. Additionally, the linearised operator does not have an explicit kernel and dispersion relation, so we require the general hypotheses~\ref{assumption: os_kernel}--\ref{assumption: os_transv} in Theorem~\ref{thm: general_local}, but our results in the regimes of small wavenumber and low Reynolds number are consistent with the neutral stability curves found in~\cite{benjamin1957wave, yih1963stability}.

\subsection{Outline of the paper}
In Section~\ref{section: prelims} we reformulate the problem in terms of the stream function and in a rectangular domain by a conformally mapping. We then define the function spaces we will work in and recall some abstract bifurcation results. Section~\ref{section: elliptic} establishes that the linearisation of the reformulated system is elliptic in the sense of~\cite{agmon1964estimates} and has Fredholm index $-1$ when linearised at the Nusselt solution. In Section~\ref{section: os}, we reduce the kernel and transversality hypotheses of the bifurcation results from Section~\ref{section: prelims} to analogous conditions on the Orr--Sommerfeld operator. Some a priori uniform estimates are proved for solutions to the nonlinear problem in Section~\ref{section: uniform}. We bring these results together to prove Theorems~\ref{thm: general_local} and~\ref{thm: general_global} in Section~\ref{section: proof_main}. Finally, we conduct perturbative analyses to verify the hypotheses on the kernel and transversality condition for the Orr--Sommerfeld operator in the low wavenumber regime and in the low Reynolds number regime. In Appendix~\ref{section: extra_abstract_theory}, we provide proofs for some of the abstract bifurcation theory for completeness, and we briefly recall relevant theory for elliptic systems in Appendix~\ref{section: ADN_theory} for the reader's convenience.

\section{Preliminaries}\label{section: prelims}
This section details our reformulation of~\eqref{eq:dimensionless}, replacing the velocity with the stream function and using the conformal map $f$ mentioned in the introduction to rewrite the problem on a flattened rectangular domain. We also introduce some functional notation for the problem, and state convenient versions of the bifurcation results from Hale~\cite{hale1978bifurcation} and Buffoni--Toland~\cite{buffoni2003analytic} that we will use in Section~\ref{section: proof_main}.

\subsection{Stream function reformulation}
The incompressibility condition~\eqref{eq: incomp} implies the existence of a scalar potential, known as the \emph{stream function} $\Psi$ of the flow, such that
\begin{equation}\label{eq:stream_function}
    \Psi_X = - V_2, \quad \Psi_Y = V_1,
\end{equation}
where non-numerical subscripts denote partial derivatives, i.e.~$\Psi_X = \d_X \Psi$ and so on. Reformulating~\eqref{eq:dimensionless} with the velocity $V$ replaced by the stream function $\Psi$ reduces the number of unknowns for the problem by one and incorporates~\eqref{eq: incomp} into the system automatically, reducing the number of interior equations by one. 

Level sets of the stream function are tangent to the velocity field, and are called \emph{streamlines}. The kinematic boundary condition~\eqref{eq: kinematic} is therefore equivalent to the free surface being a level set of the stream function. Similarly, from the first component of~\eqref{eq: noslip}, $\Psi$ is constant on $\Bed$ as well, and we normalise so that $\Psi=0$ on $\Bed$. Then the value of the stream function on $\Surf$ is equal to the (relative) mass flux $m$, which in our non-dimensional units is given by $\frac{2}{3} - \gamma$, so~\eqref{eq: kinematic} can be replaced by
\begin{equation*}
    \Psi = \frac{2}{3}-\gamma\qquad \text{on } \Surf.
\end{equation*}
Note that the other component of~\eqref{eq: noslip} becomes $\Psi_Y = - \gamma$ on $\Bed$.

Denoting by $\mathcal{D} \Psi \in \Real_{\mathrm{sym}}^{2\cross2}$ the symmetrised gradient for the velocity $\D V$ in terms of the stream function,
\begin{equation*}
    \mathcal{D} \Psi := \begin{pmatrix}
        2 \Psi_{XY} & \Psi_{YY} - \Psi_{XX} \\
        \Psi_{YY} - \Psi_{XX} & -2 \Psi_{XY}           \\
    \end{pmatrix},
\end{equation*}
the system~\eqref{eq:dimensionless} becomes
\begin{equation}\label{eq: sf_formulation}
    \begin{cases}
        - \Delta \Psi_Y + \Rey \nabla \Psi \cdot \nabla^{\perp} \Psi_Y + P_X = 2 & \text{in } \Omega, \\
        \Delta \Psi_X + \Rey \nabla^\perp \Psi \cdot \nabla \Psi_X + P_Y = - 2 \cot \theta & \text{in } \Omega, \\
        \Psi = \frac{2}{3}-\gamma& \text{on } \Surf, \\
        [(P + \sigma H)I - \mathcal{D} \Psi ] \nu = 0 & \text{on } \Surf, \\
        \Psi = 0 & \text{on } \Bed, \\
        \Psi_Y = - \gamma & \text{on } \Bed.
    \end{cases}
\end{equation}
The stream function $\PsiNusselt$ corresponding to the Nusselt solution~\eqref{eq: shear_sol} is given by
\begin{equation}\label{eq: Psi}
    \PsiNusselt(Y ; \gamma) := -\gamma Y + Y^2 - \tfrac{1}{3} Y^3.
\end{equation}

\subsection{Reformulation in conformal variables}
As mentioned in the introduction, we write the fluid domain as the image of a rectangular strip under a conformal map $f$. Since we are restricting to periodic solutions, we work with a periodic rectangular domain $\Rect$ of unit height and wavenumber $k > 0$,
\begin{equation}
    \Rect := \Torus_k \times (0,1), \quad \text{ where } \Torus_k := \Real \big/ \tfrac{2 \pi}{k} \mathbb{Z},
\end{equation}
denoting the upper boundary of $\Rect$ by $\Top := \Torus_k \times \{1\}$ and the lower boundary by $\Bottom := \Torus_k \times \{0\}$. We now introduce $f \colon \overline{\Rect} \to \overline{\Omega}$, identifying 
\begin{equation}\label{eq: image_of_f}
    \Omega = f(\Rect), \quad \Surf = f(\Top), \quad \Bed = f(\Bottom).
\end{equation}
Denote by $\xi$ and $\eta$ the components of $f$, i.e.~for every $(X,Y) \in \Omega$, there exists $(x,y) \in \Rect$ such that
\begin{equation*}
    (X,Y) = (\xi(x,y), \eta(x,y)) = f(x,y).
\end{equation*}

For this change of coordinates to be well-behaved, that is, for the domain $\Omega$ to be well-defined with $\Surf$ not self-intersecting or intersecting the bed $\Bed$, we require the following conditions on $f$. First, we suppose that $\xi$ and $\eta$ are continuous in $\overline{\Rect}$, continuously differentiable in $\Rect$ and satisfy the Cauchy--Riemann equations,
\begin{equation}\label{eq: cauchy_riemann}
    \xi_x = \eta_y, \quad \xi_y = - \eta_x \qquad \text{in }\Rect.
\end{equation}
Let $f$ also be injective on $\Top$,
\begin{equation}\label{eq: f_injective}
    \inf_{\substack{x_1, x_2 \in \Torus_k \\ x_1 \neq x_2}} \frac{\lvert f(x_1, 1) - f(x_2, 1) \rvert}{\lvert x_1 - x_2 \rvert} > 0,
\end{equation}
and finally suppose $\eta$ is strictly positive on $\Top$,
\begin{equation}\label{eq: not_touching_bed}
    \eta(x,1) > 0 \quad \text{for all} ~ x \in \Torus_k.
\end{equation}
Then the following lemma confirms that $\overline{\Omega}$ is well-defined as the image of $\overline{\Rect}$ under $f$.
\begin{lemma}\label{lemma: conformal}
    Let $f= (\xi, \eta) \in C^1(\Omega) \cap C(\overline{\Omega})$ satisfy~\eqref{eq: cauchy_riemann}--\eqref{eq: not_touching_bed}. Then $f$ is bijective $\overline{\Rect} \to \overline{\Omega}$.
\end{lemma}
\begin{proof}
    The claim follows from the Darboux--Picard theorem; see for instance,~\cite[Theorem 9.16]{burckel2021classical}. The only hypothesis to check is injectivity on the boundary of $\Rect$. Clearly $f$ is injective on $\Bottom$, mapping to $\Bed$, and~\eqref{eq: f_injective} implies injectivity on $\Top$.
    Finally,~\eqref{eq: not_touching_bed} ensures that $f(\Top)$ does not touch $f(\Bottom)$, so the images of the components of $\d \Rect$ are pairwise disjoint. 
\end{proof}

Conversely, given periodic and smooth $\Omega$, there exist a wavenumber $k$ and a conformal map $f$ from $\Rect$ to $\Omega$ which admits an extension as a homeomorphism between $\overline{\Rect}$ and $\overline{\Omega}$, with the same H\"older regularity as $\Omega$. For more details see for instance,~\cite[Lemma 2.1]{wahlen2024large}.
 
We now write~\eqref{eq: sf_formulation} as a problem for $\psi$, $p$ and $\eta$ on $\Rect$, where $\psi$ and $p$ are the stream function and pressure as functions of the conformal variables,
\begin{equation}\label{eq: sf_p_cov}
    \psi(x,y) := \Psi(f(x,y)), \quad p(x,y) := P(f(x,y)).
\end{equation}
Denote by $\tilde{\d}_i$ the transformations of the respective derivatives $\d_i$ in $\Rect$, i.e.~if $G(X,Y)=g(x,y)$, then $\tilde{\d}_X g = \d_X G$, $\tilde{\d}_Y g = \d_Y G$, or explicitly,
\begin{equation*}
    \tilde{\d}_X := \frac{1}{\lvert \nabla \eta \rvert^2} \left( \eta_y \d_x - \eta_x \d_y \right), \quad    
    \tilde{\d}_Y := \frac{1}{\lvert \nabla \eta \rvert^2} \left( \eta_x \d_x + \eta_y \d_y \right).
\end{equation*}
Let us also define $\tilde{\nabla} := (\tilde{\d}_{X}, \tilde{\d}_{Y})$ and write the symmetrised velocity gradient in conformal variables as $\mathcal{D} \Psi(X,Y) =: \tilde{\mathcal{D}} \psi(x,y)$.
The unit normal vector $\nu$ and mean curvature $H$ of the surface $\Surf$ become
\begin{equation*}
    \nu(\eta) := \frac{1}{\lvert \nabla \eta \rvert^2} \begin{pmatrix} -\eta_x \\ \eta_y \end{pmatrix} \quad \text{and} \quad H(\eta) := \frac{\eta_y \eta_{xx} - \eta_x \eta_{xy}}{\lvert \nabla \eta \rvert^3}. 
\end{equation*}
We also add a new interior equation $\Delta \eta = 0$ in $\Rect$, recovering $\xi$ up to an additive constant as the harmonic conjugate of $\eta$ using~\eqref{eq: cauchy_riemann}. With this notation,~\eqref{eq: sf_formulation} becomes the problem
\begin{subequations}\label{eq:flattened}
    \begin{align}
        - \tilde{\d}_Y \left(\lvert \nabla \eta \rvert^{-2} \Delta \psi\right) + \Rey \tilde{\d}_Y \psi\tilde{\d}_{XY} \psi - \Rey \tilde{\d}_X \psi \tilde{\d}_{YY} \psi + \tilde{\d}_X p &= 2 && \text{in } \Rect, \label{eq: flattened1}\\
        \tilde{\d}_X \left(\lvert \nabla \eta \rvert^{-2}\Delta \psi\right) - \Rey \tilde{\d}_Y \psi \tilde{\d}_{XX} \psi + \Rey \tilde{\d}_X \psi \tilde{\d}_{XY} \psi + \tilde{\d}_Y p &= - 2 \cot \theta  && \text{in } \Rect, \label{eq: flattened2}\\
        \Delta \eta &= 0   && \text{in } \Rect, \label{eq: flattened3}\\
        \psi &= \frac{2}{3} - \gamma  && \text{on } \Top, \label{eq: flattened_kin}\\
        \big[\big(p + \sigma H(\eta)\big)I - \tilde{\mathcal{D}} \psi \big] \left( - \eta_x, \eta_y\right)^T &= 0  && \text{on } \Top, \label{eq: flattened_dyn} \\
        \psi &= 0  && \text{on } \Bottom, \label{eq: flattened_noslip1}\\
        \tilde{\d}_Y \psi &= - \gamma    && \text{on } \Bottom, \label{eq: flattened_noslip2}\\
        \eta &= 0  && \text{on } \Bottom, \label{eq: flattened_bed}
    \end{align}
\end{subequations}
in the rectangular domain $\Rect$, where~\eqref{eq: flattened1} and~\eqref{eq: flattened2} have been simplified using~\eqref{eq: flattened3}.

\subsection{Functional framework}
We will work in the Banach spaces
\begin{equation}\label{eq: spaces}
    \begin{aligned}
        \X &:= \left\{(\psi, p, \eta) \in C^{4, \alpha}(\Rect) \times C^{2, \alpha}(\Rect) \times C^{4, \alpha}(\Rect) : \Delta \eta = 0, ~ \psi|_{\Bottom} = \eta \rvert_{\Bottom} = 0, ~ \eta_x(0,1) = 0 \right\}, \\
        \Y &:= C^{1, \alpha}(\Rect) \times C^{1, \alpha}(\Rect) \times C^{4, \alpha} \left(\Top \right)\times \left(C^{2, \alpha}\left(\Top \right)\right)^2 \times C^{3, \alpha}\left(\Bottom\right).
    \end{aligned}
\end{equation}
The inclusion of the homogeneous linear equations $\Delta \eta = 0$ and $\psi|_{\Bottom} = \eta \rvert_{\Bottom} = 0$ in $\X$ is for convenience, and the constraint $\eta_x(0,1) = 0$ is introduced to eliminate the translation symmetry of solutions. For small amplitudes, it corresponds to shifting the wave peak or trough to $x=0$.

Define an open set $\U \subset \X \times \Real^2$ by
\begin{equation}\label{eq: open_set}
    \U = \left\{(\psi, p, \eta, \theta, \gamma) \in \X \times (0,\pi) \times (0, \infty) ~:~ \inf_{\substack{z_1, z_2 \in \Top \\ z_1 \neq z_2}} \frac{\lvert f(z_1) - f(z_2) \rvert}{\lvert z_1 - z_2 \rvert} > 0, ~\inf_\Top \eta > 0, ~\inf_\Top \lvert \nabla \psi \rvert + \sigma > 0 \right\}.
\end{equation}
The first two conditions are~\eqref{eq: f_injective} and~\eqref{eq: not_touching_bed}, where $f=(\xi, \eta)$ with $\xi$ being recovered by~\eqref{eq: cauchy_riemann}. By Lemma~\ref{lemma: conformal}, these conditions ensure $f$ is bijective. The additional condition $\inf_\Top \lvert \nabla \psi \rvert + \sigma > 0$ is necessary for ellipticity to hold, but has a physical meaning in that $\lvert \nabla \psi \rvert$ vanishing on $\Top$ would represent a stagnation point on the surface.

Now we define a function $\F \colon \U \to \Y$ such that the tuple $(u, \theta, \gamma) \in \U \subset \X \times \Real^2$ solving~\eqref{eq:flattened} is equivalent to $\F(u, \theta, \gamma) = 0$, namely
\begin{equation}\label{eq: F}
    \F (\psi, p, \eta; \theta, \gamma) = \begin{pmatrix}
        - \tilde{\d}_Y \left(\lvert \nabla \eta \rvert^{-2} \Delta \psi \right) - \Rey \tilde{\nabla}^{\perp} \psi \cdot \tilde{\nabla} \tilde{\d}_Y \psi + \tilde{\d}_X p - 2 \\
        \tilde{\d}_X \left(\lvert \nabla \eta \rvert^{-2} \Delta \psi \right) + \Rey \tilde{\nabla}^{\perp} \psi \cdot \tilde{\nabla} \tilde{\d}_X \psi + \tilde{\d}_Y p + 2 \cot \theta \\
        \left[\psi - \frac{2}{3} + \gamma \right]\big\vert_\Top \\
        \big[\big( p + \sigma H(\eta)\big)I - \tilde{\mathcal{D}} \psi \big] \left( - \eta_x, \eta_y\right)^T \big|_\Top \\
        \big[\tilde{\d}_Y \psi + \gamma\big] \big|_\Bottom
    \end{pmatrix}.
\end{equation}
It is standard to check that $\F$ is real analytic, since it is the composition of rational functions and bounded linear maps, that is, derivatives and traces.

We denote the Nusselt shear solution by 
\begin{equation}\label{eq: u0}
    u_0(\theta, \gamma) := (\PsiNusselt(y; \gamma), P_0(y; \theta), y),
\end{equation}
where $\PsiNusselt$ is defined in~\eqref{eq: Psi} and $P_0$ is defined in~\eqref{eq: shear_sol}. Note that for any $\theta \in (0, \pi)$ and  $\gamma > 0$, $(u_0(\theta, \gamma), \theta, \gamma) \in \U$, as the conformal map $f_0$ for the Nusselt profile is the identity $f_0(x,y) = (x,y)$, and so the left-hand sides of~\eqref{eq: f_injective} and~\eqref{eq: not_touching_bed} are both $1$. 

We denote by $\mathscr{L}(\theta, \gamma) := \F_u(u_0(\theta, \gamma), \theta, \gamma) \colon \X \to \Y$ the linearised operator around the Nusselt solution. This is found by straightforward calculation to be
\begin{equation} \label{eq: triv_lin}
    \mathscr{L}(\theta, \gamma) \begin{pmatrix} \dot{\psi} \\ \dot{p} \\ \dot{\eta} \end{pmatrix}
    = \begin{pmatrix}
        - \Delta \dot{\psi}_y + \dot{p}_x + \Rey\big( \PsiNusselt' (\dot{\psi}_{xy}- \PsiNusselt' \dot{\eta}_{xy}\big) - \PsiNusselt'' \dot{\psi}_x ) - P_0' \dot{\eta}_x - 6 \dot{\eta}_y - 2 \PsiNusselt'' \dot{\eta}_{xx}  \\
        \Delta \dot{\psi}_x + \dot{p}_y - \Rey \PsiNusselt' (\dot{\psi}_{xx} -  \PsiNusselt' \dot{\eta}_{xx}) - P_0' \dot{\eta}_y - \PsiNusselt''' \dot{\eta}_x - 2 \PsiNusselt'' \dot{\eta}_{xy} \\
        \dot{\psi} \vert_T\\
        \big(\dot{\psi}_{xx} - \dot{\psi}_{yy}\big)- \PsiNusselt' \left(\dot{\eta}_{xx} - \dot{\eta}_{yy}\right) \vert_\Top \\
        2 \dot{\psi}_{xy} + \dot{p} - 2 \PsiNusselt' \dot{\eta}_{xy} + \sigma \dot{\eta}_{xx} \vert_\Top \\
        \dot{\psi}_y - \PsiNusselt' \dot{\eta}_y \vert_\Bottom
    \end{pmatrix}.
\end{equation}
It will be convenient here and in Section~\ref{section: os} in the reduction to the Orr--Sommerfeld operator $\OS$ to compose $\mathscr{L}$ with an invertible change of basis map, $\mathcal{Q}(\theta, \gamma) \colon \X \to \X$,
\begin{equation}\label{eq: cov_matrix}
    \mathcal{Q}(y ; \theta, \gamma) := \begin{pmatrix}
        1 & 0 & \PsiNusselt'(y) \\ 0 & 1 & P_0'(y) \\ 0 & 0 & 1
    \end{pmatrix}.
\end{equation}
This change of variables is motivated by the operation of the chain rule in linearising $\F$, and is sometimes referred to as the \emph{T-isomorphism}; see~\cite{ehrnstrom2012steady, varholm2020global, wahlen2024large}. The composition $\mathscr{L} \mathcal{Q}$ is then explicitly
\begin{equation}\label{eq:lin_op}
    \mathscr{L} \mathcal{Q} \begin{pmatrix} \varphi \\ q \\ \zeta \end{pmatrix}
    = \begin{pmatrix}
        \left(-  \Delta \d_y + \Rey \PsiNusselt' \d_{xy} - \Rey \PsiNusselt'' \d_x \right) \varphi + q_x \\
        \left( \Delta \d_x - \Rey \PsiNusselt' \d_{xx}\right) \varphi + q_y \\
        \varphi + \PsiNusselt' \zeta \big\vert_{\Top} \\
        \left(\d_{xx} - \d_{yy}\right)\varphi - \PsiNusselt''' \zeta \big\vert_{\Top} \\
        2  \varphi_{xy} + q + \left(P_0' + \sigma \d_{xx} \right) \zeta \big\vert_{\Top} \\
        \varphi_y \vert_\Bottom
    \end{pmatrix}.
\end{equation}
The advantage of this operator over $\mathscr{L}$ in~\eqref{eq: triv_lin} is that $\zeta$ only appears in the components evaluated on the top boundary $\Top$.

\subsection{Abstract multi-parameter bifurcation theory}
This last preliminary subsection recalls some abstract theory which will be used to prove our bifurcation results Theorems~\ref{thm: general_local} and~\ref{thm: general_global}.

Firstly, we recall a result by Hale~\cite{hale1978bifurcation} which generalises the Crandall--Rabinowitz theorem~\cite{crandall1971bifurcation} to linearised operators having a one-dimensional kernel and range with codimension $n \in \mathbb{N}$. This result reduces to Crandall and Rabinowitz's when $n=1$; we will apply it with $n=2$. We state an analytic version (see also~\cite[Theorem 1.1]{lopez1989multiparameter}) in convenient notation.

\begin{theorem}[\cite{hale1978bifurcation}]\label{bif_lemma}
    Let $\F \map{\U}{\Y}$ be an analytic mapping, where $\U \subset \X \times \Real^n$ is open, $n \in \mathbb{N}$ and $\X$ and $\Y$ are Banach spaces. Suppose the following hold:
    \begin{enumerate}[label=\rm(\roman*)]
        \item\label{H1} $\F(0, \lambda) = 0$ for all $\lambda$;
        \item\label{H2} For some $\lambda_0 \in \Real^n$, the operator $\F_u(0, \lambda_0)$ is Fredholm and there exists some $\xi_0 \in \X$ such that $\ker \F_u(0, \lambda_0) = \operatorname{span} \{\xi_0\}$ and $\operatorname{codim} \operatorname{ran} \F_u(0, \lambda_0) = n$;
        \item\label{H3} For all nonzero $\mu \in \Real^n$, 
        \begin{equation*}
            D_{u \lambda}\F(0, \lambda_0) [\xi_0, \mu] \notin \operatorname{ran} \F_u(0, \lambda_0).
        \end{equation*}
    \end{enumerate}
    Then there exist $\eps > 0$ and a nontrivial curve of solutions,
    \begin{equation*}
        \mathscr{C}_{\mathrm{loc}} := \{\left(U(s), \Lambda(s)\right) : s \in (-\eps, \eps) \} \subset \F^{-1}(0),
    \end{equation*}
    parameterised by analytic functions $U \colon (-\eps, \eps) \to \U$ and $\Lambda \colon (-\eps, \eps) \to \Real^n$ satisfying $U(0) = 0$, $U'(0) = \xi_0$ and $\Lambda(0) = \lambda_0$. Furthermore, there exists an open neighbourhood $\mathcal{N} \subset \U$ such that $(0, \lambda_0) \in \mathcal{N}$ and
    \begin{equation}\label{eq: bif_uniqueness}
        \mathscr{C}_{\mathrm{loc}} = \{(u,\lambda) \in \mathcal{N} : \F(u,\lambda) = 0, u \neq 0 \}.
    \end{equation}
\end{theorem}
Here we have written the transversality-type condition~\ref{H3} in a form most analogous to the transversality condition of~\cite{crandall1971bifurcation}. This hypothesis is equivalent to the third condition in the definition of a `simple eigenvalue' in~\cite{hale1978bifurcation}. Indeed,~\ref{H3} implies that there are $n$ linearly independent vectors $\F_{u \lambda_i}(0, \lambda_0) \xi_0$ in the complement of the range of $\F_u(0, \lambda_0)$ in $\Y$, which by~\ref{H2} has dimension $n$, implying Hale's third condition. 

We note that it would also be possible to use Kielh\"ofer's multiparameter result~\cite[Theorem I.19.6]{kielhofer2012bifurcation}, by considering our operator $\F_u(u_0, \theta, \gamma)$ without the phase constraint as an index zero operator with a two-dimensional kernel and range of codimension two. 

Lastly, we state a multiparameter adaptation of Buffoni and Toland's analytic global bifurcation result~\cite[Theorem 9.1.1]{buffoni2003analytic}. Following~\cite{haziot2023large}, we require only that the linearised operator is \emph{semi-Fredholm} with closed range and finite-dimensional kernel (this is equivalent to the condition of local properness). The additional alternative of loss of compactness replaces the hypothesis that all bounded closed subsets of solutions in $\U$ are compact (see also~\cite{chen2018existence}). We sketch a proof in Appendix~\ref{section: extra_abstract_theory}.

\begin{theorem}[\cite{buffoni2003analytic}]\label{theorem: gbt}
    In the setting of Theorem~\ref{bif_lemma}, suppose that in addition, for all $(u,\lambda) \in \U$, the operator $\F_u(u,\lambda)$ is semi-Fredholm with closed range and finite-dimensional kernel. Then there exists a curve $\mathscr{C}$ which extends $\mathscr{C}_{\mathrm{loc}}$, parameterised as 
    \begin{equation*}
        \mathscr{C} = \{(U(s), \Lambda(s)) \colon s \in \Real \} \subset \F^{-1}(0),
    \end{equation*}
    where $(U, \Lambda) \colon \Real \to \X \times \Real^n$ is continuous. Near each point along $\mathscr{C}$, we can reparameterise $\mathscr{C}$ so that $s \mapsto (U(s), \Lambda(s))$ is real analytic. Furthermore, one of the following alternatives holds.
    \begin{enumerate}[label=\rm(\alph*)]
        \item Along each directed branch as $s \to \pm \infty$, either
        \begin{enumerate}[label=\rm(\roman*)]
            \item\label{abstract_blowup_alt} \textup{(Blow-up)} The quantity
            \begin{equation}\label{eq: blowup_quantity}
                N(s) := \lVert U(s) \rVert_\X + \lvert \Lambda(s) \rvert + \frac{1}{\operatorname{dist}\left((U(s), \Lambda(s)), \d \U\right)} \longrightarrow \infty; ~\text{or}
            \end{equation} 
            \item\label{abstract_comp_alt} \textup{(Loss of compactness)} There exists a sequence $s_n \to \pm \infty$ such that $\sup_n N(s_n) < \infty$ but $\{U(s_n)\}$ has no subsequences converging in $\U$. 
        \end{enumerate}\label{blowup_comp_alt}
        \item\label{abstract_loop_alt} \textup{(Closed loop)} There exists some $T > 0$ such that $(U(s+T), \Lambda(s+T)) = (U(s), \Lambda(s))$ for all $s \in \Real$. 
    \end{enumerate}
\end{theorem}

\section{Elliptic theory}\label{section: elliptic}
In this section, we prove that the operator for the linearisation around a general solution can be extended to an elliptic operator in the sense of Agmon--Douglis--Nirenberg~\cite{agmon1964estimates}, with a corresponding Schauder estimate. Consequently, the operator is semi-Fredholm with closed range and finite-dimensional kernel, as required by Theorem~\ref{theorem: gbt}. A homotopy argument then allows us to compute the Fredholm index of the operator linearised around the Nusselt solution to be $-1$, as required by Theorem~\ref{bif_lemma}.

\subsection{Schauder estimate}
Our objective in this subsection is to prove a global Schauder estimate for the operator $\F_u(u,\theta, \gamma)$ linearised around any $(u,\theta, \gamma) \in \U$, for $\U$ defined in~\eqref{eq: open_set}. As an immediate corollary, $\F_u(u,\theta, \gamma)$ is therefore \emph{locally proper}, that is, the intersection of the pre-image of any compact set with any closed and bounded set in $\U$ is compact. Equivalently, $\F_u(u, \theta, \gamma)$ is semi-Fredholm with closed range and finite-dimensional kernel. 

To this end, we study a natural extension of $\F_u(u,\theta, \gamma)$, which we write in the form $\left(\mathcal{A}, \mathcal{B}_1, \mathcal{B}_0 \right)$, where operators $\mathcal{A},~\mathcal{B}_1$ and $\mathcal{B}_0$ are operators on the interior, top boundary component, and lower boundary component respectively. We verify that on each portion of the boundary, the corresponding operator $(\mathcal{A}, \mathcal{B}_i)$ satisfies the ``Shapiro--Lopatinsky'' condition, also known as the ``covering'' or ``complementing'' condition. This is an algebraic condition on the Fourier coefficients of the symbol of the operator, and implies local Schauder estimates near each boundary. Together with interior Schauder estimates from interior ellipticity, this yields a global Schauder estimate for $\left(\mathcal{A}, \mathcal{B}_1, \mathcal{B}_0 \right)$ and thus for $\F_u(u,\theta, \gamma)$.  In Appendix~\ref{section: ADN_theory}, we provide a terse review of the relevant theory for elliptic systems. See also~\cite[Chapters 9 and 10]{wloka1995boundary} and~\cite[Chapter 2]{volpert2011elliptic} for a more detailed introduction. 

\begin{prop}\label{Schauder}
    Let $k > 0$, $\Rey \geq 0$, and $(u, \theta, \gamma) = (\psi, p, \eta, \theta, \gamma) \in \U$. Suppose $\delta > 0$ is such that
    \begin{equation}\label{eq: hypothesis_for_schauder}
        \inf_\Top \lvert \nabla \eta \rvert > \delta, \quad \sup_\Top \lvert \nabla \eta \rvert < \dfrac{1}{\delta}, \quad \inf_\Top \lvert \psi_y \rvert + \sigma > \delta.
    \end{equation}
    Then the operator $\F_u(u, \theta, \gamma) \colon \X \to \Y$ enjoys the following a priori estimate:
    \begin{equation}\label{eq: system_schauder}
        \lVert \dot{u} \rVert_{\X} \leq C \big( \lVert \F_u(u, \theta, \gamma) \dot{u} \rVert_\Y + \lVert \dot{u} \rVert_{L^1(\Rect) \times L^1(\Rect) \times L^1(\Rect)} \big),
    \end{equation}
    where $C > 0$ depends only on $\delta$ and $\lVert u \rVert_{\X}$.
\end{prop}
\begin{proof}
    First note that~\eqref{eq: hypothesis_for_schauder} implies lower and upper bounds for $\lvert \nabla \eta \rvert$ over $\overline{\Rect}$ in terms only of $\delta$, 
    \begin{equation}\label{eq: infsupRect}
        \inf_{\overline{\Rect}} \lvert \nabla \eta \rvert > C(\delta), \quad \sup_{\overline{\Rect}} \lvert \nabla \eta \rvert < C(\delta^{-1}),
    \end{equation}
    by standard maximum principle arguments.

    We compute the linearisation of $\F$ around $(u, \theta, \gamma) \in \U$, making use of the homogeneous conditions $\Delta \eta = 0$ and $\psi \vert_\Bottom = \eta \vert_\Bottom = 0$ in the definition~\eqref{eq: spaces} of $\X$,
    \begin{equation}\label{eq: gen_lin}
        \begin{aligned}
            \F_{1u} (u, \theta, \gamma) (\dot{\psi}, \dot{p}, \dot{\eta}) &= - \lvert \nabla \eta \rvert^{-4} \nabla \eta \cdot \nabla \Delta \dot{\psi} + D_1 \dot{\psi} + \lvert \nabla \eta \rvert^{-2} \nabla \eta \cdot \nabla^{\perp} \dot{p} + D_2 \dot{\eta}, \\
            \F_{2u} (u, \theta, \gamma) (\dot{\psi}, \dot{p}, \dot{\eta}) &= \lvert \nabla \eta \rvert^{-4} \nabla \eta \cdot \nabla^{\perp} \Delta \dot{\psi} + D_3 \dot{\psi} + \lvert \nabla \eta \rvert^{-2} \nabla \eta \cdot \nabla \dot{p} + D_4 \dot{\eta}, \\
            \F_{3u} (u, \theta, \gamma) (\dot{\psi}, \dot{p}, \dot{\eta}) &= \dot{\psi} \vert_\Top, \\
            \F_{4u} (u, \theta, \gamma) (\dot{\psi}, \dot{p}, \dot{\eta}) &= \dfrac{1}{\lvert \nabla \eta \rvert^2} \eta_y (\dot{\psi}_{xx} - \dot{\psi}_{yy}) - 2 \eta_x \dot{\psi}_{xy} + D_5 \dot{\psi} - \eta_x \dot{p} + \dfrac{\sigma \eta_x}{\lvert \nabla \eta \rvert^3} \left(\eta_x \dot{\eta}_{xy} - \eta_y \dot{\eta}_{xx}\right) \\ & \quad + \dfrac{2 \psi_y}{\lvert \nabla \eta \rvert^{4}} \left((\eta_x^2-\eta_y^2) \dot{\eta}_{xx} + 2 \eta_x \eta_y \dot{\eta}_{xy}\right) + D_6 \dot{\eta} \vert_\Top, \\
            \F_{5u} (u, \theta, \gamma) (\dot{\psi}, \dot{p}, \dot{\eta}) &= \dfrac{1}{\lvert \nabla \eta \rvert^2} \eta_x (\dot{\psi}_{xx} - \dot{\psi}_{yy}) + 2 \eta_y \dot{\psi}_{xy} + D_7 \dot{\psi} + \eta_y \dot{p} + \dfrac{\sigma \eta_y}{\lvert \nabla \eta \rvert^3} \left(\eta_y \dot{\eta}_{xx} - \eta_x \dot{\eta}_{xy}\right) \\ & \quad + \dfrac{2 \psi_y}{\lvert \nabla \eta \rvert^{4}} \left((\eta_x^2-\eta_y^2) \dot{\eta}_{xy} - 2 \eta_x \eta_y \dot{\eta}_{xx}\right) + D_8 \dot{\eta} \vert_\Top, \\
            \F_{6u} (u, \theta, \gamma) (\dot{\psi}, \dot{p}, \dot{\eta}) &= \eta_y \dot{\psi}_y + \psi_y \dot{\eta}_y \vert_\Bottom.
        \end{aligned}
    \end{equation}
    Here the operators $D_1, \ldots, D_4$ are second-order differential operators with $C^{1,\alpha}(\Rect)$ coefficients and $D_5, \ldots, D_8$ are first-order differential operators with $C^{2,\alpha}(\Top)$ coefficients, all of whose norms are controlled by $(\inf_{\overline{\Rect}} \lvert \nabla \eta \rvert)^{-1}$ and $\lVert u \rVert_\X$, and their exact form will not be required for our calculations.

    Let us define a natural extension of the general linearised operator $\F_u(u, \theta, \gamma) \colon \X \longrightarrow \Y$,
    \begin{equation}\label{eq: AB1B0}
        \begin{aligned}
            (\mathcal{A}, \mathcal{B}_1, \mathcal{B}_0) \colon C^{4, \alpha}\left(\Rect\right)\times C^{2, \alpha}\left(\Rect\right) \times C^{4, \alpha}\left(\Rect\right) \longrightarrow ~&C^{1, \alpha}\left(\Rect\right)\times C^{1, \alpha}(\Rect) \times C^{2, \alpha}\left(\Rect\right) \times \\ 
            &C^{4, \alpha} \left(\Top\right)\times C^{2, \alpha}\left(\Top\right)\times C^{2, \alpha}\left(\Top\right) \times \\
            &C^{4, \alpha} \left(\Bottom\right)\times C^{3, \alpha}\left(\Bottom\right)\times C^{4, \alpha}\left(\Bottom\right)
        \end{aligned}
    \end{equation}
    by
    \begin{equation*}
        \mathcal{A} := \begin{pmatrix}
            - \lvert \nabla \eta \rvert^{-4} \nabla \eta \cdot \nabla \Delta + D_1 & \lvert \nabla \eta \rvert^{-2} \nabla \eta \cdot \nabla^{\perp} & D_2 \\
            \lvert \nabla \eta \rvert^{-4} \nabla \eta \cdot \nabla^{\perp} \Delta + D_3 & \lvert \nabla \eta \rvert^{-2} \nabla \eta \cdot \nabla & D_4 \\
            0 & 0 & \Delta
        \end{pmatrix},
    \end{equation*}
    \begin{multline*}
        \mathcal{B}_1 := \begin{pmatrix} 
            1 & 0 & 0 \\
            \dfrac{\eta_y (\d_{xx} - \d_{yy}) - 2 \eta_x \d_{xy}}{\lvert \nabla \eta \rvert^2} & - \eta_x & \dfrac{\sigma \eta_x}{\lvert \nabla \eta \rvert^3} \left(\eta_x \d_{xy} - \eta_y \d_{xx}\right) + \dfrac{2 \psi_y}{\lvert \nabla \eta \rvert^{4}} \left((\eta_x^2-\eta_y^2) \d_{xx} + 2 \eta_x \eta_y \d_{xy}\right)\\
            \dfrac{\eta_x (\d_{xx} - \d_{yy}) + 2 \eta_y \d_{xy}}{\lvert \nabla \eta \rvert^2} & \eta_y & \dfrac{\sigma \eta_y}{\lvert \nabla \eta \rvert^3} \left(\eta_y \d_{xx} - \eta_x \d_{xy}\right) + \dfrac{2 \psi_y}{ \lvert \nabla \eta \rvert^{4}} \left((\eta_x^2-\eta_y^2) \d_{xy} - 2 \eta_x \eta_y \d_{xx}\right)
    \end{pmatrix} \\ + \begin{pmatrix}
        0 & 0 & 0 \\ 
        D_5 & 0 & D_6 \\
        D_7 & 0 & D_8
    \end{pmatrix},
    \end{multline*}
    and
    \begin{equation*}
        \mathcal{B}_0 := \begin{pmatrix}
            1 & 0 & 0 \\
            \dfrac{1}{\eta_y} \d_y & 0 & \psi_y \d_y \\
            0 & 0 & 1
        \end{pmatrix}.
    \end{equation*}
    The choices of the integer orders of the spaces in~\eqref{eq: AB1B0} will be motivated below (see~\eqref{eq: DN_ordersA} and~\eqref{eq: DN_ordersB}). This defines a natural extension of $\F_u(u, \theta, \gamma)$ in the sense that, for $(\dot{\psi}, \dot{p}, \dot{\eta}) \in \tilde{\X}$ and $(f_1, f_2, h_1, h_2, h_3, h_4) \in \Y$, if
    \begin{equation*}
        \F_u(u, \theta, \gamma) (\dot{\psi}, \dot{p}, \dot{\eta})^T = (f_1, f_2, h_1, h_2, h_3, h_4)^T,
    \end{equation*}
    then
    \begin{equation*}
        (\mathcal{A}, \mathcal{B}_1, \mathcal{B}_0) (\dot{\psi}, \dot{p}, \dot{\eta})^T = (f_1, f_2, 0, h_1, h_2, h_3, 0, h_4, 0)^T.
    \end{equation*}
    Therefore, it suffices to establish a Schauder estimate on $\overline{\Rect}$ for the extended operator $(\mathcal{A}, \mathcal{B}_1, \mathcal{B}_0)$. We do this by following the standard recipe from~\cite{agmon1964estimates} for elliptic systems.

    The first step is to formally replace $\d_x, \d_y$ with $ik, il$ for some Fourier variables $k,l \in \Real$, to obtain matrix polynomials $\mathscr{A}(k,l), \mathscr{B}_1(k,l)$ and $\mathscr{B}_0(k,l)$ corresponding to $\mathcal{A}, \mathcal{B}_1$ and $\mathcal{B}_0$. Suppressing dependence on $(x,y)$ for readability, these are given by
    \begin{equation*}
        \mathscr{A}(k,l) = 
        \begin{pmatrix}
            i \lvert \nabla \eta \rvert^{-4} \left(k^2 + l^2\right) \left(\eta_x k + \eta_y l\right) + \tilde{D}_1 & i \lvert \nabla \eta \rvert^{-2} \left(- \eta_x l + \eta_y k\right) & \tilde{D}_2 \\
            i \lvert \nabla \eta \rvert^{-4} \left(k^2 + l^2\right) \left(\eta_x l - \eta_y k\right) + \tilde{D}_3 & i \lvert \nabla \eta \rvert^{-2} \left(\eta_x k + \eta_y l\right) & \tilde{D}_4 \\
            0 & 0 & - (k^2 + l^2) 
        \end{pmatrix},
    \end{equation*}
    \begin{multline*}
        \mathscr{B}_1(k,l) =
        \begin{pmatrix}
            1 & 0 & 0 \\
            \dfrac{\left(2 \eta_x k l - \eta_y k^2 + \eta_y l^2\right)}{\lvert \nabla \eta \rvert^2} +  & - \eta_x & \dfrac{\eta_x k \sigma \left(- \eta_x l + \eta_y k\right)}{\lvert \nabla \eta \rvert^3} + \dfrac{2 k \psi_y \left(- \eta_x^2 k - 2 \eta_x \eta_y l + \eta_y^2 k\right)}{\lvert \nabla \eta \rvert^4} \\
            \dfrac{\left(- \eta_x k^2 + \eta_x l^2 - 2 \eta_y k l\right)}{\lvert \nabla \eta \rvert^2} & \eta_y & \dfrac{\eta_y k \sigma \left(\eta_x l - \eta_y k\right)}{\lvert \nabla \eta \rvert^3} + \dfrac{2 k \psi_y \left(- \eta_x^2 l + 2 \eta_x \eta_y k + \eta_y^2 l\right)}{\lvert \nabla \eta \rvert^4}
        \end{pmatrix} \\ + \begin{pmatrix}
            0 & 0 & 0 \\ 
            \tilde{D}_9 & 0 & \tilde{D}_{10} \\
            \tilde{D}_{11} & 0 & \title{D}_{12}
        \end{pmatrix},
    \end{multline*}
    and
    \begin{equation*}
        \mathscr{B}_0(k,l) = 
        \begin{pmatrix}
            1 & 0 & 0\\
            \dfrac{il}{\eta_y} & 0 & i\psi_y l\\
            0 & 0 & 1
        \end{pmatrix},
    \end{equation*}
    where $\tilde{D}_i(k,l)$ are polynomials in $k,l$ of lower degree. 
    
    We verify that $\mathcal{A}$ is \emph{properly elliptic on} $\overline{\Rect}$  in the sense of Douglis--Nirenberg. Consider $\mathscr{A}$ evaluated at an interior point $(x,y) \in \overline{\Rect}$. As polynomials in $k,l$, the entries of $\mathscr{A}$ have degree
    \begin{equation}\label{eq: A_matrix_orders}
        (\alpha_{ij}) = \begin{pmatrix}
            3 & 1 & 3 \\
            3 & 1 & 3 \\
            -\infty & -\infty & 2
        \end{pmatrix},
    \end{equation}
    and we select \emph{Douglis--Nirenberg numbers} (defined in~\eqref{eq: DN_numbers}) for $(\alpha_{ij})$, choosing $(t_1, t_2, t_3)= (3, 1, 3)$ and $(s_1, s_2, s_3) = (0, 0, -1)$, so that the matrix
    \begin{equation}\label{eq: DN_ordersA}
        (s_i + t_j) = \begin{pmatrix}
            3 & 1 & 3 \\
            3 & 1 & 3 \\
            2 & 0 & 2
        \end{pmatrix}
    \end{equation}
    has entries greater than or equal to the respective entries of~\eqref{eq: A_matrix_orders}. Thus, $\mathcal{A}$ is a differential operator of the form
    \begin{equation*}
        \mathcal{A}_{ij}(x, \d) = \sum_{\lvert \rho \rvert \leq s_i + t_j} a_{ij, \rho}(x) \d^\rho \colon C^{1+t_j, \alpha}(\Rect) \to C^{1-s_i, \alpha}(\Rect),
    \end{equation*}
    with coefficients $a_{ij, \rho} \in C^{1-s_i, \alpha}(\overline{\Rect})$ whose norms are controlled uniformly from above by $(\inf_{\overline{\Rect}} \lvert \nabla \eta \rvert)^{-1}$ and $\lVert u \rVert_\X$.
    
    From this set of Douglis--Nirenberg numbers we obtain a \emph{principal part} $\pi \mathscr{A}$ of $\mathscr{A}$, that is, the terms that are of order $s_i + t_j$, 
    \begin{equation}\label{eq: principalA}
        \pi \mathscr{A}(k, l) = \begin{pmatrix}
            i \lvert \nabla \eta \rvert^{-4} \left(k^2 + l^2\right) \left(\eta_x k + \eta_y l\right) & i \lvert \nabla \eta \rvert^{-2} \left(- \eta_x l + \eta_y k\right) & 0\\
            i \lvert \nabla \eta \rvert^{-4} \left(k^2 + l^2\right) \left(\eta_x l - \eta_y k\right) & i \lvert \nabla \eta \rvert^{-2} \left(\eta_x k + \eta_y l\right) & 0\\
            0 & 0 & - (k^2 + l^2) 
        \end{pmatrix}.
    \end{equation}
    The determinant of this principal part gives the \emph{characteristic polynomial} for $\mathscr{A}$, 
    \begin{equation}\label{eq: char_poly}
        \det \pi \mathscr{A}(k,l) = \lvert \nabla \eta \rvert^{-4} (k^2+l^2)^3.
    \end{equation}
    Given $(k,l) \in \Real^2$ such that $k^2 + l^2 = 1$,~\eqref{eq: char_poly} is bounded below in $\overline{\Rect}$ by a constant depending only on $\sup_{\overline{\Rect}} \lvert \nabla \eta \rvert$ and bounded above by a constant depending on $\inf_{\overline{\Rect}} \lvert \nabla \eta \rvert$. Thus, by~\eqref{eq: infsupRect}, $\mathcal{A}$ is uniformly elliptic in $\Rect$ with an ellipticity constant depending only on $\delta$. This ellipticity is also proper, since for fixed $k \geq 0$, $\det \pi \mathscr{A}(k,l)$ has the same number of roots in the upper half-plane $\{\operatorname{Im} l > 0 \}$ as in the lower half-plane. Define $\rho_+(l) := (l-ik)^3$, the factor of $\det \pi \mathscr{A}(k,l)$ which contains all the roots inside the positive upper half-plane.

    We next verify the \emph{Shapiro--Lopatinsky} condition at each component of the boundary of $\Rect$. Here, we fix $k > 0$ and view $\mathscr{B}_i$ and $\mathscr{A}$ as functions of $l$. At each boundary component, this condition is equivalent to the matrix polynomial $\pi \mathscr{B}_i(l) \adj(\pi \mathscr{A})(l)$ being full rank modulo $\rho_+(l)$, where $\adj M = \det M (M)^{-1}$ denotes the adjugate of a matrix $M$. 
    
    Considering the top boundary operator at a point $(x,1) \in T$, the matrix of orders of coefficients of $\mathscr{B}_1$ is given by
    \begin{equation*}
        (\beta_{ij}) = \begin{pmatrix}
            0 & -\infty & -\infty \\
            2 & -\infty & 2 \\
            2 & 0 & 2
        \end{pmatrix}.
    \end{equation*}
    Defining $m_i = \max_j (\beta_{ij} - t_j)$, we have $(m_1, m_2, m_3) = (-3, -1, -1)$, and similarly to~\eqref{eq: DN_ordersA}, the matrix
    \begin{equation}\label{eq: DN_ordersB}
        (m_i + t_j) = \begin{pmatrix}
            0 & -2 & 0 \\
            2 & 0 & 2 \\
            2 & 0 & 2
        \end{pmatrix}
    \end{equation}
    determines the principal part of $\mathscr{B}_1$,
    \begin{equation*}
        \pi \mathscr{B}_1(l) = \begin{pmatrix}
            \dfrac{\left(2 \eta_x k l - \eta_y k^2 + \eta_y l^2\right)}{\lvert \nabla \eta \rvert^2} & - \eta_x & \dfrac{\eta_x k \sigma \left(- \eta_x l + \eta_y k\right)}{\lvert \nabla \eta \rvert^3} + \dfrac{2 k \psi_y \left(- \eta_x^2 k - 2 \eta_x \eta_y l + \eta_y^2 k\right)}{\lvert \nabla \eta \rvert^4}\\
            \dfrac{\left(- \eta_x k^2 + \eta_x l^2 - 2 \eta_y k l\right)}{\lvert \nabla \eta \rvert^2} & \eta_y & \dfrac{\eta_y k \sigma \left(\eta_x l - \eta_y k\right)}{\lvert \nabla \eta \rvert^3} + \dfrac{2 k \psi_y \left(- \eta_x^2 l + 2 \eta_x \eta_y k + \eta_y^2 l\right)}{\lvert \nabla \eta \rvert^4}
        \end{pmatrix}.
    \end{equation*}
    Here, we are seeing $\mathcal{B}_1$ as an operator 
    \begin{equation*}
        (\mathcal{B}_1)_{kj}((x,1), \d) = \sum_{\lvert \sigma \rvert \leq m_k + t_j} b_{kj, \sigma}(x,1) \d^\sigma \colon C^{1+ t_j, \alpha}(\Top) \to C^{1-m_k, \alpha}(\Top),
    \end{equation*}
    with coefficients $b_{kj, \sigma} \in C^{1-m_k, \alpha}(\Top)$ whose norms are controlled uniformly from above by $(\inf_{\Top} \lvert \nabla \eta \rvert)^{-1}$ and $\lVert u \rVert_\X$.

    It is convenient to multiply $\pi \mathscr{B}_1$ by an invertible matrix which is independent of $l$, defining a new matrix $\tilde{\mathscr{B}}_1$ by
    \begin{equation}\label{eq: principal_B1}
        \tilde{\mathscr{B}}_1 =
        \begin{pmatrix}
            1 & 0 & 0 \\
            0 & \eta_y & \eta_x \\
            0 & - \eta_x & \eta_y
        \end{pmatrix} \; \pi \mathscr{B}_1 = 
        \begin{pmatrix}
            1 & 0 & 0\\
            \left(l^2 - k^2\right) & 0 & \dfrac{2 k \psi_y}{\lvert \nabla \eta \rvert^2} \left(\eta_y k - \eta_x l\right)\\
            - 2 k l & \lvert \nabla \eta \rvert^2 & \dfrac{2 k \psi_y}{\lvert \nabla \eta \rvert^2} \left(\eta_x k + \eta_y l\right) + \dfrac{\sigma k}{\lvert \nabla \eta|} \left(\eta_x l - \eta_y k\right)
        \end{pmatrix}.
    \end{equation}
    We compute the adjugate of $\pi \mathscr{A}$ at $(x,1) \in T$,
    \begin{equation*}
        \adj(\pi \mathscr{A})(l) =
        \begin{pmatrix}
            - \dfrac{i \left(k^2 + l^2\right) \left(\eta_x k + \eta_y l\right)}{\lvert \nabla \eta \rvert^2} & \dfrac{i \left(k^2 + l^2\right) \left(- \eta_x l + \eta_y k\right)}{\lvert \nabla \eta \rvert^2} & 0\\
            - \dfrac{i \left(k^2 + l^2\right)^2 \left(- \eta_x l + \eta_y k\right)}{\lvert \nabla \eta \rvert^4} & - \dfrac{i \left(k^2 + l^2\right)^2 \left(\eta_x k + \eta_y l\right)}{\lvert \nabla \eta \rvert^4} & 0\\
            0 & 0 & - \dfrac{\left(k^2 + l^2\right)^2}{\lvert \nabla \eta \rvert^4}
        \end{pmatrix},
    \end{equation*}
    so that
    \begin{multline*}
        M_1(l) := \tilde{\mathscr{B}}_1 \adj(\pi \mathscr{A})(l) 
        = \dfrac{i}{\lvert \nabla \eta \rvert^2} \left(k^2 + l^2\right) \times \\
        \begin{pmatrix}
            - \left(\eta_x k + \eta_y l\right) & \eta_y k - \eta_x l & 0\\
            \left(k^2 - l^2\right) \left(\eta_x k + \eta_y l\right) & \left(k^2 - l^2\right) \left(\eta_x l - \eta_y k\right) & \dfrac{2 ik \psi_y \left(k^2 + l^2\right) \left(\eta_y k - \eta_x l\right)}{\lvert \nabla \eta \rvert^4}\\
            \eta_x l (3k^2 + l^2) - \eta_y k(k^2 - l^2) & \eta_x k (l^2 - k^2) - \eta_y l (3 k^2 + l^2) & ik \left(k^2 + l^2\right) \left(\dfrac{2 \psi_y (\eta_x k + \eta_y l)}{\lvert \nabla \eta \rvert^4} - \dfrac{\sigma (\eta_y k - \eta_x l)}{\lvert \nabla \eta \rvert^3}\right)
        \end{pmatrix}.
    \end{multline*}
    
    Now the rows of $M_1(l)$ are linearly independent modulo $\rho_+(l) = (l-ik)^3$ if and only if the block matrix $\begin{pmatrix} M_1(ik) & M_1'(ik) & M_1''(ik) \end{pmatrix}$ has full rank. To see this, consider the Taylor expansion of $M_1(l)$ around the point $l=ik$,
    \begin{equation*}
        M_1(l) = \sum_{j=0}^{3} \dfrac{1}{j!}M_1^{(j)}(ik)(l - ik)^j.
    \end{equation*}
    We then can write the remainder modulo $(l - ik)^3$ of $M_1$ as
    \begin{equation*}
        M_1(l) \equiv \sum_{j=0}^{2} \dfrac{1}{j!}M_1^{(j)}(ik)(l - ik)^j \mod (l - ik)^3.
    \end{equation*}
    Consequently, the rows of $M_1(l)$ are linearly dependent modulo $\rho_+(l)$ if and only if each of the $3 \times 3$ minor determinants of $\begin{pmatrix} M_1(ik) & M_1'(ik) & M_1''(ik) \end{pmatrix}$ vanish.

    In fact, we need only consider
    \begin{align*}
        M_1''(ik) = \begin{pmatrix}
            - \dfrac{2 i k \left(\eta_x + 3 i \eta_y\right)}{\lvert \nabla \eta \rvert^2} & \dfrac{2 k \left(3 \eta_x + i \eta_y\right)}{\lvert \nabla \eta \rvert^2} & 0\\
            \dfrac{4 i k^{3} \left(3 \eta_x + 5 i \eta_y\right)}{\lvert \nabla \eta \rvert^2} & - \dfrac{4 k^{3} \left(5 \eta_x + 3 i \eta_y\right)}{\lvert \nabla \eta \rvert^2} & - \dfrac{16 i k^{4} \psi_y \left(\eta_x + i \eta_y\right)}{\lvert \nabla \eta \rvert^6}\\
            - \dfrac{4 k^{3} \left(\eta_x + 3 i \eta_y\right)}{\lvert \nabla \eta \rvert^2} & - \dfrac{4 i k^{3} \left(3 \eta_x + i \eta_y\right)}{\lvert \nabla \eta \rvert^2} & \dfrac{8 i k^{4} \left(\eta_x + i \eta_y\right) \left(\sigma \lvert \nabla \eta \rvert- 2 i \psi_y \right)}{\lvert \nabla \eta \rvert^3}
        \end{pmatrix},
    \end{align*}
    which has
    \begin{equation}\label{eq: det_top}
        \lvert \det M_1''(ik) \rvert^2 = 65536 k^{16} \lvert \nabla \eta \rvert^{-14} \left(\sigma^2 \lvert \nabla \eta \rvert^2 + 4 \lvert \psi_y \rvert^2\right).
    \end{equation}
    The infimum of the right-hand side of~\eqref{eq: det_top} over $\Top$ is bounded away from zero uniformly by a constant dependent on $\sup_\Top \lvert \nabla \eta \rvert$ and $\inf_\Top \lvert \psi_y \rvert + \sigma$. Therefore, by~\eqref{eq: hypothesis_for_schauder} the Shapiro--Lopatinsky condition holds for $(\mathcal{A}, \mathcal{B}_1)$ uniformly on $\Top$ with a constant depending on $\delta$, and by Theorem~\ref{adn_schauder}, we have a Schauder estimate up to the boundary at the top.

    The computation on the lower boundary $\Bottom$ is similar but less complicated than on $\Top$. We consider a point $(x,0) \in \Bottom$ and again fix $k > 0$, but since $e_y$ is inward-pointing at the bed, we evaluate all polynomials in $l$ at $-l$ instead. We pick Douglis--Nirenberg numbers $(m_1, m_2, m_3) = (-3, -2, -3)$ for $\mathcal{B}_0$, and find principal part
    \begin{equation}\label{eq: principal_B0}
        \pi \mathscr{B}_0(-l) = \begin{pmatrix}
            1 & 0 & 0\\
            \dfrac{-il}{\eta_y} & 0 & -i\psi_y l\\
            0 & 0 & 1
        \end{pmatrix},
    \end{equation}
    seeing $\mathcal{B}_0$ as an operator
    \begin{equation*}
        (\mathcal{B}_0)_{kj}((x,1), \d) = \sum_{\lvert \sigma \rvert \leq m_k + t_j} b_{kj, \sigma}(x,1) \d^\sigma \colon C^{1+ t_j, \alpha}(\Top) \to C^{1-m_k, \alpha}(\Top),
    \end{equation*}
    with coefficients $b_{kj, \sigma} \in C^{1-m_k, \alpha}(\Top)$ whose norms are controlled uniformly from above by $(\inf_{\Bottom} \lvert \nabla \eta \rvert)^{-1}$ and $\lVert u \rVert_\X$.

    We then calculate the matrix
    \begin{equation*}
        M_0(l) := \pi \mathscr{B}_0(-l) \adj(\pi \mathscr{A})(-l) =
        \begin{pmatrix}
            \dfrac{i l \left(k^2 + l^2\right)}{\eta_y} & \dfrac{i k \left(k^2 + l^2\right)}{\eta_y} & 0\\
            \dfrac{l^2 \left(k^2 + l^2\right)}{\eta_y^2} & \dfrac{k l \left(k^2 + l^2\right)}{\eta_y^2} & -\dfrac{i l \psi_y \left(k^2 + l^2\right)^2}{\eta_y^{4}}\\
            0 & 0 & - \dfrac{\left(k^2 + l^2\right)^2}{\eta_y^{4}}
        \end{pmatrix},
    \end{equation*}
    and consider $M_0''(ik)$, 
    \begin{equation*}
        M_0''(ik) = 
        \begin{pmatrix}
            -\dfrac{6 k}{\eta_{y}} & \dfrac{2 i k}{\eta_{y}} & 0\\
            - \dfrac{10 k^2}{\eta_{y}^2} & \dfrac{6 i k^2}{\eta_{y}^2} & -\dfrac{8 k^{3} \psi_{y}}{\eta_{y}^{4}}\\
            0 & 0 & \dfrac{8 k^2}{\eta_{y}^{4}}
        \end{pmatrix},
    \end{equation*}
    which has
    \begin{equation*}
        \lvert \det M_0''(ik) \rvert = 128 k^5 \lvert \eta_y \rvert^{-7},
    \end{equation*}
    again bounded away from zero by a constant depending only on $\sup_{\Bottom} \lvert \nabla \eta \rvert$. Thus again by~\eqref{eq: infsupRect}, $(\mathcal{A}, \mathcal{B}_0)$ satisfies the Shapiro--Lopatinsky condition uniformly on the bed with dependence only on $\delta$ and Theorem~\ref{adn_schauder} gives a Schauder estimate on the bed.

    Therefore, we have local a priori Schauder estimates up to the boundary for each corresponding operator $(\mathcal{A}, \mathcal{B}_i)$, as well as interior estimates from the ellipticity of $\mathcal{A}$. The constants in these estimates depend on the norms of the coefficients, the ellipticity constant and the Shapiro--Lopatinsky constants, which are all controlled by $\delta$ using~\eqref{eq: hypothesis_for_schauder} and~\eqref{eq: infsupRect}. Gluing these estimates together in the usual way, we obtain the global Schauder estimate~\eqref{eq: system_schauder}.
\end{proof}

\subsection{Fredholm index of linearisation around the Nusselt solution}
We can furthermore find the index of the linearisation $\mathscr{L} = \F_u(u_0(\theta, \gamma),\theta, \gamma)$ (see~\eqref{eq: triv_lin}) around the Nusselt solution $u_0(\theta, \gamma)$ defined in~\eqref{eq: u0} by homotoping to a simpler operator, using the continuity of the index under continuous deformation. Note that we only need to compute the index for the linearisation around the Nusselt solution, as the Fredholm property of the linearised operator preserves its index along the bifurcation curve as it is continued globally; see Theorem~\ref{theorem: gbt}.

\begin{prop}\label{Index_zero}
    The operator $\mathscr{L} \colon \X \to \Y$ defined in~\eqref{eq: triv_lin} has Fredholm index $-1$.
\end{prop}
\begin{proof}
    It is enough to show that $L := \mathscr{L} \mathcal{Q}$ given by~\eqref{eq:lin_op} has index $-1$, since $\mathcal{Q} \colon \X \to \X$ is invertible. Extending $L$ as in~\eqref{eq: AB1B0}, we define
    \begin{equation*}
        \tilde{\mathcal{A}} =   \begin{pmatrix}
            - \Delta \d_y + \Rey \PsiNusselt' \d_{xy} - \Rey \PsiNusselt'' \d_x & \d_x &  0 \\
            \Delta \d_x - \Rey \PsiNusselt' \d_{xx} & \d_y & 0 \\
            0  & 0 & \Delta
        \end{pmatrix}, 
    \end{equation*}
    \begin{equation*}
        \tilde{\mathcal{B}}_1 = \begin{pmatrix}
            1 & 0  & \PsiNusselt' \\
            \d_{xx}-\d_{yy} & 0  & -\PsiNusselt''' \\
            2 \d_{xy} & 1 & - 2 \cot \theta + \sigma \d_{xx}
        \end{pmatrix}, \qquad 
        \tilde{\mathcal{B}}_0 = \begin{pmatrix}
            1 & 0 & 0 \\
            \d_y & 0 & 0 \\
            0 & 0 & 1
        \end{pmatrix}.
    \end{equation*}
    Now, we extend the domain of $L$ to
    \begin{equation*}
        \tilde{\X} = \left\{(\dot{\psi}, \dot{p}, \dot{\eta}) \in C^{4, \alpha}(\Rect) \times C^{2, \alpha}(\Rect) \times C^{4, \alpha}(\Rect) ~:~ \Delta \eta = 0, ~ \psi|_\Bottom = \eta|_\Bottom = 0 \right\} \supset \X,
    \end{equation*}
    so that given $(\varphi, q, \zeta) \in \tilde{\X}$ and $(f_1, f_2, h_1, h_2, h_3, h_4) \in \Y$, the problem 
    \begin{equation*}
        L (\varphi, q, \zeta)^T = (f_1, f_2, h_1, h_2, h_3, h_4)^T
    \end{equation*}
    is equivalent to 
    \begin{equation*}
        \big(\tilde{\mathcal{A}}, \tilde{\mathcal{B}}_1, \tilde{\mathcal{B}}_0 \big) (\varphi, q, \zeta)^T = (f_1, f_2, 0, h_1, h_2, h_3, 0, h_4, 0)^T.
    \end{equation*}
    Consequently, the extended operator $L \colon \tilde{\X} \to \Y$ has the same index as $\big(\tilde{\mathcal{A}}, \tilde{\mathcal{B}}_1, \tilde{\mathcal{B}}_0 \big) \colon \tilde{\X} \to \tilde{\Y}$, where
    \begin{equation*}
        \tilde{\Y} = \{(f_1, f_2, 0, h_1, h_2, h_3, 0, h_4, 0) ~:~ (f_1, f_2, h_1, h_2, h_3, h_4) \in \Y \}.
    \end{equation*}

    Using that $\Delta \zeta = 0$ to simplify $\tilde{\mathcal{B}}_1$, these operators $\tilde{\mathcal{A}}$, $\tilde{\mathcal{B}}_1$, and $\tilde{\mathcal{B}}_0$ have matrix polynomials with Douglis--Nirenberg principal parts given by
    \begin{equation*}
        \pi \tilde{\mathscr{A}}(k,l) = \begin{pmatrix}
            i l (k^2 + l^2) & ik &  0 \\
            -i k (k^2 + l^2) & il & 0 \\
            0  & 0 & -(k^2 + l^2)
        \end{pmatrix}, 
    \end{equation*}
    \begin{equation*}
        \pi \tilde{\mathscr{B}}_1(l) = \begin{pmatrix}
            1 & 0  & \PsiNusselt' \\
            l^2-k^2 & 0  & 0 \\
            -2 kl & 1 & - \sigma k^2
        \end{pmatrix}, \qquad 
        \pi \tilde{\mathscr{B}}_0(l) = \begin{pmatrix}
            1 & 0 & 0 \\
            il & 0 & 0 \\
            0 & 0 & 1
        \end{pmatrix},
    \end{equation*}
    using the same Douglis--Nirenberg numbers as $\mathscr{A}, \mathscr{B}_1$ and $\mathscr{B}_0$ in Proposition~\ref{Schauder}.
    We now consider a continuous family of operators $\left(\mathcal{A}(t), \mathcal{B}_1(t), \mathcal{B}_0(t)\right)$ parameterised by $t \in [0,1]$, with matrix polynomials
    \begin{equation*}
        \mathscr{A}(l ; t) =   \begin{pmatrix}
            - i l (k^2 + l^2) + k\Rey (t-1)(\PsiNusselt' l + i \PsiNusselt'') & ik &  0 \\
            i k (k^2 + l^2) + \Rey \PsiNusselt' k^2 (1-t) & il & 0 \\
            0  & 0 & -(k^2 + l^2)
        \end{pmatrix}, 
    \end{equation*}
    \begin{equation*}
        \mathscr{B}_1 (l ; t) =   \begin{pmatrix}
            1 & 0  & \PsiNusselt' (1-t) \\
            l^2-(1-2t)k^2 & 0  & \PsiNusselt'''(t-1) \\
            -2 kl(1-t) & 1-t & 2\cot \theta (t-1) - (\sigma + t)(t + k^2)
        \end{pmatrix}, \qquad 
        \mathscr{B}_0 (l ; t) =   \begin{pmatrix}
            1 & 0 & 0 \\
            il & 0 & 0 \\
            0 & 0 & 1
        \end{pmatrix}.
    \end{equation*}
    Here, $\left(\mathcal{A}(0), \mathcal{B}_1(0), \mathcal{B}_0(0)\right) = \big( \tilde{\mathcal{A}}, \tilde{\mathcal{B}}_1, \tilde{\mathcal{B}}_0 \big)$, while at $t=1$,
    \begin{equation*}
        \tilde{\mathcal{A}}(1) =   \begin{pmatrix}
            - \Delta \d_y & \d_x &  0 \\
            \Delta \d_x & \d_y & 0 \\
            0  & 0 & \Delta
        \end{pmatrix}, 
    \end{equation*}
    \begin{equation*}
        \tilde{\mathcal{B}}_1(1) = \begin{pmatrix}
            1 & 0  & 0 \\
            \Delta & 0  & 0 \\
            0 & 0 & (1 + \sigma) (1 - \d_{xx})
        \end{pmatrix}, \qquad 
        \tilde{\mathcal{B}}_0(1) = \begin{pmatrix}
            1 & 0 & 0 \\
            \d_y & 0 & 0 \\
            0 & 0 & 1
        \end{pmatrix}.
    \end{equation*}
    We have constructed this family of operators $\left(\mathcal{A}(t), \mathcal{B}_1(t), \mathcal{B}_0(t)\right)$ so that the index of $\left(\mathcal{A}(1), \mathcal{B}_1(1), \mathcal{B}_0(1)\right)$ will be easy to calculate below.

    Now we verify for all $t \in [0,1]$, the operators $\left(\mathcal{A}(t), \mathcal{B}_i(t)\right)$ on each boundary satisfy the Shapiro--Lopatinsky condition. Setting $M_i(l;t) =  (\pi \mathscr{B}_i(l;t)) \pi \mathscr{A}(l;t)$, we have
    \begin{equation*}
        M_1(l ; t) = \begin{pmatrix}
            - i l \left(k^2 + l^2\right) & i k \left(k^2 + l^2\right) & \PsiNusselt' \left(k^2 + l^2\right)^2 \left(t - 1\right)\\
            - i l \left(k^2 + l^2\right) \left(k^2 (2t - 1) + l^2\right) & i k \left(k^2 + l^2\right) \left(k^2 (2t - 1) + l^2\right) & 0\\
            i k \left(k^2 - l^2\right) \left(k^2 + l^2\right) \left(t - 1\right) & i l \left(k^2 + l^2\right) \left(3 k^2 + l^2\right) \left(t - 1\right) & k^2 (\sigma + t) \left(k^2 + l^2\right)^2
        \end{pmatrix},
    \end{equation*}
    and
    \begin{equation*}
        M_0(l ; t) = \begin{pmatrix}
            - i l \left(k^{2} + l^{2}\right) & i k \left(k^{2} + l^{2}\right) & 0\\
            l^{2} \left(k^{2} + l^{2}\right) & - k l \left(k^{2} + l^{2}\right) & 0\\
            0 & 0 & - \left(k^{2} + l^{2}\right)^{2}
        \end{pmatrix}.
    \end{equation*}
    From these we compute 
    \begin{equation*}
        \begin{pmatrix} M_1'(ik ; t) & M_1''(ik ; t) \end{pmatrix} = \begin{pmatrix}
            2 i k^2 & - 2 k^2 & 0 & 6 k & 2 i k & - 8 \PsiNusselt' \left(t - 1\right)\\
            4 i k^{4} \left(t - 1\right) & - 4 k^{4} \left(t - 1\right) & 0 & 4 k^{3} \left(3 t - 5\right) & 4 i k^{3} \left(t - 3\right) & 0\\
            -4 k^{4} \left(t - 1\right) & - 4 i k^{4} \left(t - 1\right) & 0 & 12 i k^{3} \left(t - 1\right) & - 4 k^{3} \left(t - 1\right) & - 8 k^{4} (\sigma + t)
        \end{pmatrix},
    \end{equation*}
    and
    \begin{equation*}
        \begin{pmatrix} M_0'(ik ; t) & M_0''(ik ; t) \end{pmatrix} = \begin{pmatrix}
            2 i k^{2} & - 2 k^{2} & 0 & 6 k & 2 i k & 0\\
            - 2 i k^{3} & 2 k^{3} & 0 & - 10 k^{2} & - 6 i k^{2} & 0\\
            0 & 0 & 0 & 0 & 0 & 8 k^{2}
        \end{pmatrix}
    \end{equation*}
    We find nonzero $3 \times 3$ minors of $M_1$ and $M_0$, for instance
    \begin{equation*}
        \det \begin{pmatrix}
            - 2 k^2 & 2 i k & 8 \PsiNusselt' k^2 (1-t)\\
            4 k^{4} (1-t) & 4 i k^{3} \left(t - 3\right) & 0\\
            4 i k^{4} (1-t) & 4 k^{3} (1-t) & - 8 k^{4} (\sigma + t)
        \end{pmatrix} = -128 k^{9} \big(2\PsiNusselt'(1) (t-1)^2 + i (\sigma + t)\big),
    \end{equation*}
    and
    \begin{equation*}
        \det \begin{pmatrix}
            - 2 k^{2} & 2 i k & 0\\
            2 k^{3} & - 6 i k^{2} & 0\\
            0 & 0 & 8 k^{2}
        \end{pmatrix} = 64 i k^{6},
    \end{equation*}
    which establish the Shapiro--Lopatinsky condition on each boundary for all $t \in [0,1]$. 
    
    Then, as in the proof of Proposition~\ref{Schauder}, we combine the Schauder estimates and obtain that $\left(\mathcal{A}, \mathcal{B}_1, \mathcal{B}_0\right)(t)$ is a semi-Fredholm operator for $t \in [0,1]$. Now, since the index of semi-Fredholm operators is preserved under continuous deformation (see for example~\cite[Chapter IV Theorem 5.22]{kato1966perturbation}), the index of $\left(\mathcal{A}, \mathcal{B}_1, \mathcal{B}_0\right)(t)$ is constant along $t \in [0,1]$. Thus, 
    \begin{equation}\label{eq: ind_L_AB1B0}
        \operatorname{ind} \left(L \colon \tilde{\X} \to \Y \right) = \operatorname{ind} \left(\big( \tilde{\mathcal{A}}(1), \tilde{\mathcal{B}}_1(1), \tilde{\mathcal{B}}_0(1) \big) \colon \tilde{\X} \to \tilde{\Y}\right),
    \end{equation}
    and so we evaluate the solvability of the following system for $\big(\dot{\psi}, \dot{p}, \dot{\eta}\big) \in \tilde{\X} \subset C^{4,\alpha}(\Rect) \times C^{2,\alpha}(\Rect) \times C^{4,\alpha}(\Rect)$ and $(f_1, f_2, g_1, g_2, g_3, g_4) \in \Y$:
    \begin{equation}\label{eq: indzero_system}
        \begin{cases}
            - \Delta \dot{\psi}_y + \dot{p}_x = f_1 \qquad &\text{in}~\Rect,\\
            \Delta \dot{\psi}_x + \dot{p}_y = f_2 &\text{in}~\Rect, \\
            \Delta \dot{\eta} = 0 &\text{in}~\Rect, \\
            \dot{\psi} = g_1 &\text{on}~\Top, \\
            - \Delta \dot{\psi} = g_2 &\text{on}~\Top, \\
            (1 - \d_{xx}) \dot{\eta} = \dfrac{g_3}{1+\sigma} &\text{on}~\Top, \\
            \dot{\psi} = 0 &\text{on}~\Bottom, \\
            \dot{\psi}_y = g_4 &\text{on}~\Bottom, \\
            \dot{\eta} = 0 &\text{on}~\Bottom.
        \end{cases}
    \end{equation}
    Taking the curl of the first two equations of~\eqref{eq: indzero_system}, we obtain a biharmonic boundary value problem for $\dot{\psi}$,
    \begin{align}
        \begin{cases} \label{eq: biharmonic}
            \Delta^2 \dot{\psi} = \d_x f_2 - \d_y f_1,  &\text{in}~\Rect,\\
            \dot{\psi} = g_1,  &\text{on}~\Top,\\
            \Delta \dot{\psi} = -g_2,  &\text{on}~\Top,\\
            \dot{\psi} = 0,  &\text{on}~\Bottom,\\
            \dot{\psi}_y = g_4,  &\text{on}~\Bottom,
        \end{cases}
    \end{align}
    a first-order system for $\dot{p}$,
    \begin{equation}
        \begin{cases} \label{eq: p_problem}
            \dot{p}_x = f_1 + \Delta \dot{\psi}_y, \\ 
            \dot{p}_y = f_2 - \Delta \dot{\psi}_x,
        \end{cases} \quad \text{in}~\Rect,
    \end{equation}
    and a harmonic boundary problem for $\dot{\eta}$,
    \begin{equation}
        \begin{cases}\label{eq: eta_poisson}
            \Delta \dot{\eta} = 0,  &\text{in}~\Rect,\\
            \left(1 - \d_{xx} \right) \dot{\eta} = \dfrac{g_3}{1+\sigma}, &\text{on}~\Top,\\
            \dot{\eta} = 0,  &\text{on}~\Bottom.
        \end{cases}
    \end{equation}
    It is easy to check that~\eqref{eq: indzero_system} is equivalent to~\eqref{eq: biharmonic}--\eqref{eq: eta_poisson}. 
    
    Now it is a straighforward exercise to check that~\eqref{eq: biharmonic} has a unique solution, by standard Lax--Milgram and elliptic regularity arguments.  
    Given~\eqref{eq: biharmonic}, the problem~\eqref{eq: p_problem} for $\dot{p}$ has a solution which is uniquely determined up to a constant, provided the solvability condition 
    \begin{equation*}
        \int_{\Torus_k} f_1 + \Delta \dot{\psi}_y \, dx = 0,
    \end{equation*}
    arising from periodicity, holds.  Finally, the harmonic problem~\eqref{eq: eta_poisson} for $\dot{\eta}$ with mixed boundary conditions is also uniquely solvable by classical results, since the operator $(1 - \d_{xx}) \colon C^{4,\alpha}(\Top) \to C^{2,\alpha}(\Top)$ is invertible. In total, the operator $\big( \tilde{\mathcal{A}}(1), \tilde{\mathcal{B}}_1(1), \tilde{\mathcal{B}}_0(1) \big)$ corresponding to~\eqref{eq: indzero_system} has a one-dimensional kernel and range of codimension $1$, and so it has index zero. By~\eqref{eq: ind_L_AB1B0}, the index of $L \colon \tilde{\X} \to \Y$ is also zero. Finally, since $\X$ is a codimension $1$ subspace of $\tilde{\X}$, the index of $L \colon \X \to \Y$ is $-1$. 
\end{proof}

\section{Linear theory around Nusselt shear flow}\label{section: os}
In this section, we show how the Orr--Sommerfeld operator $\OS$ defined in Section~\ref{section: os_intro} relates naturally to the linearised operator $\mathscr{L} = \F_u \big(u_0(\theta, \gamma), \theta, \gamma \big)$ defined in~\eqref{eq: triv_lin}. This allows us to reduce the kernel and transversality hypotheses in Theorem~\ref{bif_lemma} to conditions on the Orr--Sommerfeld operator alone, by completely soft arguments. Then in Section~\ref{section: adjoint}, we provide a convenient test for the Orr--Sommerfeld transversality condition, by introducing an explicit formal adjoint for $\OS$ with respect to a suitable bilinear map and establishing a Fredholm alternative.

\subsection{Reduction to Orr--Sommerfeld}\label{section: reduction}
As briefly mentioned in Section~\ref{section: os_intro}, we define the Orr--Sommerfeld operator $\OS(k, \Rey, \theta, \gamma)\map{\Xos}{\Yos}$ by the left-hand side of the system~\eqref{eq: os_system}, that is,
\begin{equation}\label{eq:os_op}
    \OS(k, \Rey, \theta, \gamma)\hat{\varphi} :=
    \begin{pmatrix}
        \hat{\varphi}^{(4)} - \big( 2  k^2 + ik \Rey \PsiNusselt'\big) \hat{\varphi}'' +  \big( ik \Rey \PsiNusselt''' +  k^4 + ik^3 \Rey \PsiNusselt'\big) \hat{\varphi} \\
        \left[ \hat{\varphi}'' + \big(k^2 + 2(\PsiNusselt')^{-1}\big) \hat{\varphi}  \right] \big\vert_{y=1} \\
        \left[ \hat{\varphi}''' - \big( 3  k^2 + ik \Rey \PsiNusselt'\big) \hat{\varphi}' + ik (\PsiNusselt')^{-1} \left(2  \cot \theta + \sigma k^2 \right) \hat{\varphi} \right] \big\vert_{y=1}
    \end{pmatrix},
\end{equation}
between spaces $\Xos$ and $\Yos$ defined as follows,
\begin{align}
    \Xos &:= \left\{ \hat{\varphi} \in C^{4, \alpha} \left([0,1]; \Complex\right) : \hat{\varphi}(0)=\hat{\varphi}'(0)=0, ~ \operatorname{Im}\hat{\varphi}(1) = 0 \right\}, \label{eq: os_domain} \\
    \Yos &:= C^{\alpha}([0,1];\Complex) \times \Complex^2.  \label{eq: os_codomain}
\end{align}
We view these as Banach spaces over $\Real$ with the expected real norms. The first two conditions $\hat{\varphi}(0)=\hat{\varphi}'(0)=0$ in the definition of $\Xos$ arise from the no-slip condition at the bed, while the condition $\operatorname{Im}\hat{\varphi}(1) = 0$ comes from the phase constraint $\eta_x(0,1)=0$ in the definition of $\X$.

Given $u \in \X$ and $f \in \Y$, denote by $\hat{u}$ and $\hat{f}$ their first Fourier coefficients with respect to $x$, i.e.
\begin{equation*}
    \hat{u}(y) := \frac{k}{2 \pi} \int_{\Torus_k} u(\tilde{x},y) e^{-ik \tilde{x}} d\tilde{x},
\end{equation*}
and similarly for $\hat{f}$. Let us denote by $\mathcal{P}_{\X}$ and $\mathcal{P}_{\Y}$ the projections onto the first Fourier modes of $\X$ and $\Y$ respectively, i.e.~for all $u \in \X$, $f \in \Y$,
\begin{equation*}
    \mathcal{P}_{\X} u = \Re \big\{ \hat{u}(y) e^{ikx} \big\}, \quad
    \mathcal{P}_{\Y} f = \Re \big\{\hat{f} e^{ikx} \big\}.
\end{equation*}
Observe that $\mathcal{P}_{\Y} \mathscr{L}  = \mathscr{L} \mathcal{P}_{\X}$. 

We define a map $\mathscr{E}(k, \Rey, \theta, \gamma) \colon \Xos \to \mathcal{P}_\X \X$ as follows. Given $\hat{\varphi} \in \Xos$, let $\hat{\zeta}$ be the unique solution in $C^{4,\alpha}([0,1])$ of the ODE boundary value problem
\begin{equation}\label{eq: eta_ode_system}
    \begin{cases}
        \hat{\zeta}''(y) = k^2 \hat{\zeta}(y), \\
        \hat{\zeta}(1) = -\dfrac{1}{\PsiNusselt'(1)} \hat{\varphi}(1), \\
        \hat{\zeta}(0) = 0,            
    \end{cases}
\end{equation}
and $\hat{q} \in C^{2,\alpha}([0,1])$ the unique solution of the first-order ODE initial value problem
\begin{equation}\label{eq: p_system}
    \begin{cases}
        \hat{q}' = k^2 \left( ik  - \Rey \Psi_0'\right)\hat{\varphi} - ik \hat{\varphi}'', \\
        \hat{q}(1) =  -2ik  \hat{\varphi}'(1) - \dfrac{1}{\Psi_0'(1)} (-P_0' + \sigma k^2) \hat{\varphi}(1).
    \end{cases}
\end{equation}
Then set
\begin{equation}\label{eq: os_extension}
    \mathscr{E} \hat{\varphi} := \Re \left\{\big(\hat{\varphi}, \hat{q}, \hat{\zeta} \big)^T e^{ikx}\right\}.
\end{equation}
This map is motivated by considering the Fourier transform of $\mathscr{L} \mathcal{Q} (\varphi, q, \zeta)^T$ for $(\varphi, q, \zeta) \in \X$, setting its second, third, fifth and sixth components equal to zero, and solving for $\hat{q}$ and $\hat{\zeta}$ in terms of $\hat{\varphi}$. Note that $\mathscr{E}$ is a bounded linear operator with trivial kernel. The following lemma expresses $\OS$ as a composition involving $\mathscr{L}$ and $\mathscr{E}$.

\begin{lemma}\label{OS_product}
    There is a decomposition $\mathcal{P}_{\Y} \Y = \Y_1 \times \Y_2$, with projections $\Pi_1$ and $\Pi_2$ onto $\Y_1$ and $\Y_2$, and an invertible linear operator $T \colon \Y_2 \to \Yos$, such that the operators $\OS$, $\mathscr{L}$, $\mathcal{Q}$ and $\mathscr{E}$ defined by~\eqref{eq:os_op},~\eqref{eq: triv_lin},~\eqref{eq: cov_matrix}, and~\eqref{eq: eta_ode_system}--\eqref{eq: os_extension} respectively satisfy
    \begin{equation} \label{eq: os_comp}
        \OS = T \Pi_2 \mathscr{L} \mathcal{Q} \mathscr{E}, \qquad \Xos \xrightarrow{\mathscr{E}} \mathcal{P}_{\X} \X \xrightarrow{\mathscr{L} \mathcal{Q}} \mathcal{P}_{\Y} \Y \xrightarrow{\Pi_2} \Y_2 \xrightarrow{T} \Yos.
    \end{equation}
    Furthermore,
    \begin{equation}\label{eq: range_E}
        \operatorname{ran} \mathscr{E} = \ker \Pi_1 \mathscr{L} \mathcal{Q} \cap \mathcal{P}_{\X} \X.
    \end{equation}
\end{lemma}
\begin{proof}
    Recall the operator $L := \mathscr{L} \mathcal{Q}$ given by~\eqref{eq:lin_op}. Note that, given $(\varphi, q, \zeta) \in \X$,~\eqref{eq: eta_ode_system} consists of the Fourier transforms in $x$ of $\Delta \zeta = 0$, the third component of $L (\varphi, q, \zeta)^T = 0$ and $\zeta \vert_\Bottom = 0$, and~\eqref{eq: p_system} is composed of the Fourier transforms of the second and fifth components of $L (\varphi, q, \zeta)^T = 0$. Therefore, for any $\hat{\varphi} \in \Xos$, by the definition~\eqref{eq: eta_ode_system}--\eqref{eq: os_extension} of $\mathscr{E}$,
    \begin{equation*}
        L \mathscr{E} \hat{\varphi} =     \begin{pmatrix}
            \Re \big\{\big[\big(  k^2 + ik \Rey \Psi_0' \big)\hat{\varphi}' -  \hat{\varphi}''' - ik \Rey \Psi_0'' \hat{\varphi} + ik\hat{q} \big] e^{ikx} \big\} \\
            0 \\
            0 \\
            \Re \big\{-\big[\hat{\varphi}''(1) + k^2 \hat{\varphi}(1) + \Psi_0''' \hat{\zeta}(1)\big] e^{ikx} \big\} \\
            0 \\
            0
        \end{pmatrix}.
    \end{equation*}

    Then we decompose $\mathcal{P}_{\Y} \Y = \Y_1 \times \Y_2$, with projections $\Pi_1$ and $\Pi_2$ onto $\Y_1$ and $\Y_2$ respectively, defined by
    \begin{equation*}
        \Pi_1 a = (a_2, a_3, a_5, a_6), \quad \Pi_2 a = (a_1, a_4), \qquad \text{for all}~a=(a_1, \ldots, a_6) \in \mathcal{P}_{\Y} \Y,
    \end{equation*}
    with 
    \begin{align*}
        \Y_1 &= \Pi_1 \mathcal{P}_{\Y} \Y = C^{1, \alpha}(\Rect) \times C^{4, \alpha} \left(\Top\right)\times C^{2, \alpha}\left(\Top\right) \times C^{3, \alpha}\left(\Bottom\right), \\
        \Y_2 &= \Pi_2 \mathcal{P}_{\Y} \Y = C^{1, \alpha}(\Rect) \times C^{2, \alpha}(\Top).
    \end{align*} 

    Define $T \colon \Y_2 \to \Yos$ by 
    \begin{equation}\label{eq: invertible_T}
        T \begin{pmatrix}
            a_1 \\ a_4 
        \end{pmatrix} = \begin{pmatrix}
            - \d_y \hat{a}_1 \\ -\hat{a}_4 \\ - \hat{a}_1(1)
        \end{pmatrix} \quad \text{for all}~\begin{pmatrix}
            a_1 \\ a_4 
        \end{pmatrix} \in \Y_2.
    \end{equation}
    Then $T$ is invertible, with 
    \begin{equation*}
        T^{-1} \begin{pmatrix}
            f \\ h_1 \\ h_2
        \end{pmatrix} = \Re \Bigg\{ \begin{pmatrix}
            \int_y^1 f \, dy - h_2 \\ -h_1
        \end{pmatrix} e^{ikx} \Bigg\},
    \end{equation*}
    and direct computation verifies that $\OS$ is given by the composition~\eqref{eq: os_comp}.

    To show~\eqref{eq: range_E}, one inclusion $\operatorname{ran} \mathscr{E} \subset \ker \Pi_1 \mathcal{P}_{\Y} L$ follows directly, since for any $\hat{\varphi} \in \Xos$,
    \begin{equation*}
        \Pi_1 \mathcal{P}_{\Y} L \mathscr{E} \hat{\varphi} = \Pi_1 L \mathscr{E} \hat{\varphi} = (0, 0, 0, 0).
    \end{equation*}
    For the other inclusion, let $(\varphi, q, \zeta) \in \mathcal{P}_\X \X$ and suppose that $\Pi_1 \mathcal{P}_\Y L (\varphi, q, \zeta) = 0$. Then $(\varphi, q, \zeta)$ is equal to its first mode,
    \begin{equation*}
        (\varphi, q, \zeta)^T = \Re \left\{\big(\hat{\varphi}, \hat{q}, \hat{\zeta} \big)^T e^{ikx} \right\},
    \end{equation*}
    and $\hat{\varphi}$, $\hat{q}$, and $\hat{\zeta}$ satisfy
    \begin{align*}
        \left( -ik  + \Rey \PsiNusselt'(y)\right)\hat{\varphi}(y) + ik \hat{\varphi}''(y) + \hat{q}'(y) &= 0, \\
        \hat{\varphi}(1) + \PsiNusselt'(1) \hat{\zeta}(1) &= 0, \\
        -2ik \hat{\varphi}'(1) - \hat{q}(1) + (2 \cot \theta + \sigma k^2) \hat{\zeta}(1) &= 0, \\
        \hat{\varphi}'(0) &= 0.
    \end{align*}
    This implies that $\hat{\zeta}$ and $\hat{q}$ are given by~\eqref{eq: eta_ode_system} and~\eqref{eq: p_system}, $\hat{\varphi} \in \Xos$, and $(\varphi, q, \zeta) = \mathscr{E} \hat{\varphi}$, completing the proof.
\end{proof}

The following consequence of this lemma reduces the kernel hypothesis~\ref{H2} of Theorem~\ref{bif_lemma} to the Orr--Sommerfeld kernel condition~\ref{assumption: os_kernel} in Theorem~\ref{thm: general_local}.
\begin{lemma}\label{Kernels_equivalence}
    Let $k, \gamma > 0$, $\Rey \geq 0$, $\theta \in (0,\pi)$, and suppose that there exists $\hat{\varphi} \in \Xos$ such that 
    \begin{align}\label{eq: 1d_ker_ossection}
        \ker \OS(k, \Rey, \theta, \gamma) = \operatorname{span}\{\hat{\varphi}\}.
    \end{align}
    Then, there is an element $\xi_0 := \mathcal{Q} \mathscr{E} \hat{\varphi} \in \X$ such that
    \begin{equation*}
        \operatorname{span} \{\xi_0\} \subset \ker \mathscr{L} (\theta, \gamma).
    \end{equation*}
    Furthermore, if in addition $\gamma \neq 2$ and the kernel of $\OS(nk, \Rey, \theta, \gamma)$ is trivial for all $n \in \mathbb{N} \setminus \{1\}$, then this inclusion is an equality, i.e.~the kernel of $\mathscr{L}(\theta, \gamma)$ is one-dimensional.
\end{lemma}
\begin{proof}
    If $\hat{\varphi} \in \ker \OS(k, \Rey, \theta, \gamma)$, then by Lemma~\ref{OS_product}, $\Pi_2 L \mathscr{E} \hat{\varphi} = T^{-1} \OS \hat{\varphi} = 0$, and by construction, $\Pi_1 L \mathscr{E} \hat{\varphi} = 0$. Thus, we have an element $\xi_0 \in \X$ which is in the kernel of the linearised operator, namely
    \begin{equation*}
        \xi_0 := \mathcal{Q} \mathscr{E} \hat{\varphi} \in \ker \mathscr{L}(\theta, \gamma).
    \end{equation*} 

    Conversely, let $(\psi, p, \eta) \in \ker \mathscr{L}(\theta, \gamma)$, and let $(\varphi, q, \zeta) = \mathcal{Q}^{-1} (\psi, p, \eta)$. Let $n \in \mathbb{N} \cup \{0\}$ and consider the Fourier coefficient $(\hat{\varphi}_n, \hat{q}_n, \hat{\zeta}_n)$ for the $n$th mode of $(\varphi, q, \zeta)$. This satisfies
    \begin{align*}
        &\begin{cases}
            \left( (nk)^2 + ink \Rey \PsiNusselt' \right)\hat{\varphi}_n' -  \hat{\varphi}_n''' - ink \Rey \PsiNusselt'' \hat{\varphi}_n + ink\hat{q}_n = 0, \\
            (nk)^2 \left( -ink  + \Rey \PsiNusselt'\right)\hat{\varphi}_n + ink \hat{\varphi}_n'' + \hat{q}_n' = 0,               \\
            (nk)^2 \hat{\zeta}_n - \hat{\zeta}_n'' = 0,
        \end{cases} y \in [0,1], \\
        &\begin{cases}
            \big[ \hat{\varphi}_n + \PsiNusselt' \hat{\zeta}_n \big] \big\vert_{y=1} = 0, \\
            \big[\hat{\varphi}_n'' + (nk)^2 \hat{\varphi}_n + \PsiNusselt''' \hat{\zeta}_n\big]\big\vert_{y=1} = 0,                           \\
            \big[-2ink  \hat{\varphi}_n' - \hat{q}_n + (-P_0' + \sigma (nk)^2) \hat{\zeta}_n\big]\big\vert_{y=1} = 0,                   \\
            \operatorname{Re} \big\{ink \hat{\zeta}_n(1) \big\} = 0, \\
            \hat{\varphi}_n(0) = \hat{\varphi}_n'(0) = \hat{\zeta}_n(0) = 0.
        \end{cases}
    \end{align*}    
    Precisely by the construction of $\mathscr{E}$, a short calculation yields that this $n$th Fourier mode is given by 
    \begin{equation*}
        \Re \left\{\big(\hat{\varphi}_n, \hat{q}_n, \hat{\zeta}_n \big)^T e^{inkx}\right\} = \mathscr{E}(nk, \Rey, \theta, \gamma) \hat{\varphi}_n,
    \end{equation*}
    with $\hat{\varphi}_n \in \ker \OS(nk, \Rey, \theta, \gamma)$. Considering first $n=0$, the kernel of $\OS(0, \Rey, \theta, \gamma)$ consists of all $\hat{\varphi}_0 \in \Xos$ satisfying
    \begin{equation*}
        \hat{\varphi}_0^{(4)} \equiv 0, \quad \hat{\varphi}_0'''(1) = 0, \quad \hat{\varphi}_0''(1)+ \frac{2}{1-\gamma}\hat{\varphi}_0(1) = 0.
    \end{equation*}
    Up to scaling, the unique solution in $\Xos$ to the first two conditions is $\hat{\varphi}_0 = y^2$, solving the third equation only if $\gamma = 2$. So, under the hypothesis $\gamma \neq 2$, the kernel of $\OS(0, \Rey, \theta, \gamma)$ is trivial. Combining this with~\eqref{eq: 1d_ker_ossection} and the assumed triviality of the kernel for $n = 2, 3, \ldots$, we have 
    \begin{equation*}
        \ker \OS(nk, \Rey, \theta, \gamma) = \begin{cases}
            \{0\} \qquad n \in \{0\} \cup \mathbb{N} \setminus \{1\}, \\
            \operatorname{span}\{\hat{\varphi}\} \qquad n = 1.
        \end{cases}
    \end{equation*}
    Therefore, all Fourier modes except the first mode vanish, and we have $(\psi, p, \eta) \in \operatorname{span} \{\xi_0\}$. 
\end{proof}

Another soft argument using Lemma~\ref{OS_product} reduces the transversality hypothesis~\ref{H3} of Theorem~\ref{bif_lemma} to the Orr--Sommerfeld transversality condition~\ref{assumption: os_transv} in Theorem~\ref{thm: general_local}. 
\begin{lemma}\label{transv_proj}
    In the setting of Lemma~\ref{Kernels_equivalence}, suppose also that there exists some nonzero $(\dot{\theta}, \dot{\gamma}) \in \Real^2$ and $w \in \X$ such that 
    \begin{equation}\label{eq: pde_nontransverse}
        (D_{\theta} \mathscr{L}(\theta, \gamma) \dot{\theta}) \xi_0 + (D_{\gamma} \mathscr{L}(\theta, \gamma) \dot{\gamma}) \xi_0 = \mathscr{L} w.
    \end{equation}
    Then there exists a nonzero $\hat{\psi} \in \Xos$ such that 
    \begin{equation}\label{eq: os_nontransverse}
        \big(\d_\theta \OS(k,\Rey,\theta,\gamma) \dot{\theta}\big) \hat{\varphi} + \big(\d_\gamma \OS(k,\Rey,\theta,\gamma) \dot{\gamma}\big) \hat{\varphi} = \OS(k,\Rey,\theta,\gamma) \hat{\psi}.
    \end{equation}
\end{lemma}
\begin{proof}
    Within this proof we will suppress the arguments of operators such as $\mathscr{L}(\theta, \gamma)$, and will use the notation $\lambda = (\theta, \gamma)$. Differentiating the composition formula~\eqref{eq: os_comp} using that $T$ defined in~\eqref{eq: invertible_T} is parameter-independent gives
    \begin{equation}\label{eq: diff_os}
        (D_\lambda \OS \dot{\lambda}) \hat{\varphi} = T \Pi_2 (D_\lambda (\mathscr{L} \mathcal{Q} \mathscr{E}) \dot{\lambda}) \hat{\varphi}.
    \end{equation}
    To compute the right-hand side of~\eqref{eq: diff_os}, we differentiate $\mathscr{L} \mathcal{Q} \mathscr{E}$ and use $\xi_0 = \mathcal{Q} \mathscr{E} \hat{\varphi}$ and~\eqref{eq: pde_nontransverse} in the following calculation,
    \begin{align*}
        (D_{\lambda} (\mathscr{L} \mathcal{Q} \mathscr{E}) \dot{\lambda}) \hat{\varphi} &= (D_{\lambda} \mathscr{L} \dot{\lambda}) \xi_0 + \mathscr{L} (D_{\lambda} \mathcal{Q} \dot{\lambda}) \mathscr{E} \hat{\varphi} + \mathscr{L} \mathcal{Q} (D_{\lambda} \mathscr{E} \dot{\lambda}) \hat{\varphi} \\
        &= \mathscr{L} \big[w + (D_{\lambda} \mathcal{Q} \dot{\lambda}) \mathscr{E} \hat{\varphi} + \mathcal{Q} (D_{\lambda} \mathscr{E} \dot{\lambda}) \hat{\varphi}\big].
    \end{align*}
    Recalling $L = \mathscr{L} \mathcal{Q}$, by the invertibility of $\mathcal{Q}$ there is then some $\tilde{\varphi} \in \X$ such that
    \begin{equation}\label{eq: diff_LE}
        (D_{\lambda} (L \mathscr{E}) \dot{\lambda}) \hat{\varphi} = L \tilde{\varphi}.
    \end{equation}
    Using only that $\mathcal{P}_\Y L = L \mathcal{P}_\X$ and $\mathcal{P}_\X \mathscr{E} = \mathscr{E}$, we find 
    \begin{equation*}
        \mathcal{P}_{\Y} D_{\lambda} (L \mathscr{E}) = D_{\lambda} (L \mathscr{E}),
    \end{equation*}    
    so applying $\mathcal{P}_\Y$ to~\eqref{eq: diff_LE} gives
    \begin{equation}\label{eq: diff_LE2}
        (D_{\lambda} (L \mathscr{E}) \dot{\lambda}) \hat{\varphi} = L \mathcal{P}_\X \tilde{\varphi}.
    \end{equation}
    Now by~\eqref{eq: diff_LE} and the identity $\Pi_1 L \mathscr{E} \equiv 0$ from~\eqref{eq: range_E} in Lemma~\ref{OS_product},
    \begin{equation*}
        \Pi_1 L \mathcal{P}_\X \tilde{\varphi} = \Pi_1 (D_{\lambda} (L \mathscr{E}) \dot{\lambda}) \mathcal{P}_\X \hat{\varphi} =  (D_{\lambda} (\Pi_1 L \mathscr{E}) \dot{\lambda}) \mathcal{P}_\X \hat{\varphi} = 0,
    \end{equation*}
    so the other inclusion in~\eqref{eq: range_E} implies the existence of some $\hat{\psi} \in \Xos$ such that $\mathcal{P}_\X \tilde{\varphi} = \mathscr{E} \hat{\psi}$. Thus,~\eqref{eq: diff_LE2} implies
    \begin{equation*}
        (D_{\lambda} (L \mathscr{E}) \dot{\lambda}) \hat{\varphi} = L \mathscr{E} \hat{\psi}.
    \end{equation*}
    Inserting this in~\eqref{eq: diff_os} and using the decomposition of $\OS$, we have
    \begin{equation*}
        (D_\lambda \OS \dot{\lambda}) \hat{\varphi} = T \Pi_2 L \mathscr{E} \hat{\psi} = \OS \hat{\psi}. \qedhere
    \end{equation*}    
\end{proof}

\subsection{Transversality via formal adjoint of Orr--Sommerfeld}\label{section: adjoint}
We can further reduce the Orr--Sommerfeld transversality hypothesis~\ref{assumption: os_transv} to the non-singularity of a certain matrix. This matrix involves elements of the kernel of a formal adjoint of $\OS$ as well as derivatives of $\OS$ with respect to the parameters $\theta, \gamma$. 

We will introduce the formal adjoint over an extended space $\extXos \supset \Xos$ which does not include the phase constraint $\operatorname{Im} \hat{\varphi}(1) = 0$, that is
\begin{equation*}
    \extXos := \left\{ \hat{\varphi} \in C^{4, \alpha} \left([0,1]; \Complex\right) : \hat{\varphi}(0)=\hat{\varphi}'(0)=0 \right\},
\end{equation*}
so that $\extXos$ is a vector space over $\Complex$. Observe that $\OS \colon \extXos \to \Yos$ is complex linear. Then we define a bilinear map $\langle \, \cdot \, , \, \cdot \, \rangle_{\mathrm{os}} \colon \extXos \times \Yos \to \Real$ by 
\begin{equation}\label{eq: bilinear}
    \left\langle \varphi^*, (f, h_1, h_2)^T \right\rangle_{\mathrm{os}} 
    := \Re \left\{\int_0^1 \overline{\varphi^*(y)} f(y) dy + \overline{(\varphi^*)'(1)} h_1 - \overline{\varphi^*(1)} h_2 \right\},
\end{equation}
for all $\varphi^* \in \extXos$ and $(f, h_1, h_2) \in \Yos$, with $\overline{z}$ denoting the complex conjugate of $z \in \Complex$. 

The following lemma establishes that the formal `adjoint' operator $\OS^* \colon \extXos \to \Yos$ of $\OS$ with respect to $\langle \cdot, \cdot \rangle_{\mathrm{os}}$ is given by
\begin{equation}\label{eq: os_adjoint}
    \OS^*(k, \Rey, \theta, \gamma) \varphi^* :=
    \begin{pmatrix}
        (\varphi^*)^{(4)} - \left( 2k^2 - ik \Rey \PsiNusselt'\right) (\varphi^*)'' + 2ik \Rey \PsiNusselt'' (\varphi^*)' + \left( k^4 - ik^3 \Rey \PsiNusselt'\right) \varphi^* \\
        \left[(\varphi^*)''+ k^2 \varphi^*\right] \big\vert_{y=1} \\
        \left[(\varphi^*)''' - \left(2 (\PsiNusselt')^{-1} + 3k^2 - ik \Rey \PsiNusselt'\right) (\varphi^*)' - ik (\PsiNusselt')^{-1} \left(2 \cot \theta + \sigma k^2 \right) \varphi^*\right] \big\vert_{y=1}
    \end{pmatrix},
\end{equation}
and standard functional analysis results then yield a Fredholm alternative for $\OS$ and $\OS^*$. 
\begin{lemma}\label{adj_fred_alt}
    Let $\OS^* \colon \extXos \to \Yos$ be given by~\eqref{eq: os_adjoint} and $\langle \cdot, \cdot \rangle_{\mathrm{os}} \colon \extXos \times \Yos \to \Real$ be given by~\eqref{eq: bilinear}. Then $\OS^*$ is formally `adjoint' to $\OS$ viewed as a map $\extXos \to \Yos$ in the sense that
    \begin{equation}\label{eq: adjointness}
        \langle \varphi_2, \OS \varphi_1\rangle_{\mathrm{os}} = \langle \varphi_1, \OS^* \varphi_2 \rangle_{\mathrm{os}} \quad \text{for all }\varphi_1, \varphi_2 \in \extXos.
    \end{equation}
    Furthermore, for every $g \in \Yos$, $g$ is in the range of $\OS$ if and only if $\langle \varphi_2, g \rangle_{\mathrm{os}} = 0$ for all $\varphi_2 \in \ker \OS^*$.
\end{lemma}
\begin{proof}
    Let $(f, h_1, h_2) \in \operatorname{ran} \OS$, i.e.~there is a $\varphi_1 \in \extXos$ such that
    \begin{equation*}
        \OS(k, \Rey, \theta, \gamma) \varphi_1 = (f, h_1, h_2)^T.
    \end{equation*}
    Then we compute, first integrating by parts, and then taking complex conjugates,
    \begin{align*}
        \langle \varphi_2, \OS \varphi_1 \rangle_{\mathrm{os}}    
        &= \Re \left\{\int_0^1 \overline{\varphi_2}f dy + \overline{(\varphi^*)'(1)} h_1 - \overline{\varphi^*(1)} h_2 \right\}\\ 
        &= \Re \Bigg\{\int_0^1 \overline{\varphi_2} \left[\varphi_1^{(4)} - \left( 2k^2 + ik \Rey \PsiNusselt'\right) \varphi_1'' +  \left( ik \Rey \PsiNusselt''' + k^4 + ik^3 \Rey \PsiNusselt'\right) \varphi_1 \right] dy \\ 
        &\qquad + \overline{(\varphi_2)'(1)} \left[ \Big(k^2 + \frac{2}{\PsiNusselt'}\Big) \varphi_1(1) + \varphi_1''(1)  \right] \\ 
        &\qquad - \overline{\varphi_2(1)} \left[ \varphi_1'''(1) - \left( 3  k^2 + ik \Rey \PsiNusselt'\right) \varphi_1'(1) + \frac{ik}{\PsiNusselt'} \left(2  \cot \theta + \sigma k^2 \right) \varphi_1(1) \right] \Bigg\}\\
        &= \Re \Bigg\{\int_0^1 \varphi_1 \overline{\left[(\varphi_2)^{(4)} - \left( 2 k^2 - ik \Rey \PsiNusselt'\right) \varphi_2'' +  2ik \Rey \PsiNusselt'' \varphi_2' + \left(k^4 - ik^3 \Rey \PsiNusselt'\right) \varphi_2\right]} dy \\
        &\qquad + \varphi_1'(1) \overline{\left[ \varphi_2''(1) + k^2 \varphi_2(1) \right]} \\ &\qquad - \varphi_1(1) \overline{\left[ \varphi_2'''(1) - \left(2(\PsiNusselt')^{-1} + 3 k^2 - ik \Rey \PsiNusselt'\right) \varphi_2'(1) - ik(\PsiNusselt')^{-1} \left(2 \cot \theta + \sigma k^2 \right) \varphi_2(1) \right]} \Bigg\} \\
        &= \langle \varphi_1, \OS^* \varphi_2 \rangle_{\mathrm{os}},
    \end{align*}
    establishing~\eqref{eq: adjointness}.
    
    Now each of $\OS$ and $\OS^*$ is a Fredholm operator $\extXos \to \Yos$, due to being a linear ODE operator with smooth coefficients and a nonzero constant leading-order coefficient. Furthermore, each has Fredholm index zero, since the ODE is fourth order with two independent boundary conditions in addition to the two implicit homogeneous conditions in the definition of $\extXos$.

    For the second claim of the lemma, by Theorem 5.G in~\cite{zeidler2012applied} it suffices to show that $\langle \varphi, (f,h_1,h_2)^T \rangle_{\mathrm{os}}$ vanishing for all $(f,h_1,h_2)^T \in \Yos$ would imply $\varphi = 0$, and conversely that if $\langle \varphi, (f,h_1,h_2)^T \rangle_{\mathrm{os}} = 0$ for all $\varphi \in \Xos$ then $(f, h_1, h_2) = 0$. Suppose first that $\langle \varphi, (f,h_1,h_2)^T \rangle_{\mathrm{os}} = 0$ for all $(f,h_1,h_2)^T \in \Yos$. Then picking $(\varphi,0,0)^T \in \Yos$, we have $\int_0^1 |\varphi|^2 dy = 0$, so $\varphi\equiv 0$. Conversely, suppose $\langle \varphi, (f,h_1,h_2)^T \rangle_{\mathrm{os}} = 0$ for all $\varphi\in \extXos$. We have 
    \begin{multline}
        \left\langle \varphi, (f, h_1, h_2)^T \right\rangle_{\mathrm{os}} 
        = \int_0^1 \left(\Re \varphi\Re f + \operatorname{Im}\varphi\operatorname{Im}f\right) dy + \left[\Re \varphi'(1) \Re h_1 + \operatorname{Im} \varphi'(1) \operatorname{Im} h_1\right] \\ - \left[\Re \varphi(1) \Re h_2 + \operatorname{Im} \varphi(1) \operatorname{Im} h_2\right].
    \end{multline}
    Taking $\varphi\in C^\infty_c\left([0,1]; \Real\right)$ compactly supported in $[0,1]$, the fundamental lemma of the calculus of variations implies that $\Re f = 0$ and $\operatorname{Im} f = 0$ using $\varphi$ and $i\varphi$ respectively, so we obtain $f \equiv 0$. Considering instead $\varphi \in \extXos$ satisfying $\varphi(1) = 0$ and $\varphi'(1) \neq 0$, such as~$\varphi=y^2(y-1)$, we obtain $h_1 = 0$, by using $\varphi$ and $i \varphi$ to show that $\operatorname{Im} h_1 = 0$ and $\Re h_1 = 0$. Similarly, taking $\varphi \in \extXos$ such that $\varphi'(1) = 0$ but $\varphi(1) \neq 0$, e.g.~$\varphi=2y^3-3y^2$, we find that $h_2 = 0$, and the proof is complete.
\end{proof}

An immediate corollary of this Fredholm alternative is an equivalent condition for the hypothesis~\ref{assumption: os_transv} of Theorem~\ref{thm: general_local}:
\begin{cor}\label{cor: transv_det}
    Suppose that there exists $\hat{\varphi} \in \Xos$ such that $\ker \OS(k, \Rey, \theta, \gamma) = \operatorname{span} \{\hat{\varphi}\}$, viewing $\OS$ as a map $\Xos \to \Yos$. Then there exists a nonzero $\hat{\varphi}^* \in \Xos$ such that $\OS^*(k, \Rey, \theta, \gamma) \hat{\varphi}^* = 0$, and the condition~\eqref{eq: os_nontransverse} implies 
    \begin{equation}\label{eq: sing_det}
        \det \begin{pmatrix}
            \left\langle \hat{\varphi}^*, \d_\theta \OS (k, \Rey, \theta, \gamma) \hat{\varphi} \right\rangle_{\mathrm{os}} & \left\langle \hat{\varphi}^*, \d_\gamma \OS (k, \Rey, \theta, \gamma) \hat{\varphi} \right\rangle_{\mathrm{os}} \\
            \left\langle i\hat{\varphi}^*, \d_\theta \OS (k, \Rey, \theta, \gamma) \hat{\varphi} \right\rangle_{\mathrm{os}} & \left\langle i\hat{\varphi}^*, \d_\gamma \OS (k, \Rey, \theta, \gamma) \hat{\varphi} \right\rangle_{\mathrm{os}}
        \end{pmatrix} = 0.
    \end{equation}
\end{cor}
\begin{proof}
    The Fredholm alternative of Lemma~\ref{adj_fred_alt} implies that the kernel of $\OS^* \colon \extXos \to \Yos$ is two-dimensional. By the definition of $\Xos \subset \extXos$ and the complex linearity of $\OS^*$, there is an element $\hat{\varphi}^* \in \Xos$ such that $\ker \OS^* = \operatorname{span} \{\hat{\varphi}^*, i\hat{\varphi}^*\}$. Then again by this Fredholm alternative,~\eqref{eq: os_nontransverse} implies~\eqref{eq: sing_det}.
\end{proof}

\section{Uniform estimates}\label{section: uniform}
This section establishes some useful a priori estimates with uniform control in terms of the constant $\delta > 0$ in the inequalities
\begin{equation} \label{eq: uniform_inequalities}
    \inf_{\substack{z_1, z_2 \in \Top \\ z_1 \neq z_2}} \frac{\lvert f(z_1) - f(z_2) \rvert}{\lvert z_1 - z_2 \rvert} > \delta, \quad \sup_\Top \lvert \nabla \eta \rvert < \frac{1}{\delta}, \quad \inf_\Top \lvert \psi_y \rvert + \sigma > \delta.
\end{equation}
This set of inequalities is a strengthened version of~\eqref{eq: hypothesis_for_schauder}. The main tools are linear Schauder estimates of the type shown in~\eqref{eq: system_schauder} of Proposition~\ref{Schauder}.

\subsection{Estimate for the conformal map}
We open this section with a basic uniform estimate for the conformal map $f$:
\begin{lemma}\label{uniform_dist_dU}
    Let $f=(\xi, \eta)$ satisfy~\eqref{eq: not_touching_bed} and suppose there exists $\delta > 0$ such that the first inequality in~\eqref{eq: uniform_inequalities} holds. Then, there exists $C > 0$ dependent only on $\delta$ such that $f$ satisfies 
    \begin{equation}\label{eq: f_uniform}
        \inf_{\substack{z_1, z_2 \in \Rect \\ z_1 \neq z_2}} \frac{\lvert f(z_1) - f(z_2) \rvert}{\lvert z_1 - z_2 \rvert} > C(\delta), \quad \text{and} \quad \inf_\Top \eta > \delta.
    \end{equation}
    In particular, $\operatorname{dist} ((u,\theta, \gamma), \d \U) > C(\delta)$.
\end{lemma}

\begin{proof}
    The first inequality of~\eqref{eq: f_uniform} holds by Lemma~\ref{lemma: conformal}. To show that $\inf_\Top \eta > \delta$, let $(X_0, Y_0)$ be the point of minimum height of the surface $\Surf$, i.e.~$Y_0 = \inf_\Top \eta$. Then, let the points $(x_0, 0) \in \Bottom$ and $(x_1, 1) \in \Top$ map to $(X_0, 0) \in \Bed$ and $(X_0, Y_0) \in \Surf$ respectively, i.e.~$f(x_0, 0) = (X_0, 0)$ and~$f(x_1, 1) = (X_0, Y_0)$.  Since $f$ is injective on $\overline{\Rect}$ by Lemma~\ref{lemma: conformal}, 
    \begin{equation*}
        (x_1 - x_0, 1) = f^{-1}(X_0, Y_0) - f^{-1}(X_0, 0),
    \end{equation*}
    and so we obtain
    \begin{equation*}
        1 \leq \lvert f^{-1}(X_0, Y_0) - f^{-1}(X_0, 0) \rvert \leq \sup_\Omega \lvert D(f^{-1}) \rvert \cdot \lvert Y_0 \rvert.
    \end{equation*}
    From this, using~\eqref{eq: uniform_inequalities},
    \begin{equation*}
        \inf_\Top \eta \geq \frac{1}{\sup_\Omega \lvert D(f^{-1}) \rvert} = \inf_\Rect \lvert Df \rvert = \inf_\Rect \lvert \nabla \eta \rvert = \inf_\Top \lvert \nabla \eta \rvert > \delta. \qedhere
    \end{equation*}
\end{proof}

\subsection{A priori estimate}
This subsection is devoted to establishing the following a priori estimate on an arbitrary solution $u$ in $\U$. 
\begin{prop}\label{uniform_prop}
    Let $\delta > 0$ and suppose $(u, \theta, \gamma) \in \F^{-1}(0)$, where $u = (\psi, p, \eta)$ satisfies~\eqref{eq: uniform_inequalities}. Then, $u \in C^{5, \alpha}(\Rect) \times C^{3, \alpha}(\Rect) \times C^{5, \alpha}(\Rect)$, and
    \begin{equation}\label{eq: best_est}
        \lVert \psi \rVert_{C^{5, \alpha}(\Rect)} + \lVert p \rVert_{C^{3, \alpha}(\Rect)} + \lVert \eta \rVert_{C^{5, \alpha}(\Rect)} \leq C\Big(\delta, \lvert \cot \theta \rvert, \sup_\Top \lvert D^2\eta \rvert, \sup_\Rect \lvert \nabla \psi \rvert, \sup_\Top \lvert D^2 \psi \rvert \Big).
    \end{equation}
\end{prop}
Note that by Lemma~\ref{uniform_dist_dU},~\eqref{eq: uniform_inequalities} guarantees $\inf_\Rect \lvert \nabla \eta \rvert > C(\delta)$. To prove Proposition~\ref{uniform_prop}, we begin with an a priori estimate which comes from the ellipticity of $\F_u(u, \theta, \gamma)$.
\begin{lemma}\label{lemma: first_schauder}
    In the setting of Proposition~\ref{uniform_prop}, $u \in C^{5, \alpha}(\Rect) \times C^{3, \alpha}(\Rect) \times C^{5, \alpha}(\Rect)$ and
    \begin{equation}\label{eq: whole_Schauder}
        \lVert \psi \rVert_{C^{5, \alpha}(\Rect)} + \lVert p \rVert_{C^{3, \alpha}(\Rect)} + \lVert \eta \rVert_{C^{5, \alpha}(\Rect)} \leq C \left(\delta, \lvert \cot \theta \rvert, \sup_\Top \lvert D^2\eta \rvert, \sup_\Rect \lvert \nabla \psi \rvert, \sup_\Top \lvert D^2 \psi \rvert, \lVert p \rVert_{L^2(\Rect)} \right).
    \end{equation}
\end{lemma}
\begin{proof}
    Here, $(\psi, p, \eta) \in \U$ satisfies~\eqref{eq:flattened}, where $\tilde{\d}_Y (\lvert \nabla \eta \rvert^{-2} \Delta \psi)$ and $\tilde{\d}_X (\lvert \nabla \eta \rvert^{-2} \Delta \psi)$ in~\eqref{eq: flattened1} and~\eqref{eq: flattened2} are given by 
    \begin{align*}
        \tilde{\d}_Y \left(\lvert \nabla \eta \rvert^{-2} \Delta \psi\right) &= \frac{1}{\lvert \nabla \eta \rvert^4} \nabla \eta \cdot \nabla \Delta \psi - \frac{2 \Delta \psi}{\lvert \nabla \eta \rvert^6} \big[(\eta_x^2 - \eta_y^2) \eta_{xx} + 2 \eta_x \eta_y \eta_{xy} \big], \\
        \tilde{\d}_X \left(\lvert \nabla \eta \rvert^{-2} \Delta \psi\right) &= \frac{1}{\lvert \nabla \eta \rvert^4} \nabla \eta \cdot \nabla^{\perp} \Delta \psi + \frac{2 \Delta \psi}{\lvert \nabla \eta \rvert^6} \big[ 2\eta_x \eta_y \eta_{xx} + (\eta_x^2 - \eta_y^2) \eta_{xy} \big].
    \end{align*}
    Considering~\eqref{eq: flattened_dyn}, we have
    \begin{equation*}
        \tilde{\mathcal{D}} \psi = \begin{pmatrix}
            2 \tilde{\d}_{XY} \psi & \tilde{\d}_{YY} \psi - \tilde{\d}_{XX} \psi \\
            \tilde{\d}_{YY} \psi - \tilde{\d}_{XX} \psi & - 2 \tilde{\d}_{XY} \psi
        \end{pmatrix},
    \end{equation*}
    where, by straighforward computation,
    \begin{align*}
        \tilde{\d}_{YY} \psi - \tilde{\d}_{XX} \psi &= \frac{1}{\lvert \nabla \eta \rvert^4} \big[(\eta_y^2 - \eta_x^2) \psi_{yy} + 4 \eta_x \eta_y \psi_{xy} \big] + \frac{2 \psi_y}{\lvert \nabla \eta \rvert^6} \big[(\eta_x^2 - 3 \eta_y^2) \eta_x \eta_{xy} + (\eta_y^2 - 3 \eta_x^2) \eta_y \eta_{xx} \big], \\
        \tilde{\d}_{XY} \psi &= \frac{1}{\lvert \nabla \eta \rvert^4}\big[(\eta_y^2 - \eta_x^2) \psi_{xy} - \eta_x \eta_y \psi_{yy} \big] + \frac{2 \psi_y}{\lvert \nabla \eta \rvert^6} \big[(\eta_x^2 - 3 \eta_y^2) \eta_x \eta_{xx} - (\eta_y^2 - 3 \eta_x^2) \eta_y \eta_{xy} \big],
    \end{align*}
    so~\eqref{eq: flattened_dyn} on $\Top$ simplifies to
    \begin{align}
        - \psi_{yy} + \frac{2 \psi_y}{\lvert \nabla \eta \rvert^2}\big(\eta_x \eta_{xy} - \eta_y \eta_{xx}\big) &= 0, \\
        2 \psi_{xy} + \lvert \nabla \eta \rvert^2 p + \frac{\sigma}{\lvert \nabla \eta \rvert}\big(\eta_y \eta_{xx} - \eta_x \eta_{xy}\big) - \frac{2 \psi_y}{\lvert \nabla \eta \rvert^2} \big(\eta_y \eta_{xy} + \eta_x \eta_{xx}\big) &= 0. \label{eq: dyn_pressure}
    \end{align}

    Thus,~\eqref{eq:flattened} is equivalent to 
    \begin{equation}\label{eq: nonlinear}
        \begin{cases}
            -\dfrac{1}{\lvert \nabla \eta \rvert^4} \nabla \eta \cdot \nabla \Delta \psi + \tilde{\d}_X p = 2 - \dfrac{2}{\lvert \nabla \eta \rvert^6} \big[(\eta_x^2 - \eta_y^2) \eta_{xx} + 2 \eta_x \eta_y \eta_{xy} \big] \Delta \psi + \Rey \tilde{\nabla}^{\perp} \psi \cdot \tilde{\nabla} \tilde{\d}_Y \psi & \text{in } \Rect,\\
            \dfrac{1}{\lvert \nabla \eta \rvert^4} \nabla \eta \cdot \nabla^{\perp} \Delta \psi + \tilde{\d}_Y p = - 2 \cot \theta - \dfrac{2}{\lvert \nabla \eta \rvert^6} \big[ 2\eta_x \eta_y \eta_{xx} + (\eta_x^2 - \eta_y^2) \eta_{xy} \big] \Delta \psi - \Rey \tilde{\nabla}^{\perp} \psi \cdot \tilde{\nabla} \tilde{\d}_X \psi & \text{in } \Rect,\\
            \Delta \eta = 0 & \text{in } \Rect,\\
            \psi = \tfrac{2}{3} - \gamma & \text{on } \Top,\\
            - \psi_{yy} + \dfrac{2 \psi_y}{\lvert \nabla \eta \rvert^2} \big(\eta_x \eta_{xy} - \eta_y \eta_{xx}\big) = 0 & \text{on } \Top,\\
            2 \psi_{xy} + \lvert \nabla \eta \rvert^2 p + \dfrac{\sigma}{\lvert \nabla \eta \rvert}\big(\eta_y \eta_{xx} - \eta_x \eta_{xy}\big) - \dfrac{2 \psi_y}{\lvert \nabla \eta \rvert^2} \big(\eta_y \eta_{xy} + \eta_x \eta_{xx}\big) = 0 & \text{on } \Top,\\
            \psi = 0 & \text{on } \Bottom,\\
            \psi_y + \gamma \eta_y = 0 & \text{on } \Bottom,\\
            \eta = 0 & \text{on } \Bottom.
        \end{cases}
    \end{equation}

    The left-hand side of~\eqref{eq: nonlinear} has precisely the same principal parts~\eqref{eq: principalA},~\eqref{eq: principal_B1} and~\eqref{eq: principal_B0} as the linearisation~\eqref{eq: gen_lin} around $u=(\psi, p, \eta)$ applied to $\dot{u} = u$. Therefore, by~\eqref{eq: uniform_inequalities},~\eqref{eq: uniform_inequalities} and Proposition~\ref{Schauder}, this system admits the Schauder estimate~\eqref{eq: system_schauder}, which reads
    \begin{equation}\label{eq: first_Sch}
        \lVert u \rVert_{\X} \leq C \cdot (\lvert \cot \theta \rvert + \lvert \tfrac{2}{3} - \gamma \rvert + \lVert u \rVert_{L^1(\Rect) \times L^1(\Rect) \times L^1(\Rect)}),
    \end{equation}
    where $C$ depends on $\delta$, $\lVert \nabla \eta \rVert_{C^{1,\alpha}(\Rect)}$,  $\lVert D^2 \eta \rVert_{C^{1,\alpha}(\Rect)}$, $\lVert \nabla \psi \rVert_{C^{1,\alpha}(\Rect)}$, $\lVert \nabla \psi \rVert_{C^{2,\alpha}(\Top)}$ and $\lvert \gamma \rvert$.

    Note that from the fourth equation in~\eqref{eq: nonlinear},
    \begin{equation}\label{eq: gamma_est}
        \lvert \gamma \rvert \leq C \sup_\Top \lvert \psi \rvert \leq C \sup_\Rect \lvert \nabla \psi \rvert.
    \end{equation}
    Since $\eta$ vanishes on $\Bottom$, for any $k \in \mathbb{N} \cup \{0\}$ and $\beta \in (0,1)$ elliptic Schauder estimates yield that
    \begin{equation*}
        \lVert \eta \rVert_{C^{k,\beta}(\Rect)} \leq \lVert \eta \rVert_{C^{k,\beta}(\Top)}.
    \end{equation*}
    Therefore,~\eqref{eq: first_Sch} becomes
    \begin{equation}\label{eq: first_schauder}
        \lVert \psi \rVert_{C^{4, \alpha}(\Rect)} + \lVert p \rVert_{C^{2, \alpha}(\Rect)} + \lVert \eta \rVert_{C^{4, \alpha}(\Rect)} \leq C(\delta, \lvert \cot \theta \rvert, \lVert \psi \rVert_{C^{2,\alpha}(\Rect)}, \lVert \psi \rVert_{C^{3,\alpha}(\Top)}, \lVert \eta \rVert_{C^{3,\alpha}(\Top)},\lVert p \rVert_{L^1(\Rect)}).
    \end{equation}
    By Theorem~\ref{adn_schauder}, we can generalise~\eqref{eq: first_schauder} to obtain that for any integer $l \geq 0$, $u \in C^{3+l, \alpha}(\Rect) \times C^{1+l, \alpha}(\Rect) \times C^{3+l, \alpha}(\Rect)$ with 
    \begin{equation}\label{eq: gen_schauder}
        \lVert \psi \rVert_{C^{3+l, \alpha}(\Rect)} + \lVert p \rVert_{C^{1+l, \alpha}(\Rect)} + \lVert \eta \rVert_{C^{3+l, \alpha}(\Rect)} \leq C(\delta, \lvert \cot \theta \rvert, \lVert \psi \rVert_{C^{1+l,\alpha}(\Rect)}, \lVert \psi \rVert_{C^{2+l,\alpha}(\Top)}, \lVert \eta \rVert_{C^{2+l,\alpha}(\Top)}, \lVert p \rVert_{L^1(\Rect)}).
    \end{equation}
    In particular, bootstrapping~\eqref{eq: gen_schauder} from $l= 0$ to $l=2$ yields $u \in C^{5, \alpha}(\Rect) \times C^{3, \alpha}(\Rect) \times C^{5, \alpha}(\Rect)$ and
    \begin{equation}\label{eq: 5to3_schauder}
        \lVert \psi \rVert_{C^{5, \alpha}(\Rect)} + \lVert p \rVert_{C^{3, \alpha}(\Rect)} + \lVert \eta \rVert_{C^{5, \alpha}(\Rect)} \leq C(\delta, \lvert \cot \theta \rvert, \lVert \psi \rVert_{C^{1,\alpha}(\Rect)}, \lVert \psi \rVert_{C^{2,\alpha}(\Top)}, \lVert \eta \rVert_{C^{2,\alpha}(\Top)}, \lVert p \rVert_{L^1(\Rect)}).
    \end{equation}
    Now to estimate the right-hand side, Theorem~\ref{adn_sobolev} gives the following Sobolev estimate for the system~\eqref{eq: nonlinear} for $r > 2(1-\alpha)^{-1}$,
    \begin{equation}\label{eq: weak_elliptic}
        \lVert \psi \rVert_{W^{3, r}(\Rect)} + \lVert p \rVert_{W^{1, r}(\Rect)} + \lVert \eta \rVert_{W^{3, r}(\Rect)} \leq C(\delta, \lvert \cot \theta \rvert, \lVert \eta \rVert_{C^{2}(R)}, \lVert \psi \rVert_{C^1(\Rect)}, \sup_\Top \lvert D^2 \psi \rvert, \lVert p \rVert_{L^r(\Rect)}).
    \end{equation}
    Applying the maximum principle to the harmonic function $\eta$ and using the boundary condition $\eta \vert_\Bottom = 0$, it is straightforward to check that 
    \begin{equation*}
        \lVert \eta \rVert_{C^2(\Rect)} \leq C(\delta, \sup_\Top \lvert D^2\eta \rvert).
    \end{equation*}
    It is similarly easy to check by integrating from the bed that $\lVert \psi \rVert_{C^1(\Rect)} \leq \sup_\Rect \lvert \nabla \psi \rvert$. Interpolating using Ehrling's lemma, we can replace $\lVert p \rVert_{L^r(\Rect)}$ by $\lVert p \rVert_{L^2(\Rect)}$ in~\eqref{eq: weak_elliptic}, obtaining
    \begin{equation}\label{eq: weak_elliptic2}
        \lVert \psi \rVert_{W^{3, r}(\Rect)} + \lVert p \rVert_{W^{1, r}(\Rect)} + \lVert \eta \rVert_{W^{3, r}(\Rect)} \leq C\Big(\delta, \lvert \cot \theta \rvert, \sup_\Top \lvert D^2\eta \rvert, \sup_\Rect \lvert \nabla \psi \rvert, \sup_\Top \lvert D^2 \psi \rvert, \lVert p \rVert_{L^2(\Rect)}\Big).
    \end{equation}
    The result now follows by combining~\eqref{eq: 5to3_schauder} and~\eqref{eq: weak_elliptic2} and a Sobolev embedding.
\end{proof}

The remainder of the results in this subsection handle the pressure term in~\eqref{eq: whole_Schauder}. We start with an energy estimate for the velocity in the physical domain $\Omega$.  
\begin{lemma}\label{lemma: V_H1}
    In the setting of Proposition~\ref{uniform_prop}, $V \colon \Omega \to \Real^2$ defined by~\eqref{eq:stream_function} and~\eqref{eq: sf_p_cov} satisfies
    \begin{equation}\label{eq: V_H1}
        \lVert V \rVert_{H^1(\Omega)} \leq \lVert V \rVert_{L^2(\Omega)} + \sup_{\Surf} \lvert H \rvert^{1/2} \lVert V \rVert_{L^2(\Surf)} + C(\lvert \gamma \rvert, \lvert \Omega \rvert).
    \end{equation}
\end{lemma}
\begin{proof}
    Using~\eqref{eq:stream_function},~\eqref{eq: image_of_f} and~\eqref{eq: sf_p_cov} there exist $\Omega \in C^{4,\alpha}$, $V \in C^{3,\alpha}(\Omega)$ and $P \in C^{2,\alpha}(\Omega)$ satisfying~\eqref{eq:dimensionless}. For the following computation, it is more convenient to use the lab velocity $U = V + \gamma e_X$ instead of the travelling velocity $V$. The system solved by $(U,P)$ in $\Omega$ is
    \begin{subequations}
        \begin{align}
            \Rey ((U - \gamma e_X) \cdot \nabla) U - \Delta U + \nabla P &= 2 e_X - 2 \cot \theta e_Y &&\text{in } \Omega, \label{eq: lab_ns}\\
            \nabla \cdot U &= 0 &&\text{in } \Omega, \label{eq: lab_incomp}\\
            U \cdot \nu &= \gamma e_X \cdot \nu &&\text{on } \Surf, \label{eq: lab_kin}\\
            [(P + \sigma H)I - \D U] \nu &= 0   &&\text{on } \Surf, \label{eq: lab_dyn}\\
            U &= 0 &&\text{on } \Bed. \label{eq: lab_noslip}    
        \end{align}
    \end{subequations}
    We integrate~\eqref{eq: lab_ns} against $U$ over a fundamental period of $\Omega$,
    \begin{equation}\label{eq: integrate_U_NS}
        \int_\Omega U \cdot \big[\Rey \big((U - \gamma e_X) \cdot \nabla \big) U - \Delta U + \nabla P \big] \, dA = \int_\Omega 2 U_1 + 2 \cot \theta U_2 \, dA.
    \end{equation}
    Here, $dA = dXdY$ is an area element for $\Omega$. Writing $V = U - \gamma e_X$ for convenience, the first term on the left-hand side of~\eqref{eq: integrate_U_NS} is
    \begin{align*}
        \int_\Omega U \cdot \big(V \cdot \nabla \big) U \, dA = \frac{1}{2} \int_\Omega \big(V \cdot \nabla \big) \lvert U \rvert^2 \, dA =- \frac{1}{2} \int_\Omega \lvert U \rvert^2 \nabla \cdot V \, dA + \int_{\d \Omega} \lvert U \rvert^2 V \cdot \nu\, ds,
    \end{align*}
    which vanishes due to~\eqref{eq: lab_incomp},~\eqref{eq: lab_kin} and~\eqref{eq: lab_noslip}. Using the vector identities $\Delta U = \nabla \cdot \mathbb{D} U - \nabla (\nabla \cdot U)$, $\lvert \mathbb{D}U \rvert^2 = 2 DU : \mathbb{D} U$ and~\eqref{eq: lab_incomp}, the second and third terms of~\eqref{eq: integrate_U_NS} are given by 
    \begin{equation}\label{eq: integrate_U_stress}
        \int_\Omega U \cdot (\nabla P - \Delta U) \, dA = \int_\Omega U \cdot (\nabla P - \nabla \cdot \mathbb{D} U) \, dA = \dfrac{1}{2} \int_\Omega \lvert \mathbb{D} U \rvert^2 \, dA + \int_{\Surf} U \cdot ((P - \mathbb{D} U) \nu) \, ds.
    \end{equation}
    To compute the final integral in~\eqref{eq: integrate_U_stress} over $\Surf$, we apply~\eqref{eq: lab_dyn},~\eqref{eq: lab_kin} and the Frenet--Serret formula $H\nu = {d \tau}/{ds}$, where $\tau$ is the unit tangent to $S$ and $s$ is the arc-length parameter for $\Surf$,
    \begin{equation*}
        U \cdot ((P - \mathbb{D} U) \nu) = \sigma H U \cdot \nu = \sigma \gamma e_X \cdot \nu = e_X \cdot \frac{d \tau}{ds}.
    \end{equation*}
    This term therefore vanishes when integrated over $\Surf$ by periodicity,
    \begin{equation*}
        \int_{\Surf} e_X \cdot \frac{d \tau}{ds} \, ds = \int_{\Surf} \frac{d}{ds}(e_X \cdot \tau) \, ds = 0,
    \end{equation*}
    so~\eqref{eq: integrate_U_stress} reduces to
    \begin{equation}\label{eq: int_U_stress}
        \int_\Omega U \cdot (\nabla P - \Delta U) \, dA = \dfrac{1}{2} \int_\Omega \lvert \mathbb{D} U \rvert^2 \, dA.
    \end{equation}
    Meanwhile, the first term on the right-hand side of~\eqref{eq: integrate_U_NS} is
    \begin{equation}\label{eq: int_U1}
        \int_\Omega U_1 \, dA = \int_\Omega V_1 \, dA + \gamma \lvert \Omega \rvert = \frac{2 \pi}{K} \Big(\frac{2}{3} - \gamma \Big) + \gamma \lvert \Omega \rvert = \frac{4 \pi}{3 K} + \gamma \Big(\lvert \Omega \rvert - \frac{2 \pi}{K}\Big),
    \end{equation}
    where $K$ is the horizontal wavenumber of the wave profile of $\Surf$. For the final term in~\eqref{eq: integrate_U_NS}, we let $\Psi$ be defined by~\eqref{eq:stream_function} and compute
    \begin{equation*}
        \int_\Omega U_2 \, dA = - \int_\Omega \Psi_X \, dA = - \int_\Omega \nabla \cdot (\Psi e_X) \, dA = \int_{\d \Omega} \Psi (e_X \cdot \nu) \, ds = \Big(\dfrac{2}{3} - \gamma\Big) \int_{\Surf} e_X \cdot \nu \, ds,
    \end{equation*}
    which vanishes since 
    \begin{equation*}
        \int_{\Surf} e_X \cdot \nu \, ds = \int_\Omega \nabla \cdot e_X \, dA = 0.
    \end{equation*}
    Thus,~\eqref{eq: int_U_stress} and~\eqref{eq: int_U1} give the only nonzero terms in~\eqref{eq: integrate_U_NS}, and we have an $L^2$ energy identity for $V$,
    \begin{equation*}
        \dfrac{1}{2} \lVert \mathbb{D} V \rVert_{L^2(\Omega)}^2 = \dfrac{1}{2} \int_\Omega \lvert \mathbb{D} U \rvert^2 \, dA = \frac{8 \pi}{3 K} + 2 \gamma \Big(\lvert \Omega \rvert - \frac{2 \pi}{K}\Big).
    \end{equation*}

    Since $V$ is divergence-free with no-penetration boundary conditions on $\Surf$ and $\Bed$, it is easy to check that
    \begin{equation*}
        \frac{1}{2} \int_\Omega \lvert \mathbb{D}V \rvert^2 \, dA = \int_\Omega \lvert DV \rvert^2 \, dA + \int_\Surf H \lvert V \rvert^2 \; ds.
    \end{equation*}
    Consequently, we have an $L^2$ estimate for the gradient of $V$,
    \begin{equation*}
        \lVert DV \rVert_{L^2(\Omega)}^2 \leq \lVert \mathbb{D}V \rVert_{L^2(\Omega)}^2 + \sup_{\Surf} \lvert H \rvert \lVert V \rVert_{L^2(\d \Omega)}^2
        \leq  \frac{16 \pi}{3 K} + 4 \gamma \Big(\lvert \Omega \rvert - \frac{2 \pi}{K}\Big) + \sup_{\Surf} \lvert H \rvert \lVert V \rVert_{L^2(\Surf)}^2,
    \end{equation*}
    which implies~\eqref{eq: V_H1}.
\end{proof}
This $H^1$ estimate on $V$ yields an $H^2$ estimate for $\psi$ in $\Rect$ by the following corollary.
\begin{cor}\label{lemma: psi_H2}
    In the setting of Proposition~\ref{uniform_prop},
    \begin{equation*}
        \lVert \psi \rVert_{H^2(\Rect)} \leq C \Big(\delta, \sup_\Top \lvert D^2\eta \rvert, \sup_\Rect \lvert \nabla \psi \rvert \Big).
    \end{equation*}
\end{cor}
\begin{proof}
    By~\eqref{eq:stream_function} and~\eqref{eq: sf_p_cov}, suppressing dependence on $(x,y)$ we have
    \begin{equation*}
        \psi_x = \eta_x V_1(\xi, \eta) - \eta_y V_2(\xi, \eta), \quad \text{and} \quad
        \psi_y = \eta_y V_1(\xi, \eta) - \eta_x V_2(\xi, \eta).
    \end{equation*}
    Expressing second-order derivatives of $\psi$ similarly, we obtain
    \begin{equation}\label{eq: D2psi_L2}
        \lVert D^2 \psi \rVert_{L^2(\Rect)} \leq C (\delta, \lVert \eta \rVert_{C^2(\Rect)}) \lVert V \rVert_{H^1(\Omega)},
    \end{equation}
    where $\delta$ bounds the Jacobian in the change of variables. Now, Lemma~\ref{lemma: V_H1} implies
    \begin{equation*}
        \lVert V \rVert_{H^1(\Omega)} \leq C \Big(\delta, \lvert \gamma \rvert, \sup_\Top \lvert D^2\eta \rvert, \lVert \nabla \psi \rVert_{L^2(\Rect)} \Big),
    \end{equation*}
    after using that $\lvert \Omega \rvert \leq C(\delta)$ and 
    \begin{equation*}
        \sup_\Surf \lvert H \rvert \leq C(\delta) \sup_\Top \lvert D^2\eta \rvert.
    \end{equation*}
    Therefore,~\eqref{eq: D2psi_L2} becomes
    \begin{equation*}
        \lVert D^2 \psi \rVert_{L^2(\Rect)} \leq C \Big(\delta, \lvert \gamma \rvert, \sup_\Top \lvert D^2\eta \rvert, \lVert \nabla \psi \rVert_{L^2(\Rect)} \Big).
    \end{equation*}
    Together with the Poincar\'e inequality $\lVert \psi \rVert_{H^2(\Rect)} \leq C \lVert \nabla \psi \rVert_{L^2(\Rect)} + \lVert D^2 \psi \rVert_{L^2(\Rect)}$, this implies
    \begin{equation*}
        \lVert \psi \rVert_{H^2(\Rect)} \leq C (\delta, \lvert \gamma \rvert, \sup_\Top \lvert D^2\eta \rvert, \lVert \nabla \psi \rVert_{L^2(\Rect)}).
    \end{equation*}
    Finally, we can remove the dependence on $\gamma$ by~\eqref{eq: gamma_est} to obtain the result.
\end{proof}

With this estimate, we can finally control $\lVert p \rVert_{L^2(\Rect)}$ and establish Proposition~\ref{uniform_prop}.
\begin{proof}[Proof of Proposition~\ref{uniform_prop}]
    It suffices to show that $p$ satisfies
    \begin{equation}\label{eq: p_L2}
        \lVert p \rVert_{L^2(\Rect)} \leq C \Big(\delta, \sup_\Top \lvert D^2\eta \rvert, \sup_\Top \lvert D^2 \psi \rvert, \lVert \psi \rVert_{H^2(\Rect)} \Big),
    \end{equation}
    as Corollary~\ref{lemma: psi_H2} then implies
    \begin{equation*}
        \lVert p \rVert_{L^2(\Rect)} \leq C \Big(\delta, \sup_\Top \lvert D^2\eta \rvert, \sup_\Rect \lvert \nabla \psi \rvert, \sup_\Top \lvert D^2 \psi \rvert \Big),
    \end{equation*}
    which gives the result~\eqref{eq: best_est}. 
    
    From~\eqref{eq: flattened1},~\eqref{eq: flattened2} and the component~\eqref{eq: dyn_pressure} of~\eqref{eq: flattened_dyn}, $q := p + 2 \eta \cot \theta \in C^{2,\alpha}(\Rect)$ solves
    \begin{equation}\label{eq: p_poisson}
        \begin{cases}
            \Delta q = f, & \text{in } \Rect, \\
            q = g, & \text{on } T, \\
            q_y = \d_x h, & \text{on } \Bottom,
        \end{cases}
    \end{equation}
    with data $f, g, h$ given by
    \begin{align*}
        f &:=  2 \Rey \lvert \nabla \eta \rvert^2 \big[(\tilde{\d}_{XY} \psi)^2 - \tilde{\d}_{XX} \psi \tilde{\d}_{YY} \psi \big],\\
        g &:= -\dfrac{2}{\lvert \nabla \eta \rvert^2} \psi_{xy} - \dfrac{\sigma}{\lvert \nabla \eta \rvert^2}\big(\eta_y \eta_{xx} - \eta_x \eta_{xy}\big) + \dfrac{2 \psi_y}{\lvert \nabla \eta \rvert^4} \big(\eta_y \eta_{xy} + \eta_x \eta_{xx}\big) + 2 \eta \cot \theta, \\
        h &:= - \dfrac{1}{\eta_y^2} \psi_{yy}.
    \end{align*}
    This Poisson problem~\eqref{eq: p_poisson} admits a basic elliptic estimate for $q$, 
    \begin{equation}\label{eq: shifted_p_L2}
        \lVert q \rVert_{L^2(\Rect)} \leq C \big(\lVert f \rVert_{H^{-2}(\Rect)} + \lVert g \rVert_{H^{-1/2}(\Top)} + \lVert h \rVert_{H^{-1/2}(\Bottom)} \big).
    \end{equation}
    Here,
    \begin{align*}
        \lVert f \rVert_{H^{-2}(\Rect)} \leq \lVert f \rVert_{L^1(\Rect)} \leq C \big\lVert \lvert \nabla \eta \rvert^2 \big[(\tilde{\d}_{XY} \psi)^2 - \tilde{\d}_{XX} \psi \tilde{\d}_{YY} \psi \big] \big\rVert_{L^1(\Rect)} \leq C \Big(\lvert \gamma \rvert, \sup_\Top \lvert D^2\eta \rvert, \lVert \psi \rVert_{H^2(\Rect)}^2 \Big),
    \end{align*}
    and the boundary terms are controlled by
    \begin{align*}
        \lVert g \rVert_{H^{-1/2}(\Top)} &\leq C \Big(\delta, \lvert \cot \theta \rvert, \sup_\Top \lvert D^2\eta \rvert, \sup_\Top \lvert D^2 \psi \rvert \Big), \\
        \lVert h \rVert_{H^{-1/2}(\Bottom)} &\leq C \Big(\delta, \sup_\Top \lvert D^2 \psi \rvert \Big).
    \end{align*}
    Therefore,~\eqref{eq: shifted_p_L2} gives
    \begin{align*}
        \lVert q \rVert_{L^2(\Rect)} \leq  C \Big(\lvert \gamma \rvert, \sup_\Top \lvert D^2\eta \rvert, \lVert \psi \rVert_{H^2(\Rect)}^2 \Big) + C \Big(\delta, \lvert \cot \theta \rvert, \sup_\Top \lvert D^2\eta \rvert, \sup_\Top \lvert D^2 \psi \rvert \Big) + C \Big(\delta, \sup_\Top \lvert D^2 \psi \rvert \Big),
    \end{align*}
    from which we obtain~\eqref{eq: p_L2}.
\end{proof}

\section{Proof of main results}\label{section: proof_main}
\subsection{General existence theory}
We are now ready to prove a more detailed version of Theorem~\ref{thm: general_local}. Recall again the definition of $\F \colon \U \to \Y$ in~\eqref{eq: F} between spaces defined in~\eqref{eq: spaces}.
\begin{theorem}[Local bifurcation]\label{thm: gen_local}
    Fix $k > 0$ and $\Rey \geq 0$, and suppose that for some $\theta_* \in (0,\pi/2]$, $\gamma_* \in (0,\infty)\setminus\{2\}$, hypotheses~\ref{assumption: os_kernel} and~\ref{assumption: os_transv} of Theorem~\ref{thm: general_local} hold. 
    Then there exist $\eps > 0$ and a curve
    \begin{equation*}
        \mathscr{C}_{\mathrm{loc}} := \left\{\left(U(s), \Theta(s), \Gamma(s)\right) : |s| < \eps \right\} \subset \F^{-1}(0),
    \end{equation*}
    with the following properties:
    \begin{enumerate}[label=\rm(\alph*)]
        \item\label{gen_local_existence} The functions $U \colon (-\eps, \eps) \to \U$, $\Theta \colon (-\eps, \eps) \to (0, \pi)$ and $\Gamma \colon (-\eps, \eps) \to (0, \infty)$ are real-analytic, with asymptotic expansions in $s$,
        \begin{align}
            U(s) &= u_0(\theta_*, \gamma_*) + s \xi_0 + O(s^2), \label{eq: U_exp} \\
            (\Theta(s), \Gamma(s)) &= (\theta_*, \gamma_*) + O(s^2), \label{eq: Lambda_exp}
        \end{align}
        where $\xi_0 = \mathcal{Q} \mathscr{E} \hat{\varphi}_*$ for $\mathcal{Q}$ defined in~\eqref{eq: cov_matrix} and $\mathscr{E}$ defined in~\eqref{eq: os_extension};
        \item\label{gen_local_uniqueness} There exists an open set $\mathcal{N} \subset \U$ such that
        \begin{equation*}
            \mathscr{C}_{\mathrm{loc}} = \{(u, \theta, \gamma) \in \mathcal{N} : \F(u, \theta, \gamma) = 0, u \neq u_0(\theta, \gamma)\}.
        \end{equation*}
    \end{enumerate}
\end{theorem}
\begin{proof}
    We run the argument of Theorem~\ref{bif_lemma} with two parameters, i.e.~$n=2$, for the function $\mathcal{G} \colon \mathscr{V} \to \Y$ defined by 
    \begin{equation}\label{eq: perturbed_function}
        \mathcal{G}(w, \lambda) := \F(u_0(\lambda) + w, \lambda).
    \end{equation}
    Here $\lambda := (\theta, \gamma) \in (0, \pi) \times \Real$, $u_0(\lambda)$ is the Nusselt shear solution defined in~\eqref{eq: u0}, and $\mathscr{V}$ is obtained from $\U$ in the obvious way. Clearly for all $\lambda \in (0, \pi) \times (0, \infty)$,
    \begin{equation}\label{eq: trivial_line}
        \mathcal{G}(0, \lambda) = \F(u_0(\lambda), \lambda) = 0,
    \end{equation}
    satisfying the first hypothesis~\ref{H1} of Theorem~\ref{bif_lemma}.

    Proposition~\ref{Index_zero} implies that
    \begin{equation*}
        \mathcal{G}_w(0, \lambda) = \F_u(u_0(\lambda), \lambda) = \mathscr{L} (\lambda)
    \end{equation*}
    is a Fredholm operator with index $-1$ for any $\lambda \in (0, \pi) \times (0, \infty)$. Letting $\lambda_0 := (\theta_*, \gamma_*)$, by Lemma~\ref{Kernels_equivalence}, assumption~\ref{assumption: os_kernel} implies that
    \begin{equation*}
        \ker \mathcal{G}_w(0, \lambda_0) = \ker \mathscr{L}(\lambda_0) = \operatorname{span} \{\xi_0\},
    \end{equation*}
    so hypothesis~\ref{H2} of Theorem~\ref{bif_lemma} also holds. 
    
    By Lemma~\ref{transv_proj}, assumption~\ref{assumption: os_transv} implies that 
    \begin{equation*}
        (D_{\theta} \mathscr{L}(\lambda_0) \dot{\theta}) \xi_0 + (D_{\gamma} \mathscr{L}(\lambda_0) \dot{\gamma}) \xi_0 \notin \operatorname{ran} \mathscr{L}(\lambda_0)
    \end{equation*}
    for any nonzero $(\dot{\theta}, \dot{\gamma}) \in \Real^2$.  Now $D_{\theta} \mathscr{L}(\lambda_0) = D_{w \theta} \mathcal{G}(0, \lambda_0)$ and $D_{\gamma} \mathscr{L}(\lambda_0) = D_{w \gamma} \mathcal{G}(0, \lambda_0)$, so the final hypothesis~\ref{H3} of Theorem~\ref{bif_lemma} holds, namely that for any nonzero $\mu \in \Real^2$,
    \begin{equation}\label{eq: G_transverse}
        D_{w \lambda} \mathcal{G}(0, \lambda_0) [\xi_0, \mu] \notin \operatorname{ran} \mathcal{G}_w (0, \lambda_0).
    \end{equation} 

    Therefore, Theorem~\ref{bif_lemma} implies the existence of analytic functions $W \colon (-\eps, \eps) \to \U$, $\Theta \colon (-\eps, \eps) \to (0, \pi)$ and $\Gamma \colon (-\eps, \eps) \to (0, \infty)$, satisfying $W(0) = 0$, $W'(0) = \xi_0$ and $(\Theta(0), \Gamma(0)) = \lambda_0$ such that 
    \begin{equation}\label{eq: G_curve}
        \mathcal{G}(W(s), \Theta(s), \Gamma(s)) = 0.
    \end{equation}
    Setting 
    \begin{equation}\label{eq: U_def}
        U(s) := u_0(\Theta(s), \Gamma(s)) + W(s),
    \end{equation}
    we have $\F(U(s), \Theta(s), \Gamma(s)) = \mathcal{G}(W(s),\Theta(s), \Gamma(s))=0$ for all $s \in (-\eps, \eps)$.

    To establish the expansions~\eqref{eq: U_exp} and~\eqref{eq: Lambda_exp}, we show that $\Theta'(0) = \Gamma'(0) = 0$ and $U'(0) = W'(0)$. Differentiating~\eqref{eq: G_curve} twice with respect to $s$ at $s=0$ and using~\eqref{eq: trivial_line}, we obtain
    \begin{equation} \label{eq: curve''}
        \mathcal{G}_{ww}(0,\lambda_0)[\dot{w}, \dot{w}] + 2 \mathcal{G}_{w \lambda}(0,\lambda_0) [\dot{w}, \dot{\lambda}] + \mathcal{G}_w(0,\lambda_0)[\ddot{w}] = 0,
    \end{equation}
    where $\dot{w} := W'(0) = \xi_0$, $\dot{\lambda} := (\Theta'(0), \Gamma'(0))$ and $\ddot{w} := U''(0)$. Now we utilise the translational symmetry of the problem. Denoting by $\tau w$ the translation of $w$ in the $x$ direction by $\pi /(2k)$,~\eqref{eq: G_curve} implies that
    \begin{equation*}
        \mathcal{G}(\tau W(s), \Theta(s), \Gamma(s)) = 0.
    \end{equation*}
    By differentiating this equation twice as above, we obtain
    \begin{equation*}
        \mathcal{G}_{ww}(0,\lambda_0)[\tau \dot{w}, \tau \dot{w}] + 2 \mathcal{G}_{w \lambda}(0,\lambda_0) [\tau \dot{w}, \dot{\lambda}] + \mathcal{G}_w(0,\lambda_0)[\tau \ddot{w}] = 0.
    \end{equation*}
    Using $\tau \dot{w} = - \dot{w}$, this becomes
    \begin{equation}\label{eq: tr_curve''}
        \mathcal{G}_{ww}(0,\lambda_0)[\dot{w}, \dot{w}] - 2 \mathcal{G}_{w \lambda}(0,\lambda_0) [\dot{w}, \dot{\lambda}] + \mathcal{G}_w(0,\lambda_0)[\tau \ddot{w}] = 0.
    \end{equation}  
    Subtracting~\eqref{eq: tr_curve''} from~\eqref{eq: curve''} yields
    \begin{equation*}
        4 \mathcal{G}_{w \lambda}(0,\lambda_0) [\dot{w}, \dot{\lambda}] + \mathcal{G}_w(0,\lambda_0)[\ddot{w} - \tau \ddot{w}] = 0.
    \end{equation*}
    The transversality condition~\eqref{eq: G_transverse} then implies $\dot{\lambda} = (\Theta'(0), \Gamma'(0)) = 0$, verifying~\eqref{eq: Lambda_exp}. Differentiating~\eqref{eq: U_def} at $s = 0$ and using $(\Theta'(0), \Gamma'(0)) = 0$ we find that $U'(0) = W'(0) = \xi_0$ which confirms~\eqref{eq: U_exp}. 
    
    This completes the first claim~\ref{gen_local_existence}. The uniqueness claim~\ref{gen_local_uniqueness} follows from the corresponding uniqueness statement~\eqref{eq: bif_uniqueness} in Theorem~\ref{bif_lemma}.
\end{proof}

Theorem~\ref{thm: gen_local} implies Theorem~\ref{thm: general_local}, with the physical velocity $V(s)$, pressure $P(s)$ and domain $\Omega(s)$ recovered by~\eqref{eq:stream_function},~\eqref{eq: sf_p_cov} and~\eqref{eq: image_of_f}. To demonstrate the expansions in Remark~\ref{solution_expansions}, we compute the components of~\eqref{eq: U_exp} to find expansions for $\psi(s)$, $p(s)$ and $\eta(s)$ along the curve. By the definiton of $\mathcal{Q}$ and $\mathscr{E}$,
\begin{equation*}
    \xi_0 = \mathscr{Q} \mathscr{E} \hat{\varphi} = \operatorname{Re} \Bigg\{ \begin{pmatrix}
    \hat{\varphi}_* + \PsiNusselt' \hat{\zeta}_* \\ \hat{q}_* + P_0' \hat{\zeta}_* \\ \hat{\zeta}_*
    \end{pmatrix} e^{ikx} \Bigg\}.
\end{equation*}
Now, $\hat{\zeta}_*(1) \in \Real$ by the definition of $\X$, which implies $\hat{\zeta}_*(y) \in \Real$ for all $y \in [0,1]$, so
\begin{align*}
    \psi(x, y; s) &= \PsiNusselt(y) + s \PsiNusselt'(y) \hat{\zeta}_*(y) \cos kx + s \operatorname{Re} \big\{ \hat{\varphi}_*(y) e^{ikx} \big\} + O(s^2), \\
    p(x, y; s) &= P_0(y) + s P_0'(y) \hat{\zeta}_*(y) \cos kx + s \operatorname{Re} \big\{ \hat{q}_*(y) e^{ikx} \big\} + O(s^2), \\
    \eta(x, y; s) &= y + s \hat{\zeta}_*(y) \cos kx + O(s^2).
\end{align*}
Thinking of the explicit polynomial functions $\PsiNusselt$ and $P_0$ as defined on all of $\Real$, we can write these expansions as
\begin{align*}
    \psi(x, y; s) &= \PsiNusselt(\eta(x, y; s)) + s \operatorname{Re} \big\{ \hat{\varphi}_*(y) e^{ikx} \big\} + O(s^2), \\
    p(x, y; s) &= P_0(\eta(x, y; s)) + s \operatorname{Re} \big\{ \hat{q}_*(y) e^{ikx} \big\} + O(s^2).
\end{align*}
Writing $Y = \eta(x, y; s) = y + O(s)$ and $X = x + O(s)$, and scaling so that $\hat{\zeta}_*(1) = 1$, we obtain~\eqref{eq: Y_exp} and~\eqref{eq: VP_exp}.

We now state and prove the main global existence result, which is a more precise version of Theorem~\ref{thm: general_global}.
\begin{theorem}[Global bifurcation]\label{thm: gen_global}
    Let $\mathscr{C}_{\mathrm{loc}}$ be the local curve of solutions furnished by Theorem~\ref{thm: gen_local}. Then $\mathscr{C}_{\mathrm{loc}}$ is contained in a continuous curve of solutions,
    \begin{equation*}
        \mathscr{C} := \left\{\left(U(s), \Theta(s), \Gamma(s)\right) : s \in \Real \right\} \subset \U.
    \end{equation*} 
    Near any point along $\mathscr{C}$, we can re-parameterise $\mathscr{C}$ so that $(U(\cdot), \Theta(\cdot), \Gamma(\cdot))$ is real analytic. In addition, one of the following alternatives occur:
    \begin{enumerate}[label=\rm(\alph*)]
        \item \textup{(Blow-up)}\label{blowup_alt_psi} Defining
        \begin{equation*}
            \mathscr{N}(\psi, p, \eta, \theta) := \sup_{\substack{z_1, z_2 \in \Top \\ z_1 \neq z_2}} \frac{\lvert z_1 - z_2 \rvert}{\lvert f(z_1) - f(z_2) \rvert} + \sup_\Top (\lvert \nabla \eta \rvert + \lvert D^2 \eta \rvert + \lvert D^2 \psi \rvert) + \sup_\Rect \lvert \nabla \psi \rvert + \frac{1}{\inf_\Top \lvert \nabla \psi\rvert + \sigma} + \lvert \cot \theta\rvert,
        \end{equation*}
        then as $s \to \infty$ and $s \to -\infty$, 
        \begin{equation*} 
            \mathscr{N}(U(s), \Theta(s)) \longrightarrow \infty.
        \end{equation*}
        \item \textup{(Closed loop)} For some $T > 0$ we have $(U(T), \Theta(T), \Gamma(T)) = (u_0(\theta_*, \gamma_*), \theta_*, \gamma_*)$, and (possibly after re-parameterisation)
        \begin{equation*}
            (U(s+T), \Theta(s+T), \Gamma(s+T)) = (U(s), \Theta(s), \Gamma(s)),
        \end{equation*}
        for all $s > 0$.
    \end{enumerate}
\end{theorem}
\begin{proof}
    By Proposition~\ref{Schauder}, $\F_u(u,\theta, \gamma)$ is a semi-Fredholm operator with closed range and finite-dimensional kernel, for any $(u, \theta, \gamma) \in \U$, so by Theorem~\ref{theorem: gbt} we have a global continuation of the curve as claimed, with the alternatives~\ref{blowup_comp_alt}\ref{abstract_blowup_alt},~\ref{blowup_comp_alt}\ref{abstract_comp_alt} and~\ref{abstract_loop_alt} listed in that theorem. In the rest of the proof, we rule out the loss of compactness alternative~\ref{blowup_comp_alt}\ref{abstract_comp_alt} and improve the blow-up alternative~\ref{blowup_comp_alt}\ref{abstract_blowup_alt}.
    
    We eliminate~\ref{blowup_comp_alt}\ref{abstract_comp_alt} by the following argument. Without loss of generality, suppose we have a sequence $s_n \to \infty$ with $\sup_n N(s_n) \leq M$, for $N(s)$ defined as in~\eqref{eq: blowup_quantity}, for some $M > 0$. Using Proposition~\ref{uniform_prop}, $U(s_n) \in C^{5,\alpha}(\Rect) \times C^{3,\alpha}(\Rect) \times C^{5,\alpha}(\Rect)$, and
    \begin{equation*}
        \lVert U(s_n)\rVert_{C^{5,\alpha}(\Rect) \times C^{3,\alpha}(\Rect) \times C^{5,\alpha}(\Rect)} \leq C (N(s_n)) \leq C(M).
    \end{equation*}
    Then by Arzel\'a--Ascoli, we extract a convergent subsequence
    \begin{equation*}
        \left(U(s_n), \Theta(s_n), \Gamma(s_n)\right) \to (U, \Theta, \Gamma) \in C^{2, \alpha}(\Rect) \times C^{2, \alpha}(\Rect) \times C^{2, \alpha}(\Rect) \times \Real^2
    \end{equation*}
    and $U \in \X$ since $\X$ is closed in $C^{4,\alpha}(\Rect) \times C^{2,\alpha}(\Rect) \times C^{4,\alpha}(\Rect)$. Since
    \begin{equation*}
        \inf_n \operatorname{dist}\left((U(s_n),\lambda), \d \U\right) > M^{-1},
    \end{equation*}
    we have $(U, \Theta, \Gamma) \in \U$. Thus, alternative~\ref{blowup_comp_alt}\ref{abstract_comp_alt} does not occur.
    
    It remains to reduce~\ref{blowup_comp_alt}\ref{abstract_blowup_alt} of Theorem~\ref{theorem: gbt} to~\ref{blowup_alt_psi}.
    Suppose $(u,\lambda) = (\psi, p, \eta, \theta, \gamma) \in \U$ and that there exists $M > 0$ such that $\mathscr{N} (u, \theta) \leq M$. Then $(u,\lambda)$ satisfies~\eqref{eq: uniform_inequalities} with $\delta = M^{-1}$, so by Lemma~\ref{uniform_dist_dU}, $\operatorname{dist}\left((u,\lambda), \d \U\right) > C(M^{-1})$, while Proposition~\ref{uniform_prop} implies $\lVert u \rVert_\X \leq C(M)$. Noting that by~\eqref{eq: gamma_est}, $\lvert \gamma \rvert \leq C(M)$, we have
    \begin{equation*}
        \lVert u \rVert_\X + \lvert \lambda \rvert + \frac{1}{\operatorname{dist}\left((u,\lambda), \d \U\right)} \leq C(M).
    \end{equation*}
    Thus,~\ref{blowup_comp_alt}\ref{abstract_blowup_alt} of Theorem~\ref{theorem: gbt} implies $\mathscr{N}(U(s), \Theta(s)) \longrightarrow \infty$, and we are done.
\end{proof}
The blow-up alternative~\ref{blowup_alt_psi} in Theorem~\ref{thm: gen_global} is equivalent to the blow-up alternative~\ref{blowup_alternative} of Theorem~\ref{thm: general_global} by~\eqref{eq: image_of_f} and~\eqref{eq: cauchy_riemann} in the definition of $f$ and the relation of $V$ to $\psi$ and $\eta$ through~\eqref{eq:stream_function} and~\eqref{eq: sf_p_cov}

\subsection{Small wavenumber regime}
The aim of this subsection is to prove Theorem~\ref{thm: smallk} by verifying the hypotheses~\ref{assumption: os_kernel}--\ref{assumption: os_transv} of Theorem~\ref{thm: general_local}. We establish rigorous asymptotic expansions in $k$ for the kernel of the Orr--Sommerfeld operator and its adjoint, using analyticity to rule out higher-mode resonances in the kernel of $\OS$ for all but countably many $k$ in some neighbourhood of the origin. This confirms the kernel hypothesis~\ref{assumption: os_kernel}. The transversality condition~\ref{assumption: os_transv} is verified using the expansions. 

Fix a Reynolds number $\Rey \geq 0$. Direct calculation yields that at $\gamma = 2$, for any $\theta \in (0, \pi)$, the Orr--Sommerfeld operator $\OS(0, \Rey, \theta, 2)$ has a one-dimensional kernel spanned by $\hat{\varphi}_0 := y^2/2$. In the following result, we find a curve of angles and wavespeeds, parameterised by the wavenumber $k$, along which the Orr--Sommerfeld operator has a nontrivial kernel. This curve bifurcates from wavespeed $\gamma_0 = 2$ and angle $\theta_0 \in (0, \pi/2)$ defined by
\begin{equation}\label{eq: critical_pt}
    \cot \theta_0 = \frac{4}{5} \Rey,
\end{equation}
consistent with the critical Reynolds number for linear stability reported in~\cite{benjamin1957wave, yih1963stability, chang2002complex}.

\begin{lemma}\label{lemma: os_ker}
    Let $\gamma_0 = 2$ and $\theta_0$ be given by~\eqref{eq: critical_pt}. Then there exist $\eps > 0$ and analytic functions $\tilde{\theta} \colon (-\eps, \eps) \to (0,\pi)$, $\tilde{\gamma} \colon (-\eps, \eps) \to \Real$, and $\tilde{\varphi} \colon (-\eps, \eps) \to \Xos$ with the following properties:
    \begin{enumerate}[label=\rm(\alph*)]
        \item For $k \in (-\eps, \eps)$,
        \begin{equation*}
            \ker \OS(k, \Rey, \tilde{\theta}(k), \tilde{\gamma}(k)) = \operatorname{span}\{\tilde{\varphi}(k)\};
        \end{equation*}
        \item\label{smallk_expansions} The functions have the following expansions:
        \begin{align}
            \cot \tilde{\theta}(k) &= \frac{4}{5} \Rey + k^{2} \left(- \frac{\sigma}{2} - \frac{7743}{2800} \Rey + \frac{16}{1126125} \Rey^3\right)  + O(k^3), \label{eq: theta_exp}\\
            \tilde{\gamma}(k) &= 2 - 2 k^2 + O(k^3), \label{eq: gamma_exp}\\
            \tilde{\varphi}(y; k) &= \frac{y^2}{2} - \frac{1}{60} ik \Rey y^2(1-y)(2-y)^2 + O(k^2). \label{eq: kernel_exp}
        \end{align}
    \end{enumerate}
\end{lemma}
\begin{proof}
    The Orr--Sommerfeld operator $\OS(k, \Rey, \theta, \gamma)$ is Fredholm for all $k$, with a one-dimensional kernel at $k=0$, so by continuity, the dimension of its kernel is at most one for sufficiently small $k$. 
    To find a one-dimensional subspace in the kernel of $\OS$, we will bifurcate from the line of $k=0$ solutions by a Crandall--Rabinowitz argument for an auxiliary function $G$, which adds a normalisation condition. More precisely, we define the map $G \colon (0, \pi) \times \Real^2 \times \Xos\to \Yos \times \Real$ by
    \begin{equation}\label{eq: aux_G}
        G(\theta, k, \gamma, \hat{\varphi}) := \begin{pmatrix}
            \OS(k, \Rey, \theta, \gamma) \hat{\varphi} \\ 
            g(\hat{\varphi}) -1
        \end{pmatrix},
    \end{equation}
    where 
    \begin{equation}\label{eq: normalisation}
        g(\hat{\varphi}) := \sum_{j=0}^{3} \lvert \hat{\varphi}^{(j)}(0) \rvert^2.
    \end{equation}
    Clearly there is a correspondence between the zero set of $G$ and the kernel of $\OS$.

    We verify the hypotheses of the Crandall--Rabinowitz theorem, i.e.~Theorem~\ref{bif_lemma} with $n=1$. At zero $k$, we have a point $(\theta, 0, \gamma_0, \hat{\varphi}_0)$ in the zero set of $G$ for any $\theta \in (0, \pi)$, with $\gamma_0 = 2, \hat{\varphi}_0 = \tfrac{1}{2} y^2$. Now we consider the Fr\'echet derivative of $G$ with respect to $(k, \gamma, \varphi)$. Given $(f, c_1, c_2) \in \Yos$ and $c_3 \in \Real$, suppose
    \begin{equation}\label{eq: deriv_G}
        D_{(k, \gamma,\hat{\varphi})} G \left[(\theta, 0,\gamma_0, \hat{\varphi}_0) \right] \begin{pmatrix}
            \dot{k} \\ \dot{\gamma} \\ \dot{\phi}
        \end{pmatrix} = 
        \begin{pmatrix}
            \dot{\phi}^{(4)} - 2 i \dot{k} \Rey (y-1) \\
            \dot{\phi}''(1) - 2 \dot{\phi}(1) + \dot{\gamma} \\
            \dot{\phi}'''(1) + i \dot{k} (\Rey - \cot \theta) \\
            2 \Re \dot{\phi}''(0).
        \end{pmatrix} = \begin{pmatrix}
            f \\ c_1 \\ c_2 \\ c_3
        \end{pmatrix}.
    \end{equation}
    The ODE has general solution in the space $\Xos$ given by 
    \begin{equation*}
        \dot{\phi}(y) = a_2 y^2 + a_3 y^3 + \frac{1}{12}i \dot{k} \Rey \left(\frac{1}{5} y^5 - y^4\right) + I^4 f(y),
    \end{equation*}
    where $a_2, a_3 \in \Complex$ and $If(y)$ denotes the indefinite integral
    \begin{equation*}
        If(y) := \int^y_0 f(y_1) dy_1,
    \end{equation*}
    so that~\eqref{eq: deriv_G} is equivalent to following system for $a_2, a_3$ and $\dot{k}, \dot{\gamma} \in \Real$, 
    \begin{equation}\label{eq: system_for_frechet}
        \begin{aligned}
            4 a_3 - \frac{8i \dot{k} \Rey}{15} + \dot{\gamma} &= c_1 + 2 I^4 f(1) - I^2 f(1), \\
            6 a_3 - i \dot{k} \cot \theta &= c_2 - I^4 f(1), \\
            4 \Re a_2 &= c_3 - 2I^2 f(0), \\
            \operatorname{Im} \{a_2 + a_3 \} - \frac{\dot{k} \Rey}{15} &= - \operatorname{Im} I^4 f(1).
        \end{aligned}
    \end{equation}
    By eliminating $a_2$, $\Re a_3$ and $\dot{\gamma}$, the solvability of~\eqref{eq: system_for_frechet} is equivalent to a real-valued matrix problem,
    \begin{equation}\label{eq: frechet_matrix_system}
        \begin{pmatrix}
            4 & -\frac{8 \Rey}{15}\\
            6 & -\cot \theta
        \end{pmatrix} 
        \begin{pmatrix}
            \operatorname{Im} a_3 \\ \dot{k}
        \end{pmatrix} 
        = 
        \begin{pmatrix}
            \tilde{c}_1 \\ 
            \tilde{c}_2
        \end{pmatrix},
    \end{equation}
    for some $\tilde{c}_1, \tilde{c}_2 \in \Real$. This matrix has determinant $\frac{4}{5} \left(4 \Rey - 5\cot \theta\right)$, and precisely at $\theta_0$ given by~\eqref{eq: critical_pt}, the Fr\'echet derivative has a one-dimensional kernel spanned by 
    \begin{equation}\label{eq: smallk_frechet_kernel}
        \dot{k} = 1, \quad \dot{\gamma} = 0, \quad \dot{\phi} = - \frac{1}{60} i \Rey y^2(1-y)(2-y)^2.
    \end{equation}
    
    The final condition~\ref{H3} of Theorem~\ref{bif_lemma} is equivalent to the statement
    \begin{equation*}
        \d_\theta D_{(k, \gamma,\hat{\varphi})} G \left(\theta_0, 0, \gamma_0, \hat{\varphi}_0\right) \begin{pmatrix}
            \dot{k} \\ \dot{\gamma} \\ \dot{\phi}
        \end{pmatrix} \notin \operatorname{ran} D_{(k, \gamma,\hat{\varphi})} G \left(\theta_0, 0, \gamma_0, \hat{\varphi}_0\right)
    \end{equation*}
    with this choice of $\dot{k}$, $\dot{\gamma}$, $\dot{\phi}$. We calculate
    \begin{equation*}
        \d_\theta D_{(k, \gamma,\hat{\varphi})} G \left(\theta_0, 0, \gamma_0, \hat{\varphi}_0\right) \begin{pmatrix}
            \dot{k} \\ \dot{\gamma} \\ \dot{\phi}
        \end{pmatrix} = \bigg(1 + \dfrac{16}{25}\Rey^2 \bigg) \begin{pmatrix}
            0 \\ 0 \\ i \\ 0
        \end{pmatrix}.
    \end{equation*}
    Substituting $(f, c_1, c_2, c_3) = (0, 0, i, 0)$ in~\eqref{eq: system_for_frechet} yields
    \begin{equation*}
        4 \operatorname{Im} a_3 - \frac{8 \dot{k} \Rey}{15} = 0, \quad 6 \operatorname{Im} a_3 - \frac{4 \dot{k} \Rey}{5} = 1,
    \end{equation*}
    a contradiction. Therefore, the hypothesis~\ref{H3} holds and Theorem~\ref{bif_lemma} yields an $\eps_1 > 0$ and a nontrivial curve in the zero set of $G$,
    \begin{equation*}
        \left\{\left(\tilde{\theta}(s), \tilde{k}(s), \tilde{\gamma}(s), \tilde{\varphi}(s)\right) : |s| < \eps_1 \right\}.
    \end{equation*}
    Now $\tilde{k}(0)=0$ and $\tilde{k}'(0) = 1$, so there exists $\eps > 0$, an interval $\mathcal{I} \subset (-\eps_1, \eps_1)$ and an analytic inverse function $\tilde{k}^{-1} \colon (-\eps, \eps) \to \mathcal{I}$. Thus, abusing notation, we may parametrise the curve by $k \in (-\eps, \eps)$,
    \begin{equation*}
        G\left(\tilde{\theta}(k), k, \tilde{\gamma}(k), \tilde{\varphi}(k)\right) = 0.
    \end{equation*}
    The first order terms in the expansions in~\ref{smallk_expansions} are given by~\eqref{eq: smallk_frechet_kernel}. Since the construction uses the implicit function theorem, higher order terms in $\cot \tilde{\theta}(k)$, $\tilde{\gamma}(k)$ and $\tilde{\varphi}(k)$ can be obtained in the usual way. Note that obtaining $O(k^2)$ terms in the expansions for $\cot \tilde{\theta}$ and $\tilde{\gamma}$ requires expanding equations to $O(k^3)$. We omit the details. 

    For later use, we record that again by Theorem~\ref{bif_lemma} there exists an open set $\mathscr{N} \subset (0, \pi) \times \Real^2 \times \Xos$, such that $(\theta_0, 0, \gamma_0, \hat{\varphi}_0) \in \mathscr{N}$, and for some $\eps_2 > 0$,
    \begin{equation}\label{eq: uniqueness_Gzeros}
        \left\{(\theta, k, \gamma, \hat{\varphi}) \in \mathscr{N} : G(\theta, k, \gamma, \hat{\varphi}) = 0, ~ k \neq 0 \right\} = \left\{\left(\tilde{\theta}(k), k, \tilde{\gamma}(k), \tilde{\varphi}(k)\right) : 0 < \lvert k \rvert < \eps_2 \right\}. \qedhere
    \end{equation} 
\end{proof}

This result verifies the first part~\eqref{eq: os_1dker} of~\ref{assumption: os_kernel}. We establish the second part~\eqref{eq: os_nonres}, a non-resonance condition, for sufficiently small wavenumbers apart from some countable set.
\begin{lemma}\label{lemma: nonres}
    In the setting of Lemma~\ref{lemma: os_ker}, for $k \in (0, \eps)$ outside some countable set,
    \begin{equation}\label{eq: nonres}
        \ker \OS(nk, \Rey, \tilde{\theta}(k), \tilde{\gamma}(k)) = \{0\} \quad \text{for every } n \in \mathbb{N} \setminus \{1\}.
    \end{equation}
\end{lemma}

\begin{proof}
    Fix $n \in \mathbb{N} \setminus \{1\}$. At each $k \in (-\eps, \eps)$, the first component of the Orr--Sommerfeld operator,
    \begin{equation*}
        \mathscr{L}_{\mathrm{os},1} (nk, \Rey, \tilde{\theta}(k), \tilde{\gamma}(k)) \hat{\varphi} = \left(\frac{d^2}{dy^2} - (nk)^2\right)^2 \hat{\varphi} - ink \Rey \PsiNusselt' \left(\frac{d^2}{dy^2} - (nk)^2\right) \hat{\varphi} + ink \Rey \PsiNusselt'''\hat{\varphi},
    \end{equation*}
    is a linear fourth order ODE operator with polynomial coefficients and nonzero constant leading order coefficient. Therefore, its kernel in $\extXos$ is spanned (in $\mathbb{C}$) by the solutions $\hat{\varphi}_1, \hat{\varphi}_2$ with initial conditions:
    \begin{align*}
        \hat{\varphi}_1 (0;k) = \hat{\varphi}_1' (0;k) = \hat{\varphi}_1'' (0;k) = 0, ~ \hat{\varphi}_1''' (0;k) = 1; \\
        \quad \hat{\varphi}_2 (0;k) = \hat{\varphi}_2' (0;k) = \hat{\varphi}_2''' (0;k) = 0, ~ \hat{\varphi}_2'' (0;k) = 1.
    \end{align*}
    Moreover, $k \mapsto \hat{\varphi}_i^{(j)} (1;k)$ are analytic with respect to $k$. 
    
    Now for $\hat{\varphi} = a \hat{\varphi}_1 + b \hat{\varphi}_2$ to be in the kernel of $\OS (nk, \Rey, \tilde{\theta}(k), \tilde{\gamma}(k))$, it must satisfy the boundary conditions at $y=1$:
    \begin{equation*}
        \begin{aligned}
            \mathscr{L}_{\mathrm{os},2} (nk, \Rey, \tilde{\theta}(k), \tilde{\gamma}(k)) \hat{\varphi}(k) &= 0, \\
            \mathscr{L}_{\mathrm{os},3} (nk, \Rey, \tilde{\theta}(k), \tilde{\gamma}(k)) \hat{\varphi}(k) &= 0.            
        \end{aligned}
    \end{equation*}
    This leads us to define a function $\Gamma_n \colon (-\eps, \eps) \to \Real$ by
    \begin{equation*}
        \Gamma_n (k) = \det \begin{pmatrix}
            \mathscr{L}_{\mathrm{os},2} \left(nk, \Rey, \tilde{\theta}(k), \tilde{\gamma}(k)\right) \hat{\varphi}_1 (k) & \mathscr{L}_{\mathrm{os},2} \left(nk, \Rey, \tilde{\theta}(k), \tilde{\gamma}(k)\right) \hat{\varphi}_2 (k)  \\
            \mathscr{L}_{\mathrm{os},3} \left(nk, \Rey, \tilde{\theta}(k), \tilde{\gamma}(k)\right) \hat{\varphi}_1 (k) & \mathscr{L}_{\mathrm{os},3} \left(nk, \Rey, \tilde{\theta}(k), \tilde{\gamma}(k)\right) \hat{\varphi}_2 (k)
        \end{pmatrix}.
    \end{equation*}
    Arguing as in~\eqref{eq: frechet_matrix_system}, this function vanishes if and only if there is some $\hat{\varphi}$ in the kernel of $\OS (nk, \Rey, \tilde{\theta}(k), \tilde{\gamma}(k))$. Clearly $\Gamma_n(0) = 0$, and Lemma~\ref{lemma: os_ker} implies $\Gamma_1(k) = 0$ for all $k \in (-\eps, \eps)$. 

    We show that for $n \neq 1$, $\Gamma_n$ is not identically zero. Supposing for contradiction that $\Gamma_n \equiv 0$, then for every $k \in (-\eps, \eps)$ there would exist $u(k) \in \Xos$ analytic in $k$ such that $\OS (nk, \tilde{\theta}(k), \tilde{\gamma}(k)) u(k) = 0$, with $u(0) = \hat{\varphi}_0$ and $g(u(k)) = 1$. Therefore, the curve $(nk, \tilde{\theta}(k), \tilde{\gamma}(k), u(k))$ is in the zero set of $G$ defined in~\eqref{eq: aux_G}, and passes through $(0, \theta_0, \gamma_0, \hat{\varphi}_0)$. By the uniqueness statement~\eqref{eq: uniqueness_Gzeros} in the proof of Lemma~\ref{lemma: os_ker}, the curve must coincide with $(k, \tilde{\theta}(nk), \tilde{\gamma}(nk), \tilde{\varphi}(nk))$ for $k$ sufficiently small. Now by~\eqref{eq: theta_exp}, $\tilde{\theta}(nk) \neq \tilde{\theta}(k)$ (or similarly $\tilde{\gamma}(nk) \neq \tilde{\gamma}(k)$), a contradiction. 
    
    Having shown that $\Gamma_n$ is an analytic function that is not identically zero, it can have only finitely many roots in $(-\eps, \eps)$, after shrinking $\eps$. Therefore, the countable union
    \begin{equation*}
        \bigcup_{n \in \mathbb{Z} \setminus \{-1, 1\}} \Gamma_n^{-1} (0),
    \end{equation*}
    which is the set of all $k \in (-\eps, \eps)$ with resonances, is also countable. Consequently,~\eqref{eq: nonres} holds for $k \in (-\eps, \eps)$ except for a countable number of points.
\end{proof}

In order to verify the transversality hypothesis~\ref{assumption: os_transv}, we analyse the kernel of the formal adjoint $\OS^*$ defined in~\eqref{eq: os_adjoint}, and conduct an identical argument to that in Lemma~\ref{lemma: os_ker} to obtain an analytic curve in the kernel of $\OS^*$ for small wavenumbers.
\begin{lemma}
    In the setting of Lemma~\ref{lemma: os_ker}, there exists an analytic function $\tilde{\varphi}^* \colon (-\eps, \eps) \to \Xos$ such that for all $k \in (-\eps, \eps)$,
    \begin{equation*}
        \ker \OS^*(k, \Rey,\tilde{\theta}(k), \tilde{\gamma}(k)) = \operatorname{span}\{\tilde{\varphi}^*(k)\}.
    \end{equation*}
    Moreover, we have the asymptotic expansion
    \begin{equation} \label{eq: adj_exp}
        \tilde{\varphi}^*(y; k) = \frac{y^{2}}{2} - \frac{y^{3}}{6} - \dfrac{1}{840} ik \Rey y^2(1-y)(2-y)^2(1-3(1-y)^2) + O(k^2).
    \end{equation}
\end{lemma}

\begin{proof}
    We know that at each $k$, by the Fredholm alternative in Lemma~\ref{adj_fred_alt}, there is a $\hat{\varphi}^* \in \Xos$ such that $\OS^*$ has a two-dimensional kernel $\operatorname{span}\{\hat{\varphi}^*, i\hat{\varphi}^*\} \subset \extXos$. However, we need to show that this $\hat{\varphi}^*$ depends analytically on $k$. To this end, we prove an analogue to Lemma~\ref{lemma: os_ker} for the kernel of $\OS^*$. 
    
    Similarly to the proof of Lemma~\ref{lemma: os_ker}, we define a map $H \colon (0, \pi) \times \Real \times \Real \times \Xos\to \Yos \times \Real$ by
    \begin{equation*}
        H(\theta, k, \gamma, \hat{\varphi}) := \begin{pmatrix}
            \OS^*(k, \Rey, \theta, \gamma) \hat{\varphi} \\ 
            g(\hat{\varphi})- 2
        \end{pmatrix},
    \end{equation*}
    for $g$ as defined in~\eqref{eq: normalisation}. Here the normalisation constant $2$ is chosen merely for convenience, to avoid square roots in~\eqref{eq: adj_exp}. There is a one-to-one correspondence between the zero set of $H$ and the kernel of $\OS^*$ in $\Xos$.
    At $k=0$, we have a point $(\theta, 0, \gamma_0, \hat{\varphi}_0)$ in the zero set of $H$, for any $\theta \in (0, \pi)$, with $\gamma_0 = 2$ and 
    \begin{equation*}
        \hat{\varphi}^*_0 := \frac{1}{6}(3y^2-y^3).
    \end{equation*}
    We will run again run Theorem~\ref{bif_lemma} with $n=1$ to find solutions of $H=0$ in a neighbourhood of each such point $(\theta_0, 0, \gamma_0, \hat{\varphi}^*_0)$. Suppose for $(f, c_1, c_2) \in \Yos$ and $c_3 \in \Real$,
    \begin{equation*}
        \d_{(k, \gamma,\hat{\varphi})} H \left[(\theta, 0,\gamma_0, \hat{\varphi}^*_0) \right] \begin{pmatrix}
            \dot{k} \\ \dot{\gamma} \\ \dot{\phi}
        \end{pmatrix} = 
        \begin{pmatrix}
            \dot{\phi}^{(4)} - i \dot{k} \Rey (2-8y+9y^2-3y^3) \\
            \dot{\phi}''(1) \\
            \dot{\phi}'''(1) + 2 \dot{\phi}'(1) + \frac{1}{30} i \dot{k} \Rey - \dot{\gamma} \\
            2 \Re \{\dot{\phi}''(0) - \dot{\phi}'''(0)\}
        \end{pmatrix} = \begin{pmatrix}
            f \\ c_1 \\ c_2 \\ c_3
        \end{pmatrix}.
    \end{equation*}
    The ODE has general solution in the space $\extXos$ given by
    \begin{equation*}
        \dot{\phi}(y) = a_2 y^2 + a_3 y^3 - i \dot{k} \Rey \left(\frac{y^{7}}{280} - \frac{y^{6}}{40} + \frac{y^{5}}{15} - \frac{y^{4}}{12}\right) + I^4 f (y),
    \end{equation*}
    with $a_2, a_3 \in \Complex$ and $\dot{k}, \dot{\gamma} \in \Real$ satisfying
    \begin{equation}\label{eq: adj_frechet_system}
            \begin{aligned}
            2 a_2 + 6 a_3 + \frac{4}{15} i \dot{k} \Rey &= c_1 - I^2 f(1), \\
            4 a_2 + 12 a_3 + \frac{8}{15} i \dot{k} \Rey - \dot{\gamma} &= c_2 + I f(1) - 2I^2 f(1), \\
            4 \Re a_2 - 12 \Re a_3 &= c_3 + \Re I f(0) - \Re I^2 f(0), \\
            \operatorname{Im} a_2 + \operatorname{Im} a_3 + \frac{4}{105} i \dot{k} \Rey &= -If(1).            
        \end{aligned}
    \end{equation}
    After solving for $\Re a_2$, $\Re a_3$ and $\dot{\gamma}$, the system is equivalent to the real-valued matrix problem
    \begin{equation*}
        \begin{pmatrix}
            2 & 6 & \frac{4 \Rey}{15} \\
            4 & 12 & \frac{8 \Rey}{15} \\
            1 & 1 & \frac{4 \Rey}{105}
        \end{pmatrix} 
        \begin{pmatrix}
            \operatorname{Im} a_2 \\ \operatorname{Im} a_3 \\ \dot{k}
        \end{pmatrix} 
        = 
        \begin{pmatrix}
            \tilde{c}_1 \\ 
            \tilde{c}_2 \\
            \tilde{c}_3
        \end{pmatrix},
    \end{equation*}
    for some $\tilde{c}_1, \tilde{c}_2, \tilde{c}_3 \in \Real$, which is clearly degenerate. The Fr\'echet derivative has one-dimensional kernel spanned by
    \begin{equation}\label{eq: adj_frechet_ker}
        \dot{k} = 1, \quad \dot{\gamma} = 0, \quad \dot{\phi} = - \dfrac{1}{840} k \Rey y^2(1-y)(2-y)^2(1-3(1-y)^2).
    \end{equation}
    The transversality condition~\ref{H3} of Theorem~\ref{bif_lemma} is equivalent to the statement
    \begin{equation*}
        \d_\theta D_{(k, \gamma,\hat{\varphi})} H \left(\theta_0, 0, \gamma_0, \hat{\varphi}_0\right) \begin{pmatrix}
            \dot{k} \\ \dot{\gamma} \\ \dot{\phi}
        \end{pmatrix}\notin \operatorname{ran} D_{(k, \gamma,\hat{\varphi})} H \left(\theta_0, 0, \gamma_0, \hat{\varphi}_0\right)
    \end{equation*}
    with this choice of $\dot{k}$, $\dot{\gamma}$, $\dot{\phi}$. Now 
    \begin{equation*}
        \d_\theta D_{(k, \gamma,\hat{\varphi})} H \left(\theta_0, 0, \gamma_0, \hat{\varphi}_0\right) \begin{pmatrix}
            \dot{k} \\ \dot{\gamma} \\ \dot{\phi}
        \end{pmatrix} = \dfrac{2}{3}\left(1+ \dfrac{16}{25} \Rey^2\right) \begin{pmatrix}
            0 \\ 0 \\ i \\ 0
        \end{pmatrix}.
    \end{equation*}
    Substituting $(f, c_1, c_2, c_3) = (0, 0, i, 0)$ in~\eqref{eq: adj_frechet_system} yields $\gamma = -i$, a contradiction. Therefore, by Theorem~\ref{bif_lemma}, there exists $\eps > 0$ and a nontrivial curve in the zero set of $H$,
    \begin{equation*}
        \left\{\left(\tilde{\theta}(k), k, \tilde{\gamma}(k), \tilde{\varphi}^*(k)\right) : |k| < \eps \right\},
    \end{equation*}
    where we have parameterised the curve by $k$ as in Lemma~\ref{lemma: os_ker}. The expansion~\eqref{eq: adj_exp} is given by~\eqref{eq: adj_frechet_ker}.

    Furthermore, there exists a neighbourhood $\mathscr{N} \subset (0, \pi) \times \Real^2 \times \Xos$ of $(\theta_0, 0, \gamma_0, \hat{\varphi}^*_0)$ and $\eps_1 > 0$ such that
    \begin{equation*}
        \Big\{(\theta, k, \gamma, \hat{\varphi}) \in \mathscr{N} : H(\theta, k, \gamma, \hat{\varphi}) = 0, ~ k \neq 0 \Big\} = \left\{\big(\tilde{\theta}(k), k, \tilde{\gamma}(k), \tilde{\varphi}^*(k)\big) : 0 <|k| < \eps_1 \right\}. \qedhere
    \end{equation*}
\end{proof}

Gathering these results, we prove Theorem~\ref{thm: smallk}. The only hypothesis remaining is the Orr--Sommerfeld transversality condition~\ref{assumption: os_transv}, which we verify using the asymptotic expansions~\eqref{eq: theta_exp},~\eqref{eq: gamma_exp},~\eqref{eq: kernel_exp} and~\eqref{eq: adj_exp}.
\begin{proof}[Proof of Theorem~\ref{thm: smallk}]
    Lemmas~\ref{lemma: os_ker} and~\ref{lemma: nonres} establish~\ref{assumption: os_kernel} in Theorem~\ref{thm: general_local}. To verify~\ref{assumption: os_transv}, we use Corollary~\ref{cor: transv_det}, and so we compute 
    \begin{equation*}
        \d_{\gamma} \OS\big(k, \Rey, \tilde{\theta}(k), \tilde{\gamma}(k)\big) \tilde{\varphi}(y; k) = \begin{pmatrix}
            ik \Rey \left(\tilde{\varphi}''(y;k) - k^2 \tilde{\varphi}(y;k) \right) \\
            2 \left(1-\tilde{\gamma}(k)\right)^{-2} \tilde{\varphi}(1; k) \\
            ik \left[ \Rey \tilde{\varphi}'(1;k) - \left(1-\tilde{\gamma}(k)\right)^{-2} (2\cot \tilde{\theta}(k) + \sigma k^2) \tilde{\varphi}(1;k)\right]
        \end{pmatrix},
    \end{equation*}
    which gives
    \begin{multline}\label{eq: smallk_matrix_entry1}
        \left\langle \tilde{\varphi}^*(k), \d_{\gamma} \OS \big(k, \Rey, \tilde{\theta}(k), \tilde{\gamma}(k)\big) \tilde{\varphi}(k) \right\rangle_{\mathrm{os}} = 
        \Re \bigg\{ ik \Rey\int_0^1 \overline{\tilde{\varphi}^*} \left(\tilde{\varphi}'' - k^2 \tilde{\varphi} \right) dy + 2 \overline{(\tilde{\varphi}^*)'(1)} \tilde{\varphi}(1)\left(1-\tilde{\gamma}(k)\right)^{-2} \\ - ik\overline{\varphi^*(1)} \left[\Rey \tilde{\varphi}'(1) - 2 \left(1-\tilde{\gamma}(k)\right)^{-2} \cot \tilde{\theta}(k) \tilde{\varphi}(1) \right] \bigg\}.
    \end{multline}

    Now from~\eqref{eq: gamma_exp}, we have
    \begin{equation*}
        \left(1-\tilde{\gamma}(k)\right)^{-2} = 1 + 4k^2 + O(k^4),
    \end{equation*}
    and by~\eqref{eq: kernel_exp}, $\tilde{\varphi}(1;k) = \frac{1}{2} + O(k)$. Note also that by~\eqref{eq: adj_exp}, $\tilde{\varphi}^*(1;k) = \frac{1}{3} + O(k)$ and $(\tilde{\varphi}^*)'(1;k) = \frac{1}{2} + O(k)$. Using these facts, we can simplify~\eqref{eq: smallk_matrix_entry1},
    \begin{equation*}
        \big\langle \tilde{\varphi}^*(k), \d_{\gamma} \OS \big(k, \Rey, \tilde{\theta}(k), \tilde{\gamma}(k)\big) \tilde{\varphi}(k) \big\rangle_{\mathrm{os}} = \Re \Big\{ 2 \overline{(\tilde{\varphi}^*)'(1;k)} \tilde{\varphi}(1;k)\left(1-\tilde{\gamma}(k)\right)^{-2} + O(k) \Big\}= \tfrac{1}{2} + O(k),
    \end{equation*}
    and clearly 
    \begin{equation*}
        \big\langle i\tilde{\varphi}^*(k), \d_{\gamma} \OS \big(k, \Rey, \tilde{\theta}(k), \tilde{\gamma}(k)\big) \tilde{\varphi}(k) \big\rangle_{\mathrm{os}} = \Re \Big\{ 2i \big(\tfrac{1}{2} + O(k)\big)\big(\tfrac{1}{2} + O(k)\big) + O(k) \Big\} = O(k).
    \end{equation*}
    Also, 
    \begin{equation*}
        \d_{\theta} \OS\big(k, \Rey, \tilde{\theta}(k), \tilde{\gamma}(k)\big) \tilde{\varphi}(y; k) = \begin{pmatrix}
            0 \\ 
            0 \\
            -\dfrac{2ik}{(1-\tilde{\gamma}(k)) \sin^2 \tilde{\theta}(k)} \tilde{\varphi}(1;k)
        \end{pmatrix},
    \end{equation*}
    where this third component can be expanded in $k$,
    \begin{equation*}
        \dfrac{2ik}{(\tilde{\gamma}(k)-1) \sin^2 \tilde{\theta}(k)} \tilde{\varphi}(1;k) = 2ik\big(1+2k^2\big)\bigg(1+\frac{16 \Rey^2}{25}+O(k^2)\bigg)\big(\tfrac{1}{2}+O(k)\big) = ik\left(1+\frac{16 \Rey^2}{25}\right) + O(k^2).
    \end{equation*}
    Thus, 
    \begin{equation*}
        \big\langle \tilde{\varphi}^*(k), \d_{\theta} \OS \big(k, \Rey, \tilde{\theta}(k), \tilde{\gamma}(k)\big) \tilde{\varphi}(k) \big\rangle_{\mathrm{os}} = O(k^2),
    \end{equation*}
    and 
    \begin{align*}
        \big\langle i\tilde{\varphi}^*(k), \d_{\theta} \OS \big(k, \Rey, \tilde{\theta}(k), \tilde{\gamma}(k)\big) \tilde{\varphi}(k) \big\rangle_{\mathrm{os}} &= \Re \bigg\{ -ik \overline{\left[i\varphi^*(1;k)\right]}\bigg(1+\frac{16 \Rey^2}{25}\bigg) + O(k^2)\bigg\} \\
        &= -\frac{1}{3} k \left(1+\frac{16 \Rey^2}{25}\right) + O(k^2).
    \end{align*}
    Therefore, we obtain that 
    \begin{align*}
        \det \begin{pmatrix}
            \langle \tilde{\varphi}^*, \d_{\gamma} \OS \tilde{\varphi} \rangle_{\mathrm{os}} 
            & \langle i\tilde{\varphi}^*, \d_{\gamma} \OS \tilde{\varphi} \rangle_{\mathrm{os}} \\
            \langle \tilde{\varphi}^*, \d_{\theta} \OS \tilde{\varphi} \rangle_{\mathrm{os}} 
            & \langle i\tilde{\varphi}^*, \d_{\theta} \OS \tilde{\varphi} \rangle_{\mathrm{os}} \end{pmatrix} &= \det \begin{pmatrix}
            \dfrac{1}{2} + O(k^2) &  O(k) \\
            O(k^2) & - \dfrac{k}{3}  \left(1 + \dfrac{16}{25} \Rey^2 \right) + O(k^2)
        \end{pmatrix} \\ &= - \dfrac{1}{6} k \left(1 + \dfrac{16}{25} \Rey^2 \right) + O(k^2),
    \end{align*}
    which is nonzero for small $k$, so by Corollary~\ref{cor: transv_det}, hypothesis~\ref{assumption: os_transv} holds.
\end{proof}

\subsection{Low Reynolds number regime}
We now consider the low Reynolds number regime, fixing an arbitrary wavenumber $k > 0$. Firstly in the vanishing Reynolds number limit, known as the `Stokes problem', we can satisfy the hypotheses of Theorem~\ref{thm: general_local}. 
\begin{lemma}\label{lemma: zeroR}
    Let $\Rey = 0$ and $\sigma = 0$. Then, defining
    \begin{equation}\label{eq: zeroR_disp}
        \theta_0 = \frac{\pi}{2}, \quad \gamma_0 = 1+ \frac{1}{k^2 + \cosh^2 k},
    \end{equation}
    the parameters $\theta_* = \theta_0$ and $\gamma_* = \gamma_0$ satisfy~\ref{assumption: os_kernel} and~\ref{assumption: os_transv} of Theorem~\ref{thm: general_local}, with $\ker \OS(k, 0, \theta_*, \gamma_*) = \operatorname{span} \{\hat{\varphi}_0\}$, where
    \begin{equation}\label{eq: zeroR_ker}
        \hat{\varphi}_0(y) = (y-1) \cosh k \sinh ky + ky \cosh (k(1-y)).
    \end{equation}
\end{lemma}
\begin{proof}
    To analyse the kernel of the Stokes operator, let $\hat{\varphi} \in \Xos$ be such that 
    \begin{subequations}
        \begin{align}
            \left(\dfrac{d^2}{dy^2} - k^2\right)^2 \hat{\varphi} &= 0 \label{eq: stokes_ode}\\
            \hat{\varphi}''(1) + \left(\dfrac{2}{1-\gamma} + k^2 \right) \hat{\varphi} (1) &= 0 \label{eq: stokes_bc1}\\
            \hat{\varphi}'''(1) - 3 k^2 \hat{\varphi}'(1) + \dfrac{ik}{1-\gamma}\alpha \hat{\varphi}(1) &= 0. \label{eq: stokes_bc2}
        \end{align}
    \end{subequations}
    where $\alpha := 2 \cot \theta + \sigma k^2$. Here, we do not enforce $\sigma = 0$ yet; this will be required at a later step in the proof. Solutions to~\eqref{eq: stokes_ode} in $\Xos$ are of the form 
    \begin{equation*}
        \hat{\varphi}(y) = A(\sinh(ky) - ky \cosh(ky)) + By \sinh(ky),
    \end{equation*}
   where $A, B \in \Complex$, and using~\eqref{eq: stokes_bc1} and~\eqref{eq: stokes_bc2}, we obtain a complex-valued matrix problem for $A$ and $B$,
    \begin{equation*}
        \begin{pmatrix}
            -2k^3\cosh k - \dfrac{2(k \cosh k - \sinh k)}{1-\gamma}  & 2k(k \sinh k + \cosh k) + \dfrac{2 \sinh k}{1-\gamma} \\
            2k^3 (k \sinh k - \cosh k) - \dfrac{i \alpha k (k \cosh k - \sinh k)}{1-\gamma} & -2k^3 \cosh k + \dfrac{i \alpha k \sinh k}{1-\gamma},
        \end{pmatrix} \begin{pmatrix}
            A \\ B
        \end{pmatrix} = \begin{pmatrix}
            0 \\ 0
        \end{pmatrix}.
    \end{equation*}
    The kernel of $\OS(k, 0, \theta, \gamma)$ is nontrivial if and only if the determinant of this matrix vanishes,
    \begin{equation*}
        4\left(\frac{1}{1-\gamma} + k^2+\cosh^2 k\right) + i \frac{(2k-\sinh 2k) (2\cot \theta + \sigma k^2)}{1-\gamma} = 0,
    \end{equation*}
    which demands $\sigma = 0$, $\theta = \theta_0$ and $\gamma = \gamma_0(k)$ defined in~\eqref{eq: zeroR_disp}, with kernel spanned by $\hat{\varphi}_0$ given by~\eqref{eq: zeroR_ker}. This implies~\eqref{eq: os_1dker}, and the nonresonance condition~\eqref{eq: os_nonres} follows from the monotonicity of $\gamma_0(k)$ for $k$ positive, so~\ref{assumption: os_kernel} holds at $\theta_* = \theta_0$ and $\gamma_* = \gamma_0(k)$.

    To verify transversality~\ref{assumption: os_transv}, we will use Corollary~\ref{cor: transv_det}. We first compute the kernel of the formally adjoint operator $\OS^* \colon \Xos \to \Yos$ defined in~\eqref{eq: os_adjoint}, at $\Rey = \sigma = 0$, $\theta = \theta_0$, $\gamma = \gamma_0(k)$. Explicitly, this operator is
    \begin{equation*}
        \OS^*(k, 0, \theta_0, \gamma_0(k)) \hat{\varphi}^* :=
        \begin{pmatrix}
            \left(\dfrac{d^2}{dy^2} - k^2\right)^2 \hat{\varphi}^* \\
            (\hat{\varphi}^*)''(1) + k^2 \hat{\varphi}^*(1) \\
            (\hat{\varphi}^*)'''(1) +\left(2 \cosh^2(k) - k^2\right) (\hat{\varphi}^*)'(1)
        \end{pmatrix},
    \end{equation*}
    and has kernel spanned by 
    \begin{equation*}
        \hat{\varphi}_0^*(y) = (\cosh k + k \sinh k) (\sinh (ky) - ky \cosh (ky)) + k^2 y \cosh k \sinh (ky). 
    \end{equation*}

    We then compute derivatives of the Orr--Sommerfeld operator which appear in~\eqref{eq: sing_det}, using that $\hat{\varphi}_0(1) = k$,
    \begin{align*}
        \d_\gamma \OS \Big(k, 0, \theta_0, \gamma_0(k)\Big) \hat{\varphi}_0 &= \begin{pmatrix}
            0 \\ \dfrac{2}{(1-\gamma_0(k))^2} \hat{\varphi}_0(1) \\ 0
        \end{pmatrix} = \begin{pmatrix}
            0 \\ 2k (k^2 + \cosh^2 k)^2  \\ 0
        \end{pmatrix}, \\
        \d_\theta \OS \Big(k, 0, \theta_0, \gamma_0(k)\Big) \hat{\varphi}_0 &= \begin{pmatrix}
            0 \\ 0 \\ -\dfrac{2ik}{1-\gamma_0(k)} \hat{\varphi}_0(1)
        \end{pmatrix} = \begin{pmatrix}
            0 \\ 0 \\ 2ik^2 (k^2 + \cosh^2 k)
        \end{pmatrix}.
    \end{align*}
    Using that $\hat{\varphi}_0^*(1) = (\sinh(2k) - 2k)/2$ and $(\hat{\varphi}_0^*)'(1) = k^3$, we then obtain
    \begin{align*}
        \left\langle \hat{\varphi}_0^*, \d_\gamma \OS \big(k, 0, \theta_0, \gamma_0(k)\big) \hat{\varphi}_0 \right\rangle_{\mathrm{os}} &= \Re \left\{ 2k (k^2 + \cosh^2 k)^2 \overline{(\hat{\varphi}_0^*)'(1)} \right\}
        = 2k^4 (k^2 + \cosh^2 k)^2, \\
        \left\langle i\hat{\varphi}_0^*, \d_\gamma \OS \big(k, 0, \theta_0, \gamma_0(k)\big) \hat{\varphi}_0 \right\rangle_{\mathrm{os}} &= 0, \\
        \left\langle i\hat{\varphi}_0^*, \d_\theta \OS \big(k, 0, \theta_0, \gamma_0(k)\big) \hat{\varphi}_0 \right\rangle_{\mathrm{os}} &= \Re \left\{ -2ik^2 (k^2 + \cosh^2 k) \overline{i\hat{\varphi}_0^*(1)}\right\} = k^2 (k^2 + \cosh^2 k)(2k - \sinh(2k)), \\
        \left\langle \hat{\varphi}_0^*, \d_\theta \OS \big(k, 0, \theta_0, \gamma_0(k)\big) \hat{\varphi}_0 \right\rangle_{\mathrm{os}} &= 0.
    \end{align*}
    Thus, dropping parameter dependences,
    \begin{equation*}
        \det \begin{pmatrix}
            \left\langle \hat{\varphi}_0^*, \d_\gamma \OS\hat{\varphi}_0 \right\rangle_{\mathrm{os}} & \left\langle i\hat{\varphi}_0^*, \d_\gamma \OS \hat{\varphi}_0 \right\rangle_{\mathrm{os}} \\
            \left\langle \hat{\varphi}_0^*, \d_\theta \OS \hat{\varphi}_0 \right\rangle_{\mathrm{os}} & \left\langle i\hat{\varphi}_0^*, \d_\theta \OS \hat{\varphi}_0 \right\rangle_{\mathrm{os}}
        \end{pmatrix} = 2 k^6 (k^2 + \cosh^2 k)^3 (2k - \sinh(2k)),
    \end{equation*}
    which does not vanish for $k > 0$, establishing~\ref{assumption: os_transv}.
\end{proof}

Since the limit $\Rey = 0$ is non-degenerate for the nonlinear operator $\F$, we can easily obtain a similar result for sufficiently small $\Rey$ with $\tilde{\theta}$, $\tilde{\gamma}$ appropriate functions of $\Rey$. This follows from a soft argument Lemma~\ref{lemma: extra_param}, which shows that the hypotheses of Theorem~\ref{bif_lemma} are robust under certain perturbations. In particular, no asymptotic expansions are required.

\begin{proof}[Proof of Theorem~\ref{thm: lowR}]
    Examining the proof of Theorem~\ref{thm: gen_local}, we extend the function $\mathcal{G} \colon \mathcal{V} \to \Y$ defined in~\eqref{eq: perturbed_function} to $\tilde{\mathcal{G}} \colon \mathcal{V} \times \Real \to \Y$. Given $\Rey \in \Real$, $\tilde{\mathcal{G}}(0, \lambda, \Rey) = 0$ for all $\lambda \in (0, \pi) \times (0, \infty)$, i.e.~hypothesis~\ref{H1} of Theorem~\ref{bif_lemma} holds. Lemma~\ref{lemma: zeroR} verifies that the hypotheses~\ref{assumption: os_kernel} and~\ref{assumption: os_transv} of Theorem~\ref{thm: general_local} hold at parameters $(\theta_0, \gamma_0)$ given in~\eqref{eq: zeroR_disp} at $\Rey = 0$, i.e.~$\tilde{\mathcal{G}}(\, \cdot \,, \, \cdot \,; 0)$ satisfies hypotheses~\ref{H2} and~\ref{H3} of Theorem~\ref{bif_lemma}. 
    
    Therefore, by Lemma~\ref{lemma: extra_param}, there exist $\eps > 0$ and analytic functions $\tilde{\theta} \colon (-\eps, \eps) \to \Real$, $\tilde{\gamma} \colon (-\eps, \eps) \to \Real$, and $\tilde{\xi} \colon (-\eps, \eps) \to \U$ satisfying
    \begin{equation*}
        \tilde{\theta}(0) = \theta_0, \quad \tilde{\gamma}(0) = \gamma_0, \quad \tilde{\xi}(0) = \mathcal{Q} \mathscr{E} \hat{\varphi}_0,
    \end{equation*}
    for $\theta_0$ and $\gamma_0$ defined in~\eqref{eq: zeroR_disp} and $\hat{\varphi}_0$ defined in~\eqref{eq: zeroR_ker}. Furthermore, given $\Rey \in [0, \eps)$, $\tilde{\mathcal{G}}_w(0, \tilde{\theta}(\Rey),\tilde{\gamma}(\Rey), \Rey)$ has a one-dimensional kernel spanned by $\tilde{\xi}(\Rey)$ and range of codimension $2$, and for any nonzero $\mu \in \Real^2$, 
    \begin{equation*}
        D_{w \lambda}\tilde{\mathcal{G}}(0, \tilde{\theta}(\Rey), \tilde{\gamma}(\Rey), \Rey) [\tilde{\xi}(\Rey), \mu] \notin \operatorname{ran} \tilde{\mathcal{G}}_w (0, \tilde{\theta}(\Rey), \tilde{\gamma}(\Rey), \Rey).
    \end{equation*}
    Thus, the result of Theorem~\ref{thm: gen_local} holds at this $\Rey \in [0, \eps)$, completing the proof. 
\end{proof}

\section*{Acknowledgements}
The authors would like to thank Noah Stevenson for suggesting the argument in Lemma~\ref{lemma: nonres} and Jeyabal Sivaloganathan for helping us with a step in Lemma~\ref{lemma: V_H1}. Daniel Abraham was supported fully by a PhD studentship, ``Additional Funding Programme for Mathematical Sciences'', awarded by the Heilbronn Institute for Mathematical Research, grant number EP/V521917/1.

\appendix
\section{Proofs of abstract bifurcation results}\label{section: extra_abstract_theory}
Here we discuss the minor adaptations of the proof of~\cite[Theorem 9.1.1]{buffoni2003analytic} required to prove Theorem~\ref{theorem: gbt}. The key distinction is the auxiliary lemma below which is an analogue of Proposition 8.3.4 in~\cite{buffoni2003analytic}.

First we introduce the notation for the Lyapunov--Schmidt reduction for the problem
\begin{equation} \label{eq: abstract_Fzero}
    \F(u, \lambda) = 0,
\end{equation}
for an analytic map $\F \map{\X \times \Real^n}{\Y}$ between Banach spaces $\X$ and $\Y$ satisfying the assumptions of Theorem~\ref{bif_lemma}. Let $\X_1 = \operatorname{span}\{u_0\}$ and $\Y_1 = \operatorname{ran} \F_u (0, \lambda_0)$, and let $\X_0$ and $\Y_0$ be complements of $\X_1$ and $\Y_1$ in $\X$ and $\Y$ respectively, so that 
\begin{equation*}
    \X = \X_0 \oplus \X_1, \quad \Y = \Y_0 \oplus \Y_1.
\end{equation*}
Note that $\Y_0 \cong \Real^n$. Define $\F_0 := \Pi_{\Y_0} \F$, $\F_1 := \Pi_{\Y_1} \F$ as the projections of the image of $\F$ onto $\Y_0$ and $\Y_1$ respectively. 

Then,~\eqref{eq: abstract_Fzero} is equivalent to some $x_0 \in \X_0$, $x_1 \in \X_1$ and $\lambda \in \Real^n$ satisfying 
\begin{equation}\label{eq: F_decomp}
    \begin{cases}
        \F_0(x_0, x_1, \lambda) = 0, \\
        \F_1(x_0, x_1, \lambda) = 0.
    \end{cases}
\end{equation}
Now by construction, $D_{x_0} \F_1 (0, 0, \lambda_0) \colon \X_0 \to \Y_1$ is invertible, so by the implicit function theorem, for every $(x_1, \lambda) \in U_1$ in a neighbourhood $U_1 \subset \X_1 \times \Real^n$ of $(0, \lambda_0)$ there is a unique $\phi = \phi(x_1, \lambda) \in \X_0$ such that
\begin{equation}\label{eq: lyap_proj_ran}
    \F_1 (\phi(x_1, \lambda), x_1, \lambda) = 0.
\end{equation}
Define $\psi \colon U_1 \to \X$ and $f \colon \X_1 \times \Real^n \to \Y_0$ by 
\begin{equation}\label{eq: lyapunov}
    \psi(x_1, \lambda) := \phi(x_1, \lambda) + x_1, \qquad f(x_1, \lambda) := \F_0 (\psi(x_1, \lambda), \lambda).
\end{equation}
Then in $U_1$,~\eqref{eq: F_decomp} is equivalent to $f(x_1, \lambda) = 0$.

\begin{lemma}\label{lemma: pre_gbt}
    In the setting of Theorem~\ref{bif_lemma} and with the notation above, $D_{(x,\lambda_1, \ldots, \lambda_{n-1})} \F \colon \X \times \Real^{n-1} \to \Y$ has Fredholm index zero, and we have
    \begin{enumerate}[label=\rm(\alph*)]
        \item\label{pregbt_dimker} For $(x_1, \lambda)$ in a neighbourhood $U_0$ of $(0, \lambda_0)$,
        \begin{equation*}
            \dim \ker D_{(u,\lambda_1, \ldots, \lambda_{n-1})} \F (\psi(x_1, \lambda), \lambda) = \dim \ker D_{(u,\lambda_1, \ldots, \lambda_{n-1})} f (x_1, \lambda).
        \end{equation*}
        \item\label{pregbt_invertible} For $0 < |s| < \eps$, $D_{(x,\lambda_1, \ldots, \lambda_{n-1})} \F (U(s), \Lambda(s))$ is invertible if and only if $\Lambda'_n(s) \neq 0$.
    \end{enumerate}
\end{lemma}
\begin{proof}
    Let $w \in \X_0$, $v \in \X_1$ and $\mu \in \Real^{n-1}$. Now $(w+v,\mu) \in \ker D_{(u,\lambda_1, \ldots, \lambda_{n-1})} \F (\psi(x_1, \lambda), \lambda)$ iff and only if
    \begin{subequations}\label{eq: proj_diff}
        \begin{align}
            \Pi_{\Y_0} \F_u (\psi(x_1, \lambda), \lambda) (w + v) + \Pi_{\Y_0} \F_{(\lambda_1, \ldots, \lambda_{n-1})} (\psi(x_1, \lambda), \lambda) \mu = 0, \label{eq: proj_diff1}\\
            \Pi_{\Y_1}  \F_u (\psi(x_1, \lambda), \lambda) (w + v) + \Pi_{\Y_1} \F_{(\lambda_1, \ldots, \lambda_{n-1})} (\psi(x_1, \lambda), \lambda) \mu = 0. \label{eq: proj_diff2}
        \end{align}
    \end{subequations}
    Now, differentiating~\eqref{eq: lyap_proj_ran} and using~\eqref{eq: lyapunov},
    \begin{align*}
       0 &= D_{(x_1, \lambda)} \Pi_{\Y_1} \F_u (\psi(x_1, \lambda), \lambda) [v, \mu] \\
       &=\Pi_{\Y_1} \F_u (\psi(x_1, \lambda), \lambda) \left(\phi_{x_1}(x_1, \lambda) v + v + \phi_{(\lambda_1, \ldots, \lambda_{n-1})}(x_1, \lambda) \mu\right) + \Pi_{\Y_1} \F_{(\lambda_1, \ldots, \lambda_{n-1})} (\psi(x_1, \lambda), \lambda) \mu.
    \end{align*}
    Therefore,~\eqref{eq: proj_diff2} implies
    \begin{equation*}
       \Pi_{\Y_1} \F_u (\psi(x_1, \lambda), \lambda) \left(w - \phi_{x_1}(x_1, \lambda) v - \phi_{(\lambda_1, \ldots, \lambda_{n-1})}(x_1, \lambda) \mu\right) = 0,
    \end{equation*}
    noting that 
    \begin{equation*}
        w - \phi_{x_1}(x_1, \lambda) v - \phi_{(\lambda_1, \ldots, \lambda_{n-1})}(x_1, \lambda) \mu \in \X_0.
    \end{equation*}
    Now since $\Pi_{\Y_1} \F_u (0, \lambda_0) = D_{x_0} \F_1 (0, 0, \lambda_0) \colon \X_0 \to \Y_1$ is invertible, $\Pi_{\Y_1} \F_u (\psi(x_1, \lambda), \lambda) \vert_{\X_0}$ is also invertible $\X_0 \to \Y_1$ for $(\psi(x_1, \lambda), \lambda) \in U_0 \subset U_1$ where $U_0$ a sufficiently small neighbourhood of $(0,\lambda_0)$. This implies that
    \begin{equation}
       w = \phi_{x_1} (x_1, \lambda) v + \phi_{(\lambda_1, \ldots, \lambda_{n-1})}(x_1, \lambda) \mu. \label{eq: w}
    \end{equation} 
    Also,~\eqref{eq: w} implies~\eqref{eq: proj_diff2} holds for any $v \in \X_1$, $\mu \in \Real^{n-1}$.
    Thus,~\eqref{eq: proj_diff} is equivalent to
    \begin{equation*}
        \Pi_{\Y_0} \F_u (\psi(x_1, \lambda), \lambda) (\phi_{x_1} (x_1, \lambda) v + v + \phi_{(\lambda_1, \ldots, \lambda_{n-1})}(x_1, \lambda) \mu) + \Pi_{\Y_0} \F_{(\lambda_1, \ldots, \lambda_{n-1})} (\psi(x_1, \lambda), \lambda) = 0,
    \end{equation*}
    which in turn is equivalent to 
    \begin{equation*}
        f_{x_1} (x_1, \lambda) v + f_{(\lambda_1, \ldots, \lambda_{n-1})} (x_1, \lambda) \mu = 0.
    \end{equation*}
    This establishes~\ref{pregbt_dimker}. 
    
    Now by dimension counting, $D_{(x_1, \lambda_1, \ldots, \lambda_{n-1})} f$ has Fredholm index zero. For a point $(U(s), \Lambda(s)) \in \mathscr{C}_{\mathrm{loc}}$ with $s \in (-\eps, \eps)$, one can check that $U(s) = \psi(su_0, \Lambda(s))$, so $\F(U(s), \Lambda(s)) = 0$ is equivalent to 
    \begin{equation*}
        f(su_0, \Lambda(s)) = 0.
    \end{equation*}
    By differentiation with respect to $s$, we observe that
    \begin{align} \label{eq: ls_chain_rule}
        0 = f_{x_1}(s u_0, \Lambda(s))u_0 + f_{\lambda}(s u_0, \Lambda(s)) \Lambda'(s). 
    \end{align}
    By the hypothesis~\ref{H3} in Theorem~\ref{bif_lemma}, for $s \in (-\eps, \eps)$ with $\eps > 0$ sufficiently small, 
    \begin{equation*}
        f_{\lambda}(s u_0, \Lambda(s)) \colon \Real^n \to \Y_0
    \end{equation*}
    is invertible.
    In particular, the derivatives $f_{\lambda_i}(s u_0, \Lambda(s))$ are linearly independent and span $\Y_0$.
    
    Now suppose $\Lambda'_n(s) = 0$. Then~\eqref{eq: ls_chain_rule} implies
    \begin{equation*}
        \left(u_0, \Lambda'_1(s), \ldots \Lambda'_{n-1}(s)\right) \in \ker D_{(x_1, \lambda_1, \ldots, \lambda_{n-1})} f(s u_0, \Lambda(s)).
    \end{equation*}
    If on the other hand, $\Lambda'_n(s) \neq 0$, suppose that $v = c u_0 \in \X_1$ and $\mu \in \Real^{n-1}$ are such that $(v, \mu) \in \ker D_{(x_1, \lambda_1, \ldots, \lambda_{n-1})} f(s u_0, \Lambda(s))$. Then, by~\eqref{eq: ls_chain_rule},
    \begin{equation*}
        \sum_{i=1}^{n-1} f_{\lambda_i}(s u_0, \Lambda(s))(c \Lambda'_i(s) - \mu_i) + f_{\lambda_n}(s u_0, \Lambda(s)) c \Lambda'_n(s) = 0,
    \end{equation*}
    and since the $f_{\lambda_i}(s u_0, \Lambda(s))$ are linearly independent in $\Y_0$, we must have $c\Lambda'_n (s) = 0$. But then $\Lambda'_n(s) \neq 0$ implies $c=0$, which in turn implies 
    \begin{equation*}
        f_{(\lambda_1, \ldots, \lambda_{n-1})}(s u_0, \Lambda(s))\mu = 0,
    \end{equation*}
    and so by the linear independence of these vectors, $\mu = 0$. Thus, $D_{(x, \lambda_1, \ldots, \lambda_n)} \F (U(s), \Lambda(s))$ has trivial kernel and is therefore invertible, establishing~\ref{pregbt_invertible}.
\end{proof}
\begin{proof}[Proof of Theorem~\ref{theorem: gbt}]
    The proof is almost identical to that of~\cite[Theorem 9.1.1]{buffoni2003analytic}, except for extending the space $X = \X \times \Real^{n-1}$ and redefining the set of non-singular solutions of $\F(u, \lambda) = 0$ to be
    \begin{equation*}
        \mathscr{R} = \big\{(u, \lambda) \in S \colon \ker \d_{(u,\lambda_1, \ldots, \lambda_{n-1})} \F (u, \lambda) = \{0\} \big\}.
    \end{equation*} 
    Now $\Lambda'_n$ is nonzero on $(-\eps, 0) \cup (0, \eps)$ for $\eps > 0$ sufficiently small, so by Lemma~\ref{lemma: pre_gbt}, we can pick $\eps > 0$ such that $\big\{ (U(s), \Lambda(s)) : s \in (0, \eps) \big\} \subset \mathscr{R}$. From then, the proof is the same as in~\cite{buffoni2003analytic}.
\end{proof}

The following result shows that the hypotheses of Theorem~\ref{bif_lemma} are robust under perturbation with respect to an additional parameter.
\begin{lemma}\label{lemma: extra_param}
    Let $\F \colon \U \times \Real \to \Y$ be analytic, with $\U \subset \X \times \Real^n$ open, $n \in \mathbb{N}$ and $\X$ and $\Y$ Banach spaces. Suppose the following hold:
    \begin{enumerate}[label=\rm(\roman*)]
        \item For all $\mu$, $\F(\, \cdot \,, \, \cdot \,; \mu)$ satisfies hypothesis~\ref{H1} of Theorem~\ref{bif_lemma};
        \item\label{extra_param0_transv} $\F(\, \cdot \,, \, \cdot \,; 0)$ satisfies hypothesis~\ref{H2} and~\ref{H3} of Theorem~\ref{bif_lemma}.
    \end{enumerate}
    Then there exist $\eps > 0$ and analytic functions $\tilde{\lambda} \colon (-\eps, \eps) \to \Real^n$, $\tilde{\xi} \colon (-\eps, \eps) \to \U$ such that $\tilde{\lambda}(0) = \lambda_0$, $\tilde{\xi}(0) = \xi_0$, and for $\lvert \mu \rvert < \eps$,
    \begin{enumerate}[label=\rm(\alph*)]
        \item The kernel of $\F_u (0, \tilde{\lambda}(\mu); \mu)$ is spanned by $\tilde{\xi}(\mu)$, and the range of $\F_u (0, \tilde{\lambda}(\mu); \mu)$ has codimension $n$.\label{extra_param_ker}
        \item For any nonzero $\dot{\lambda} \in \Real^n$, 
        \begin{equation*}
            D_{u \lambda}\F(0, \tilde{\lambda}(\mu); \mu) [\tilde{\xi}(\mu), \dot{\lambda}] \notin \operatorname{ran} \F_u(0, \tilde{\lambda}(\mu); \mu).
        \end{equation*}\label{extra_param_transv}
    \end{enumerate}
\end{lemma}

\begin{proof}
    There exists some $l \in \X^*$, where $\X^*$ is the dual space of $\X$, such that $l \xi_0 = 1$. Define $\mathscr{H} \colon \X \times \Real^n \times \Real \to \Y \times \Real$ by 
    \begin{equation*}
        \mathscr{H}(\xi, \lambda, \mu) := \begin{pmatrix}
            \F_u (0, \lambda; \mu) \xi \\ l \xi - 1
        \end{pmatrix}.
    \end{equation*}
    Then $\mathscr{H}_{(\xi, \lambda)}(\xi_0, \lambda_0, 0) \colon \X \times \Real^n \to \Y \times \Real$ is a Fredholm operator of zero index. Consider $(\dot{\xi}, \dot{\lambda}) \in \X \times \Real^n$ in its kernel, i.e.~suppose
    \begin{equation*}
        \mathscr{H}_{(\xi, \lambda)}(\xi_0, \lambda_0, 0) \begin{pmatrix}
        \dot{\xi} \\ \dot{\lambda}
        \end{pmatrix} = \begin{pmatrix}
            \F_u (0, \lambda_0; 0) \dot{\xi}  + D_{u \lambda} \F (0, \lambda_0; 0) [\xi_0, \dot{\lambda}]\\ l \dot{\xi}
        \end{pmatrix} = \begin{pmatrix}
            0 \\ 0
        \end{pmatrix}.
    \end{equation*}
    By~\ref{extra_param0_transv}, $\dot{\lambda} = 0$, and so $\dot{\xi} = c \xi_0$ for some $c \in \Real$. But $0 = l \dot{\xi} = c l \xi_0 = c$, so $\mathscr{H}_{(\xi, \lambda)}(\xi_0, \lambda_0, 0)$ has trivial kernel and is therefore invertible. By the implicit function theorem, then, we obtain functions $\lambda(\mu)$ and $\xi(\mu)$ for $\lvert \mu \rvert < \eps$ such that $\xi(\mu) \in \ker \F_u(0, \lambda(\mu); \mu)$, and by continuity, $\dim \ker \F_u(0, \lambda(\mu); \mu) \leq 1$ for $\mu$ sufficiently small, so $\xi(\mu)$ spans $\ker \F_u(0, \lambda(\mu); \mu)$. Since its Fredholm index is $1-n$, the codimension of the range is $n$, so~\ref{extra_param_ker} holds. Finally, the transversality condition~\ref{extra_param0_transv} is an open condition, so~\ref{extra_param_transv} holds.
\end{proof}

\section{Elliptic systems theory}\label{section: ADN_theory}
Below, we briefly review some results for elliptic systems of Agmon--Douglis--Nirenberg type~\cite{agmon1964estimates}; for more details see~\cite{wloka1995boundary, volpert2011elliptic}. Let $\Omega \subset \Real^n$ be a bounded domain whose regularity will be specified later, and consider the linear differential system
\begin{equation}\label{eq: interior}
    \mathcal{A}(x, \d) u(x) = f(x), \qquad \mathcal{A}_{ij}(x, \d) = \sum_{\lvert \rho \rvert \leq \alpha_{ij}} a_{ij, \rho}(x) \d^\rho
\end{equation}
for $x \in \Omega$, with continuous coefficients $a_{ij, \rho}$ whose precise regularity will be specified later. Let $\alpha_{ij}$ denote the order of $\mathcal{A}_{ij}$, setting $\alpha_{ij} = - \infty$ if $\mathcal{A}_{ij} \equiv 0$.

We introduce some further notation to define proper ellipticity. Define the \emph{total order} $m$ for the operator $\mathscr{A}$ by
\begin{equation*}
    m :=\max_{q \in S_p} \sum_{j=1}^{p} \alpha_{j q(j)}.
\end{equation*}
Fix $x \in \overline{\Omega}$, and let $\mathscr{A}(x,\xi) = {[\mathscr{A}_{ij}]}_{p \times p}$ represent the Fourier transform of $\mathcal{A}$, formally replacing derivatives $\d_i$ with $i \xi_i$. Denoting by $\pi_q$ projection onto homogeneous polynomials of order $q$, we define the \emph{characteristic polynomial} of $\mathscr{A}(x, \dot)$ as
\begin{equation*}
    \chi(\xi) := \pi_m \det \left(\mathscr{A}(x,\xi)\right),
\end{equation*}
i.e.~the sum of the terms in the determinant of $\mathscr{A}$ of order $m$. 

For any matrix ${\left(\alpha_{ij}\right)}_{p \times p}$ with elements in $\mathbb{Z} \cup \{-\infty\}$, there exist so-called \emph{Douglis--Nirenberg numbers} $s_1, \ldots, s_p, ~ t_1, \ldots, t_p$, which are integers satisfying
\begin{equation}\label{eq: DN_numbers}
    \sum_{i=1}^{p}s_i+t_i = m, \qquad s_i + t_j \geqslant \alpha_{ij} \quad \text{for all } i,j \in \{1, \ldots, p\}.
\end{equation}
Note that these integers are not unique: for any constant integer $c$, if $s_i,~t_j$ are Douglis--Nirenberg numbers, then so are $s_i + c, ~ t_j - c$. We therefore without loss of generality assume $s_i \leq 0$ and $t_j \geq 0$. Then we define the associated \emph{principal part} $\pi \mathscr{A}$ of $\mathscr{A}$ by
\begin{equation*}
    \pi \mathscr{A}(x, \xi) := {\left(\pi_{s_i + t_j} \mathscr{A}_{ij}(x,\xi)\right)}_{p \times p},
\end{equation*}
i.e.~the matrix whose coefficients are each the sum of the terms of order $s_i + t_j$ in $\mathscr{A}_{ij}$.

Now we say that $\mathcal{A}$ is \emph{elliptic in the sense of Douglis--Nirenberg} at $x \in \overline{\Omega}$ if
\begin{equation*}
    \chi(\xi) \neq 0 \quad \text{for all } \xi \in \Real^n \setminus \{0\},
\end{equation*}
and this ellipticity is \emph{uniform} on $\overline{\Omega}$ if there exists $\Lambda > 0$ such that for all $x \in \overline{\Omega}$ and $\xi \in \Real^n$ with $\lvert \xi \rvert = 1$, 
\begin{equation}\label{eq: unif_elliptic}
    \dfrac{1}{\Lambda} \leq  \lvert \chi(\xi) \rvert \leq \Lambda.
\end{equation}

Let $\mathcal{A}$ be uniformly elliptic and consider $z \in \d \Omega$, with outward-pointing unit normal $e_n$ without loss of generality. We call $\mathcal{A}$ \emph{proper} at $z \in \d \Omega$ if for all nonzero $\xi' \in \Real^{n-1}$, the polynomial
\begin{equation*}
    \chi\left((\xi', \lambda)\right) = \det \left(\pi \mathscr{A} \left(z, (\xi', \lambda)\right)\right)
\end{equation*}
has as many roots $r$ in the upper half-plane $\{\operatorname{Im} \lambda > 0 \}$ as in the lower half-plane. For such an operator, we may factorise the polynomial 
\begin{equation*}
    \det \left(\pi \mathscr{A} \left(z, (\xi', \lambda)\right)\right) = \rho_+ (\lambda) \rho_- (\lambda),
\end{equation*}
such that $\rho_+$ contains all the roots of $\det \pi \mathscr{A}(\lambda)$ inside the positive upper half-plane $\{\operatorname{Im} \lambda > 0\}$. We say that $\mathcal{A}$ is \emph{properly elliptic on $\overline{\Omega}$} if it is uniformly elliptic on $\overline{\Omega}$ and is proper on $\d \Omega$.

Given such a properly elliptic operator $\mathcal{A}$, suppose the following boundary conditions hold for all $z \in \Gamma$, where $\Gamma \subset \d \Omega$ is a portion of the boundary:
\begin{equation}\label{eq: boundary}
    \mathcal{B}(z, \d) u(z) = g(z), \quad \mathcal{B}_{kj}(z, \d) = \sum_{\lvert \sigma \rvert \leq \beta_{kj}} b_{kj, \sigma}(z) \d^\sigma,
\end{equation}
for some $g \colon \Gamma \to \Real^r$. Here, the coefficients $b_{kj, \rho}$ are again continuous with regularity yet to be specified, and $\beta_{kj}$ denotes the order of $\mathcal{B}_{kj}$. For our purposes it suffices to consider the special case where $\Gamma$ is a hyperplane with outward-pointing normal vector $e_n$.

Let 
\begin{equation*}
    m_k := \max_j (\beta_{kj} - t_j), \qquad k = 1, \ldots, r,
\end{equation*}
so that $\beta_{kj} \leq m_k + t_j$. 

Fixing $z \in \Gamma$, denote by $\mathscr{B}(z,\xi) = {[\mathscr{B}_{kj}]}_{r \times p}$ the Fourier transform of $\mathcal{B}(z, \d)$. Then, the \emph{principal part} of $\mathscr{B}$ is
\begin{equation*}
    \pi \mathscr{B}(z, \xi) := {\left(\pi_{m_k + t_j} \mathscr{B}_{kj}(z,\xi)\right)}_{r \times p}.
\end{equation*}
Finally denote by $\adj M = \det M {(M)}^{-1}$ the adjugate of a matrix $M$. We say that $(\mathcal{A},\mathcal{B})$ satisfies the \emph{Shapiro--Lopatinsky condition} at $z \in \Gamma$ if, for all $\xi' \in \Real^{n-1}$ with $\lvert \xi' \rvert = 1$, the rows of the matrix $\pi \mathscr{B} \adj(\pi \mathscr{A}) (z, (\xi', \lambda))$ are linearly independent modulo $\rho_+(\lambda)$. This is equivalent to a related larger matrix with entries in $\Complex$ having full rank, which can be quantified by looking at minor determinants. We say that the Shapiro--Lopatinsky condition holds on $\Gamma$ \emph{uniformly} if the sum of the absolute values of all such minor determinants is bounded below by some constant $\lambda>0$, independent of $z \in \Gamma$ and $\xi'$. Note that this condition is also referred to as the `covering' or `complementing' condition.

With these definitions, we state the following a priori estimates for elliptic boundary value problems satisfying the Shapiro--Lopatinsky condition, firstly for problems in H\"older spaces and then for problems in Sobolev spaces.

\begin{theorem}[\cite{agmon1964estimates}, Theorem 9.3]\label{adn_schauder}
    Let $(\mathcal{A},\mathcal{B})$ be an operator system satisfying~\eqref{eq: interior} in $U \subset \Omega$ and~\eqref{eq: boundary} on $\d \U \cap \d \Omega \subset \Gamma$, and fix $\alpha \in (0,1)$ and an integer $l \geq l_0 := \max(0, m_k)$. Suppose $\Omega$ is of class $C^{l + \max(-s_i, t_j, -m_k), \alpha}$ and that $f_i, a_{ij, \rho} \in C^{l-s_i, \alpha}(\overline{\Omega})$ and $g_k, b_{kj, \sigma} \in C^{l-m_k, \alpha}(\Gamma)$, with
    \begin{equation*}
        \lVert a_{ij, \rho} \rVert_{C^{l-s_i, \alpha}(\Omega)} \leq K, \quad \lVert b_{kj, \rho} \rVert_{C^{l-m_k, \alpha}(\Omega)} \leq K
    \end{equation*}
    for some constant $K > 0$. Suppose further that $\mathcal{A}$ is properly elliptic on $\overline{\Omega}$, satisfying~\eqref{eq: unif_elliptic} with ellipticity constant $\Lambda > 0$, and that $(\mathcal{A},\mathcal{B})$ satisfies the Shapiro--Lopatinsky condition uniformly on $\Gamma$ with constant $\lambda > 0$. Let $U \subset \Omega$ be such that $ \d U \cap \d \Omega \subset \Gamma$. If $u_j \in C^{l_0 + t_j, \alpha}(U)$, then $u_j \in C^{l + t_j, \alpha}(U)$ and
    \begin{equation*}
        \lVert u_j \rVert_{C^{l+t_j, \alpha}(U)} \leq C \left(\sum_{i=1}^{p} \lVert f_i \rVert_{C^{l-s_i, \alpha}(\Omega)} + \sum_{k=1}^{r} \lVert g_k \rVert_{C^{l-m_k, \alpha}\left(\Gamma\right)} + \sum_{j=1}^{p} \lVert u_j \rVert_{C^0(\Omega)} \right),
    \end{equation*}
    for some constant $C$ depending on $n, p, K, \Lambda, \lambda, l, \alpha, U, \Omega$.
\end{theorem}
\begin{theorem}[\cite{agmon1964estimates}, Theorem 10.6]\label{adn_sobolev}
    Let $(\mathcal{A},\mathcal{B})$ be an operator system satisfying~\eqref{eq: interior} in $U \subset \Omega$ and~\eqref{eq: boundary} on $\d \U \cap \d \Omega \subset \Gamma$, and fix $q \geq 1$ and an integer $l \geq l_1 := \max(0, m_k + 1)$. Suppose the domain $\Omega$ is of class $C^{l + \max(-s_i, t_j, -m_k)}$ and $a_{ij, \rho} \in C^{l-s_i}(\overline{\Omega})$ and $b_{kj, \sigma} \in C^{l-m_k}(\Gamma)$, with 
    \begin{equation*}
        \lVert a_{ij, \rho} \rVert_{C^{l-s_i}(\Omega)} \leq K, \quad \lVert b_{kj, \rho} \rVert_{C^{l-m_k}(\Omega)} \leq K
    \end{equation*}
    for some $K > 0$. Suppose that $\mathcal{A}$ is properly elliptic on $\overline{\Omega}$, satisfying~\eqref{eq: unif_elliptic} with ellipticity constant $\Lambda > 0$, and that $(\mathcal{A},\mathcal{B})$ satisfies the Shapiro--Lopatinsky condition uniformly on $\Gamma$ with constant $\lambda > 0$. 
    Then,
    \begin{equation*}
        \lVert u_j \rVert_{W^{l+t_j, q}(U)} \leq C \left(\sum_{i=1}^{p} \lVert f_i \rVert_{W^{l-s_i, q}(\Omega)} + \sum_{k=1}^{r} \lVert g_k \rVert_{W^{l-m_k-1/q, q}\left(\Gamma\right)} + \sum_{j=1}^{p} \lVert u_j \rVert_{L^1(\Omega)} \right),
    \end{equation*}
    for some constant $C$ depending on $n, p, K, \Lambda, \lambda, l, q, U, \Omega$ and the modulus of continuity in the leading coefficients of $\mathcal{A}_{ij}$.
\end{theorem}
\begin{remark}
    The `remainder' terms $\lVert u_j \rVert_{C^{0}(\Omega)}$ in Theorem~\ref{adn_schauder} may be replaced with $\lVert u_j \rVert_{L^1(\Omega)}$ or a higher $L^q$ norm, using Theorem~\ref{adn_sobolev}.
\end{remark}
\label{section: appendix}

\printbibliography

\end{document}